\documentclass[11pt,reqno]{amsart}

\usepackage{amsmath,amssymb,amsfonts}

\usepackage[left=30mm,right=30mm,top=30mm,bottom=30mm]{geometry}
\usepackage{wasysym}
\usepackage{mathrsfs}
\usepackage{hyperref}
\usepackage{cite}
\usepackage{amsaddr}
\usepackage[sc]{mathpazo} 
\usepackage{graphics}
\usepackage{graphicx}
\usepackage{color}
\usepackage{empheq} 
\usepackage{cases}
\usepackage{enumitem}
\usepackage{orcidlink}
\usepackage{comment}

\newcommand{\Z}{\mathbb{Z}}
\newcommand{\N}{\mathbb{N}}
\newcommand{\R}{\mathbb{R}}
\newcommand{\C}{\mathbb{C}}
\renewcommand{\L}{\LL^2}
\newcommand{\LL}{\mathsf{L}}
\renewcommand{\H}{\mathsf{H}}
\newcommand{\HS}{\mathscr{H}}

\newcommand{\W}{\mathsf{W}}

\renewcommand{\a}{\mathfrak{a}}
\newcommand{\A}{{\mathscr{A}}}

\newcommand{\J}{\mathscr{J}}

\newcommand{\B}{\mathscr{B}}

\newcommand{\D}{\mathbf{D}}

\renewcommand{\rho}{\varrho}

\DeclareMathOperator{\dom}{dom}
\DeclareMathOperator{\supp}{supp}
\DeclareMathOperator{\dist}{dist}
\DeclareMathOperator{\diam}{diam}
\DeclareMathOperator{\capty}{\mathscr{C}}

\DeclareMathOperator{\conv}{conv}

\newcommand{\la}{\langle}
\newcommand{\ra}{\rangle}

\newcommand{\eps}{\varepsilon}
\newcommand{\e}{_{\varepsilon}}
\newcommand{\ke}{_{k,\varepsilon}}
\renewcommand{\oe}{_{0,\varepsilon}}

\newcommand{\al}{\alpha}

\newcommand{\ga}{\gamma}
\newcommand{\GG}{\mathcal{G}}

\renewcommand{\d}{\,\mathrm{d}}
\newcommand{\ds}{\displaystyle}

\newcommand{\Id}{\mathrm{Id}}

\newcommand{\cupl}{\bigcup\limits}
\newcommand{\suml}{\sum\limits}

\newcommand{\wt}{\widetilde}
\newcommand{\wh}{\widehat}

\newcommand{\ceq}{\coloneqq}

\newcommand{\restr}{\restriction}

\theoremstyle{plain}
\newtheorem{theorem}{Theorem}[section]
\newtheorem*{theorem*}{Theorem}
\newtheorem{lemma}[theorem]{Lemma}
\newtheorem*{lemma*}{Lemma}
\newtheorem{proposition}[theorem]{Proposition}
\newtheorem{corollary}[theorem]{Corollary}
\theoremstyle{remark}
\newtheorem{remark}[theorem]{Remark}
\newtheorem*{remark*}{Remark} 
\newtheorem{example}[theorem]{Example}
\newtheorem*{example*}{Example} 
\theoremstyle{definition}

\numberwithin{equation}{section}
\numberwithin{figure}{section}

\title
[The Neumann  sieve problem revisited]
{The  Neumann  sieve problem revisited}

\author[Andrii Khrabustovskyi]{Andrii Khrabustovskyi \orcidlink{0000-0001-6298-9684}}
\address{Department of Physics, Faculty of Science, University of
	Hradec Kr\'{a}lov\'{e},\\ Rokitansk\'eho 62, 500 03 Hradec Kr\'{a}lov\'{e}, Czech Republic} 

\email{andrii.khrabustovskyi@uhk.cz}

\begin{document}
	
	\begin{abstract}
	Let $\Omega\subset\mathbb{R}^n$ be a domain, $\Gamma$ be a hyperplane intersecting it. Let $\varepsilon>0$, and $\Omega_\varepsilon=\Omega\setminus\overline{\Sigma_\varepsilon}$, where $\Sigma_\varepsilon$ (``sieve'') is an $\varepsilon$-neighbourhood of $\Gamma$ punctured by many narrow passages. When $\varepsilon\to0$, the number of passages tends to infinity, while the diameters of their cross-sections tend to zero. For the case of identical straight periodically distributed and appropriately scaled passages T.~Del Vecchio (1987) proved that the Neumann Laplacian on $\Omega_\varepsilon$ converges in a strong resolvent sense to the Laplacian on $\Omega\setminus\Gamma$ subject to the so-called  $\delta'$-conditions on $\Gamma$. We will refine this result by deriving estimates on the rate of convergence in terms of various operator norms, and providing the estimate for the distance between the spectra. The assumptions we impose on the passages are rather general. For $n=2$ the results of T.~Del Vecchio are not complete, some cases remain as open problems, and in this work we will fill these gaps.
	\end{abstract}

	\keywords{homogenization; perforated domain; Neumann sieve; resolvent convergence; operator estimates; spectrum; varying Hilbert spaces}
	
	\subjclass{35B27, 35B40, 35P05, 47A55}
	
	\maketitle	 
	\tableofcontents
	
	\section{Introduction}
	\label{sec:1}
	\thispagestyle{empty}
	
	The name \emph{Neumann sieve problem} usually refers to a class of homogenization problems involving partial differential equations in domains perforated along a co-dimension $1$ manifold, moreover, on  the boundary resulting from this perforation one poses Neumann boundary  conditions. There are two main types of such perforations -- by means of small pairwise disjoint inclusions, and by means of a thin nonpentrable layer with a lot of narrow channels. In the present paper, we address exactly  the latter geometrical configuration.
	
	\subsection{State of art}
	
	Let $\Omega$ be a domain in $\R^n$ with $n\ge 2$, and 
	$\Gamma$ be a hyperplane intersecting $\Omega$.
	Let $\eps>0$ be a small parameter, and the domain $\Omega\e\subset\R^n$
	be of the form
	$\Omega\e=\Omega\setminus \overline{\Sigma\e}$
	with the removed set $\Sigma\e$ having the form of a ``sieve'' of the thickness $2\eps$.
	Namely, one has
	\begin{gather*}
		\Sigma\e=O\e\setminus\overline{ \cup_{k }T\ke},
	\end{gather*}
	where $O\e$ is the $\eps$-neighbourhood of  $\Gamma$ and
	$(T\ke )_k$ is a family of ``passages'' in $ O\e$, which connect 
	the subsets of $\Omega$ lying above and below $O\e$;
	we denote these subsets by $\Omega\e^+$ and $\Omega\e^-$, see Figure~\ref{fig:Omegae}.
	Equivalently, we have
	$$
	\Omega\e={\rm int}(\overline{\Omega\e^+\cup\Omega\e^-\cup\left(\cup_{k}T\ke\right)})
	$$
	(hereinafter $ {\rm int}(\cdot) $ stands for the interior of a set).

	\begin{figure}[h]
		\centering
		\begin{picture}(350,200)
			\includegraphics[width=0.75\textwidth]{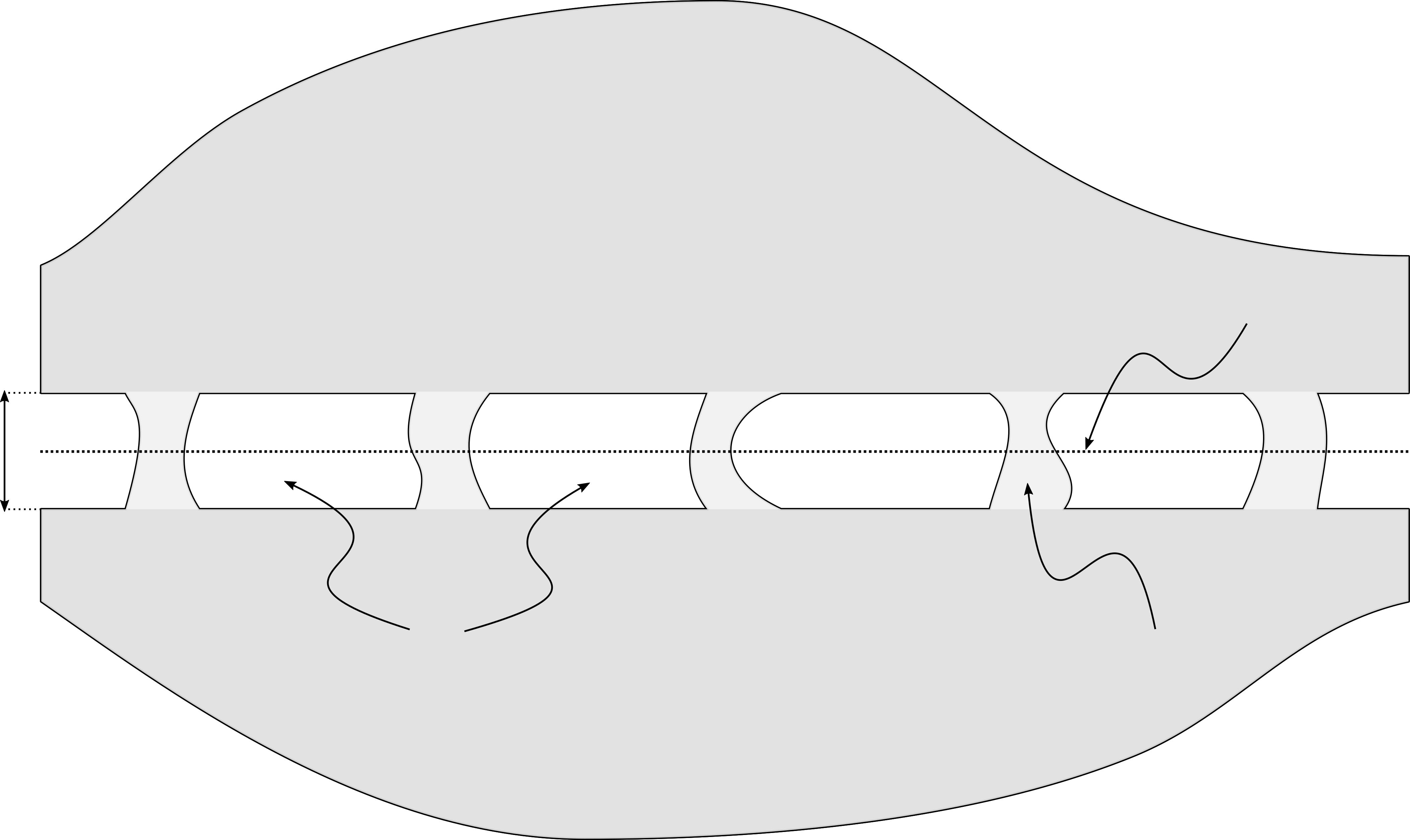}
			\put(-63,40){$T\ke$} 
			\put(-38,120){$\Gamma$} 
			\put(-185,20){$\Omega\e^-$}
			\put(-185,150){$\Omega\e^+$}
			\put(-234,45){${\Sigma\e}$}
			\put(-341,90){$_{2\eps}$}
		\end{picture}
		\caption{The domain $\Omega\e={\rm int}(\overline{\Omega\e^+\cup\Omega\e^-\cup\left(\cup_{k}T\ke\right)})$}
		\label{fig:Omegae}
	\end{figure}
	
	We consider the following boundary value problem in the domain $\Omega\e$:
	\begin{align}\label{BVP:main:1}
		-\Delta u\e + u\e = f&\text{ in }\Omega\e,\\\label{BVP:main:2}
		{\partial u\e\over \partial\nu}=0&\text{ on }\partial \Omega\e.
	\end{align} 
	Here $f\in\L(\Omega)$ is a given function, $\nu$ is the unit outward pointing normal to $\partial\Omega\e$. In  terms of operator theory the problem \eqref{BVP:main:1}--\eqref{BVP:main:2} reads
	$$
	\A\e u\e + u\e = \J\e f,
	$$
	where $\A\e$ is the Neumann Laplacian in $\L(\Omega\e)$, and $\J\e:\L(\Omega)\to\L(\Omega\e)$
	is the operator of restriction to $\Omega\e$, i.e. $\J\e f = f\restr_{\Omega\e}$.
	
	Homogenization theory is aimed to describe the asymptotic behavior of the solution
	$u\e$ to the problem \eqref{BVP:main:1}--\eqref{BVP:main:2} as $\eps\to 0$, when the number of passages (per finite volume) tends to infinity, while the  diameters of their cross-sections tends to zero.
	Apparently, for the first time this problem was addressed by T.~Del Vecchio in \cite{DelV87}. 
	In this paper the passages are identical straight cylinders, i.e.
	$$
	T\ke\cong d\e \D\times (-\eps,\eps),\ \D\text{ is a bounded domain in }\R^{n-1},\ 
	d\e\to 0\text{ as }{\eps\to 0},$$ 
	which are distributed $\rho\e$-periodically with $\rho\e\to 0$ as $\eps\to 0$;
	note that  we use other notations comparing to \cite{DelV87}.
	It was proven in \cite{DelV87} that    if   either
	\begin{gather}\label{pq:intro} \quad p\ceq \lim_{\eps\to 0}{(\GG(d\e))^{-1}\rho\e^{1-n}}<\infty
		\quad\text{or}\quad
		q\ceq \lim_{\eps\to 0}{d\e^{n-1}\eps^{-1}\rho\e^{1-n}}<\infty,
	\end{gather}
	with the function $\GG(t)$ being given by
	\begin{gather}\label{Gt}
		\GG(t)\ceq
		\begin{cases}
			t^{2-n},&n\ge 3,\\
			-\ln t,&n=2,
		\end{cases}
	\end{gather}
	then 
	\begin{gather}\label{L2strong}
		\|u\e-u\|_{\L(\Omega\e)}\to 0\text{ as }\eps\to 0,
	\end{gather}
	where the  function $u\in\H^1(\Omega\setminus\Gamma)$ is the  solution to the following (homogenized) problem:   
	\begin{align}\label{BVP:hom:1}
		-\Delta u^++ u^+=f^+&\text{\quad on }\Omega^+,\\\label{BVP:hom:2}
		-\Delta u^-+ u^-=f^-&\text{\quad on }\Omega^-,
		\\
		\label{BVP:hom:3}
		{\partial u \over\partial  \nu}=0&\text{\quad on }\partial\Omega,
		\\
		\label{BVP:hom:4}
		- {\partial u^+\over\partial  {\nu^+}}   
		=
		{\partial u^-\over\partial  {\nu^-}}
		=
		\mu(u^+-u^-)
		&
		\text{\quad on }\Gamma\cap\Omega.
	\end{align} 
	Here $\Omega^+$ and $\Omega^-$ are the subsets of $\Omega$ lying, respectively, 
	above and below $\Gamma$, $\mu\ge 0$ is a constant,   $u^\pm=u\restriction_{\Omega^\pm}$, $f^\pm=f\restriction_{\Omega^\pm}$, 
	${ \nu^\pm }$ are the   outward (with respect to $\Omega^\pm$) pointing unit normals to $\Gamma$, 
	${ \nu}$ is the unit normal to $\partial\Omega$. The constant $\mu$ is zero if either $p=0$ or $q=0$,
	otherwise $\mu>0$.
	In  terms of operators the problem \eqref{BVP:hom:1}--\eqref{BVP:hom:4} reads
	$$
	\A u + u=  f,
	$$
	where $\A$ is the Laplace operator 
	in $\L(\Omega )$ subject to the interface conditions \eqref{BVP:hom:4} on $\Gamma\cap\Omega$
	and the boundary conditions \eqref{BVP:hom:3} on $\partial\Omega$. In Section~\ref{sec:2} we
	introduce this operator more accurately by using the associated sesquilinear form.
	Note that in the  literature (see, e.g., \cite{BEL14,BLL13}) the operator $\A$ is sometimes  called 
	\emph{Schr\"{o}dinger with $\delta'$-interaction of the strength $\mu^{-1}$}.

		The choice of the boundary conditions on the subset
		$\partial\Omega\setminus \overline{O\e}$ of $\partial\Omega\e$  
		is inessential: the above result remains valid if we imposes Dirichlet (in fact, these conditions were assumed in \cite{DelV87}), Robin, mixed or any other  {$\eps$-independent}   conditions on $\partial\Omega\setminus \overline{O\e}$; of course, the boundary condition \eqref{BVP:hom:3}
		must be   changed accordingly. 
		
		The convergence \eqref{L2strong} can be equivalently re-written as follows,
		\begin{gather} 
			\label{L2strong+}
			\forall f\in\L(\Omega):\quad	\|(\A\e+\Id)^{-1} \J\e f- \J\e(\A+\Id)^{-1}f\|_{  \L(\Omega\e)}\to 0\text{ as }\eps\to0,
		\end{gather}
		where $\Id$ stands the identity operator (either in $\L(\Omega)$ or $\L(\Omega\e)$).
		Thus, $\A\e$ converges to  $\A$
		as $\eps\to 0$ in (a kind of\,\footnote{Strictly speaking, since the operators $\A\e$ and $\A$ act in different Hilbert spaces $
			\L(\Omega\e)$ and $\L(\Omega)$, one can not use the standard notion of strong resolvent convergence here. Instead, we have its  natural modification \eqref{L2strong+} involving the identification operator $\J\e$.}) \emph{strong resolvent sense}.
		\smallskip

 		Besides \cite{DelV87}, the above homogenization problem was  investigated in \cite{Ya05,Sh01}, where the estimates on the rate of convergence were obtained for the case $p=\infty\, \wedge\, q>0$ under the assumption that the solution to \eqref{BVP:hom:1}--\eqref{BVP:hom:4} is $C^2$-smooth in $\overline{\Omega^+}$ and in $\overline{\Omega^-}$. 
 		In \cite{ABCP16} the authors study the Neumann-Steklov problem (the Steklov conditions are imposed on the lateral boundary of the passages, while the Neumann conditions are prescribed on the top and bottom faces of the sieve).
 		We also mention  recent contributions \cite{BGN22,GN21}  concerning reaction-diffusion equations in domains with the above geometry, and with  diffusion coefficients being dependent on $\eps$ within passages. 
 
 		Note that in \cite{DelV87} the author used the name ``The Thick Neumann Sieve Problem''. Apparently, the word 	``thick'' was used to distinguish this problem from the one concerning 
		the sieve with  {zero thickness}, i.e. when $\Omega=\Omega\setminus\overline{\Sigma\e}$ with  $\Sigma\e$ being a subset of the hyperplane $\Gamma$
		of the form $\Sigma\e=\Gamma\setminus (\overline{\cup_k D\ke})$, where  
		$(D\ke\subset\Gamma)_k$ is a family of  
		relatively open in $\Gamma$ connected  sets. 
		Such a problem was treated for the
		first time in  
		\cite{MS66} and then revisited in
		\cite{At84,AP87,Mu85,Pi87,Da85,Ono06,CDGO08,Kh23,Sa82}. For  zero-thickness sieves 
		the result remains qualitatively the same: under suitable assumptions on the holes $D\ke$  
		$u\e$ converges to the solution of  \eqref{BVP:hom:1}--\eqref{BVP:hom:4}-type problem. 
		Note that in \cite{Ono06,CDGO08} also thick sieves with periodically distributed non-straight 
		(but satisfying certain symmetry assumptions) passages were considered.
		The Steklov problem in a domain with zero-width sieve was studied in \cite{OV05}.
		Finally, we mention the articles \cite{An04,ABZ07}, where non-linear problems in domains with zero-width sieve were treated.
		
	Neumann problems in domains with sieves being formed by an array of small obstacles were investigated, e.g., in \cite{DHS17,Sch20,SDS18,Ma08,RL10,LOPS97}. Typically, for such sieves the effect of the perforation gets lost at leading order, and thus the primary question is to find the next terms in the asymptotic expansion of the solution.

	\subsection{Our goals}
	
	In this work we restrict ourselves to the case of passages satisfying the assumption \eqref{assump:3} below. In the case of periodically distributed straight passages
	this assumption corresponds to $p<\infty$ (see~\eqref{pq:intro}. 
	The remaining case $(p=\infty\ \wedge \ q<\infty)$ will be treated elsewhere.

	We pursue three goals. The first one is to upgrade \eqref{L2strong+} to
	(a kind of) \emph{norm resolvent convergence} in various operator norms 
	and derive estimates on its rate. 
	In the literature (see below)  they are usually called \emph{operator estimates}.
	As a consequence (though not an immediate one) we derive   the estimate for the distance between the spectra of $\A\e$ and $\A$ in the weighted Hausdorff metrics.
	The second goal is to go beyond geometry being treated in \cite{DelV87}
	(periodically distributed straight identical passages) -- 
	the assumptions we impose on the geometry and distribution of the passages are rather general.
	As a result, we arrive at the interface conditions \eqref{BVP:hom:4} 
	with $\mu$ being a non-negative function (not necessary constant). 
	The last goal is to complete the results obtained in \cite{DelV87} -- 
	in this paper only the case $n\ge 3$ was investigated in full, while the  calculation of $\mu$ in the  case $n=2$, $p>0$, $0<q\le\infty$ remained to be open problem.

	Operator estimates   in homogenization theory originate from the pioneer works  \cite{BS04,BS06,Gr04,Z05a,Z05b,ZP05,Gr06,KLS12} concerning  periodic elliptic operators with rapidly oscillating coefficients.
	In the last decade  operator estimates were established
	for various boundary value problems in domains perforated by small holes, see
	\cite{Su18,ZP16,ZP05,KP18,AP21,KP22,BK22,Bo23a,Bo23b} (volume distribution of holes)
	and \cite{BCD16,BM21,BM22,GPS13} (surface distribution of holes).
	We also mention closely related  works \cite{BC09,BBC10,BBC13,CDD20} and \cite{BCFP13}
	concerning operator estimates for elliptic operators with frequently alternating boundary conditions and  boundary value problems in domains with fast oscillating boundary, respectively.
	Finally, we refer to our   recent paper \cite{Kh23}, where we derive operator estimates 
	for the Neumann sieve problem with a zero-thickness sieve.

	\subsection{Sketch of main results}\label{subsec:12}
	
	In this work we impose a restriction on the  domain $\Omega$:
	we assume that it is unbounded  and  $\Gamma\subset\Omega$ with $\dist(\Gamma,\partial\Omega)>0$. 
	For example, $\Omega$ may coincide with the whole $\R^n$ or $\Omega$ may be an unbounded waveguide-like domain. The technical difficulties arising if   $\Gamma$ is allowed to intersect $\partial\Omega$ are explained in Remark~\ref{rem:intersection} at the very end of Section~\ref{sec:5}.
	Nevertheless, under some extra restrictions on the location of passages, one can avoid these difficulties -- again see Remark~\ref{rem:intersection} for   details. 
	\smallskip
	
	Our first result (Theorem~\ref{th1}) provides the $L^2\to L^2$ operator estimate 
	\begin{gather*}  	
		\|u\e- u\|_{  \L(\Omega\e)}\leq \sigma\e\|f\|_{\L(\Omega)} ,
	\end{gather*}
	with  some explicitly defined $\sigma\e\to 0$, see \eqref{sigma:e};  just like before, $u\e=(\A\e+\Id)^{-1}$ and $u=(\A+\Id)^{-1}$ stand for the solutions to the problems 
	\eqref{BVP:main:1}--\eqref{BVP:main:2} and \eqref{BVP:hom:1}--\eqref{BVP:hom:4}, respectively.
	Furthermore, we establish three other versions of  $L^2\to L^2$  
	operator estimates  involving the operator of extension by zero $\wt\J\e:\L(\Omega\e)\to\L(\Omega)$ (cf.~\eqref{Je:wtJe}).
	
	The second result (Theorem~\ref{th2}) claims the convergence  in the
	$\L \to \H^1 $ operator norm, where we  need to use  special correctors.
	The following estimates are derived:
	\begin{align*} 
		\|u\e- u - \wt u\e^\pm\|_{  \H^1(\Omega\e^\pm)}\leq \sigma\e\|f\|_{\L(\Omega)},
		\\ 
		\|u\e  - \wt u\e^T\|_{  \H^1(\cup_k T\ke)}\leq \sigma\e\|f\|_{\L(\Omega)},
	\end{align*}
	with   $\wt u\e^\pm\in\H^1(\Omega^\pm\e)$ and 
	$\wt u\e^T\in\H^1(\cup_k T\ke)$ obeying the following   properties as $\eps\to 0$:
	\begin{gather*} 
		\|\wt u\e^\pm \|_{\L(\Omega\e^\pm)}=\mathcal{O}(\sigma\e)\|f\|_{\L(\Omega)},\quad
		\|\wt u\e^T \|_{\L(\cup_{k }T\ke)}=\mathcal{O}(\sigma\e)\|f\|_{\L(\Omega)},
		\\[1mm] 
		\|\nabla \wt u\e^+\|^2_{\L(\Omega^+\e)}+
		\|\nabla \wt u\e^-\|^2_{\L(\Omega^-\e)}+
		\|\nabla \wt u\e^T\|^2_{\L(\cup_{k}T\ke)}
		=
		\|{\mu}^{1/2}(u_+-u_-)\|^2_{\L(\Gamma)}+\mathcal{O} (\sigma\e)\|f\|^2_{\L(\Omega)},
	\end{gather*} 
	The correctors $\wt u\e^\pm$ are supported in a small neighborhood 
	of the sieve.
	
	The third result (Theorem~\ref{th3})  is  the estimate
	\begin{gather*} 
		\mathrm{dist}_{\rm H}(\sigma((\A\e+\Id)^{-1}),\sigma((\A+\Id)^{-1}))\leq 
		C\sigma\e ,
	\end{gather*}
	where $\mathrm{dist}_{\rm H}(\cdot,\cdot)$ stands for the  Hausdorff distance,
	$\sigma(\A\e)$ and  $\sigma(\A)$ are the spectra of $\A\e$ and $\A$, respectively.

	Our proofs rely on an abstract scheme  for studying convergence of operators in varying Hilbert spaces, which was developed by O.~Post in \cite{P06} and was further elaborated in  \cite{P12,AP21,KP21,PS19,MNP13}. Originally, this abstract framework was developed to study convergence of Laplacians on  manifolds  shrinking to a graph \cite{P06,P12}. It also has shown to be effective for homogenization problems in domains with holes, see \cite{KP18,AP21,KP22}.
	The main challenge is to  construct suitable identification operators between the spaces 
	$\H^1(\Omega\setminus\Gamma)$ and $\H^1(\Omega\e)$ (these spaces are the domains of the sesquilinear forms associated with the operators $\A\e$ and $\A$).
	
	\subsection{Structure of the paper}
	The rest of the paper is organized as follows. In Section~\ref{sec:2} we set up the problem 
	and formulate the main results.
	In Section~\ref{sec:3} we recall aforementioned abstract results.
	In Section~\ref{sec:4} we collect several auxiliary estimates.
	The proof of the main results is carried out in Section~\ref{sec:5}.
	Finally, in Sections~\ref{sec:6} and \ref{sec:7} we discuss various examples for which the assumptions
	we pose on  the passages  are fulfilled.

	\section{Setting of the problem and main results}
	\label{sec:2}
	
	\subsection{The domain $\Omega\e$}
	
	Let $n\in\N\setminus\{1\}$ (space dimension).
	In the following, 
	$$x'=(x_1,\dots,x_{n-1})\text{\quad and\quad }x=(x',x_n)$$ 
	stand  for  the Cartesian coordinates in $\R^{n-1}$ and $\R^{n}$, respectively.
	We denote
	$$
	\Gamma\ceq  \left\{x=(x',x_n)\in\R^n:\ x_n=0\right\}.
	$$
	Let
	$\Omega\subset\R^n$ be an unbounded Lipschitz domain satisfying
	\begin{gather}\label{Gamma:dist}
		\Gamma\subset\Omega,\quad  L\ceq\dist(\Gamma,\partial\Omega)>0.
	\end{gather}
	Let  $\eps\in (0,\eps_0]$ be a small parameter, which is 
	assumed to be sufficiently small: later on, we will specify this claim more precisely (see \eqref{eps0:last:1}--\eqref{eps0:4}), at the moment it is only enough to claim $\eps_0<L$ in order to get $\overline{O\e}\subset\Omega$, where by $O\e$ we denote  $\eps$-neighborhood of $\Gamma$, i.e. 
	\begin{gather}\label{Oe}
		O\e\ceq  \left\{x=(x',x_n)\in\R^n:\ |x_n|<\eps\right\}.
	\end{gather}
	We also introduce the sets
	\begin{align}\label{Omegapm:Gammapm}
		\Omega^\pm\e   \ceq\Omega\cap \Xi\e^\pm,\quad
		\Omega^\pm   \ceq\Omega\cap \Xi^\pm,\quad
		\Gamma^\pm\e \ceq  \partial\Xi\e^\pm= \partial O\e \cap\Xi^\pm,
	\end{align}
	where  
	\begin{align}\label{Xie}
		{\Xi}\e^\pm&\ceq\left\{x=(x',x_n)\in\R^n:\ \pm x_n> \eps\right\},
		\\
		\label{Xi0}
		\Xi^\pm&\ceq  \{x=(x',x_n)\in\R^n:\ \pm x_n>0\},
	\end{align}
	i.e, $\Omega^+\e$ and $\Omega^-\e$ (resp., $\Omega^+$ and $\Omega^-$)
	are the subsets of $\Omega$ lying above and below $O\e$ (resp., above and below $\Gamma$), and $\Gamma^+\e$ and $\Gamma^-\e$ are the top and bottom parts of $\partial O\e$.
	
	Next, we connect $\Omega^+\e$ and $\Omega^-\e$ by drilling a lot of passages in $O\e$.
	Let 
	$\left\{T\ke ,\ k\in\Z^{n-1}\right\}$ be a family of  
	domains in $\R^n$ satisfying
	\begin{itemize}
		\item $T\ke\subset O\e$, $\forall k\in\Z^{n-1}$,\smallskip
		
		\item $\overline{T\ke}\cap \overline{T_{j,\eps}}=\emptyset$, $\forall k,j\in\Z^{n-1}$, $k\not=j$,\smallskip
		
		\item  
		$ 
		D\ke^+\ceq \partial T\ke\cap\Gamma\ke^+\ 
		\text{(resp., }D\ke^-\ceq \partial T\ke\cap\Gamma\ke^-\text{)}
		$ 
		is a non-empty  relatively open in  $\Gamma\ke^+$ (resp., $\Gamma\ke^-$) connected  set.
		The sets $D\ke^+$ and $D\ke^-$ are the top and bottom faces of $T\ke$, respectively.

	\end{itemize}
	Note that instead of $\Z^{n-1}$ one can use an arbitrary infinite countable set (e.g., $\N$),
	but it is more convenient to deal with $\Z^{n-1}$ when it comes to the case of periodically distributed passages
	(see Subsection~\ref{subsec:7:2}).

	Finally, we define the domain $\Omega\e$:
	\begin{gather}\label{Omegae}
		\Omega\e\ceq 
		{\rm int}(\overline{\Omega\e^+\cup\Omega\e^-\cup\left(\cup_{k\in\Z^{n-1}}T\ke\right)})=
		\Omega\setminus\overline{\Sigma\e},
	\end{gather}
	where $\Sigma\e=O\e\setminus\overline{\cup_k T\ke}$ (``sieve'').
	The domain $\Omega\e$ is presented on Figure~\ref{fig:Omegae}.

	\subsection{Assumptions on passages}
	
	In this subsection we impose certain assumptions on the passages
	$T\ke$.
	The last assumption \eqref{assump:main} will be given in the next subsection.
	\smallskip
	
	Throughout the whole paper, the notation 
	\begin{gather*}   
		\text{$\B(r,z)$ stands for
			the open ball in $\R^n$ of radius $r>0$ and center $z\in\R^n$.}
	\end{gather*}
	
	By
	$d\ke^+$ and $x\ke^+$ (resp., $d\ke^-$ and $x\ke^-$) we denote the radius and the center of the smallest ball $\B( d\ke^+,x\ke^+)$ (resp., $\B( d\ke^-,x\ke^-)$\,) containing the set $D\ke^+$ (resp., $D\ke^-$). Since $D\ke^\pm\subset\Gamma^\pm\e$, we obviously have $x\ke^\pm\in \Gamma\e^\pm$. 
	
	We assume that the following condition hold: for each $\eps\in (0,\eps_0]$ there exists a  sequence $\left\{\rho\ke,\ k\in\Z^{n-1}\right\}$ 
	of positive numbers
	satisfying
	\begin{gather}
		\label{assump:0} 
		\rho\e\ceq \sup_{k\in\Z^{n-1}}\rho\ke\to 0\text{ as }\eps\to0,
		\\
		\label{assump:1}
		\forall k\not=j:\quad
		\B(\rho\ke,x^+\ke)\cap \B(\rho_{j,\eps},x^+_{j,\eps}) =\emptyset
		\quad\text{and}\quad \B(\rho\ke,x^-\ke)\cap \B(\rho_{j,\eps},x^-_{j,\eps}) =\emptyset
		,\\[1mm]
		\label{assump:3}
		\gamma\e\ceq \max\left\{\sup_{k\in\Z^{n-1}}\gamma\ke^+,\,\sup_{k\in\Z^{n-1}}\gamma\ke^-\right\}\leq C,
	\end{gather}  
	where  the function $\GG(t)$ is defined by \eqref{Gt},
	the numbers $\ga\ke^\pm$ are given by
	$$
	\gamma\ke^\pm\ceq (\GG(d\ke^\pm))^{-1} \rho\ke^{1-n},$$
	and the constant $C>0$ is independent of $\eps$.
	Furthermore, we assume that the following condition on $T\ke$ is fulfilled:
	\begin{gather}
		\label{non:con}
		\forall u\in \H^1(\Omega\e):\  \|u\|_{\L(\cup_{k\in\Z^{n-1}}T\ke)} \leq 
		\zeta\e\|u\|_{\H^1(\Omega\e)} \text{ with }\zeta\e\to 0\text{ as }\eps\to 0.
	\end{gather}
	Mimicking \cite[Definition~3.7]{AP21}, we call the inequality \eqref{non:con} \emph{non-concentrating property}. Later on (see Section~\ref{sec:6}), we will discuss this property in more details. Here we only note that \eqref{non:con} holds, e.g., for straight or moderately (in a suitable sense) bent passages $T\ke$; on the other hand, one can destroy the fulfillment of \eqref{non:con} by perturbing at least one of the passages  attaching to it an appropriately chosen tiny bump.

	Note that \eqref{assump:0}, \eqref{assump:3} imply
	\begin{align} 
		\label{assump:3+}
		\sup_{k\in\Z^{n-1}}{d\ke^\pm\over\rho\ke}\to 0\text{ as }{\eps\to 0},
	\end{align}	
	whence, in particular, we get $\overline{D\ke^\pm}\subset\B(\rho\ke,x\ke^\pm)$
	for small enough $\eps$.
	
	Finally, we pose the following conditions on $\eps_0$ (the maximum value of the  parameter $\eps$):
	\begin{align}\label{eps0:last:1}
		&\eps_0\le  L/8,\\ \label{eps0:last:2}
		&\eps_0\le  (2L)^{-1} .
	\end{align} 
	Moreover, we claim that $\eps_0$ sufficiently small in order to have 
	\begin{align}
 \label{eps0:3}
		\forall\eps\in (0,\eps_0] :&\quad\zeta\e\le {1/2},
  \\
		\label{eps0:1+}
		\forall\eps\in (0,\eps_0]\ \forall k\in\Z^{n-1}:&\quad
		\rho\ke \le {L/ 4},
		\\
  \label{eps0:2}
		\forall\eps\in (0,\eps_0]\ \forall k\in\Z^{n-1}:&\quad d\ke^\pm \rho\ke^{-1}\le {1/4},
	\end{align}
        and additionally, for $n=2$ we claim
        \begin{align}
		\label{eps0:1} 
		\forall\eps\in (0,\eps_0]\ \forall k\in\Z^{n-1}:& \quad\rho\ke\leq {1/ 2}, 
		\\
		\label{eps0:4}
		\forall\eps\in (0,\eps_0]\ \forall k\in\Z^{n-1}:&\quad|\ln d\ke^\pm|^{-1}|\ln\rho\ke|\le {1/2}.
        \end{align}
        Such a choice of $\eps_0$ is always possible due to 
        \eqref{assump:0}, \eqref{assump:3}--\eqref{assump:3+}.

	\subsection{The quantities $\mathscr{C}\ke$ and the last assumption}
	\label{subsec:Cke}

	To formulate the last assumption, we need to define the capacity-type quantities  $\mathscr{C}\ke$.
	
	We introduce the sets (see Figure~\ref{fig:Gke})
	\begin{gather}\label{BSG}
		B\ke^\pm\ceq\B(\rho\ke,x\ke^\pm)\cap {\Xi}\e^\pm,
		\quad
		S\ke^\pm\ceq \partial B\ke^\pm\setminus\Gamma\e^\pm,\quad
		G\ke\ceq {\rm int}(\overline{T\ke\cup  B\ke^+ \cup B\ke^-}).
	\end{gather}
	Note that due to \eqref{eps0:last:1}, \eqref{eps0:1+}, one has $B\ke^\pm\subset \Omega\e^\pm$, whence $G\ke\subset\Omega\e$.
	
	\begin{figure}[h]
		\centering
		\begin{picture}(225,200)
			\includegraphics[width=0.5\textwidth]{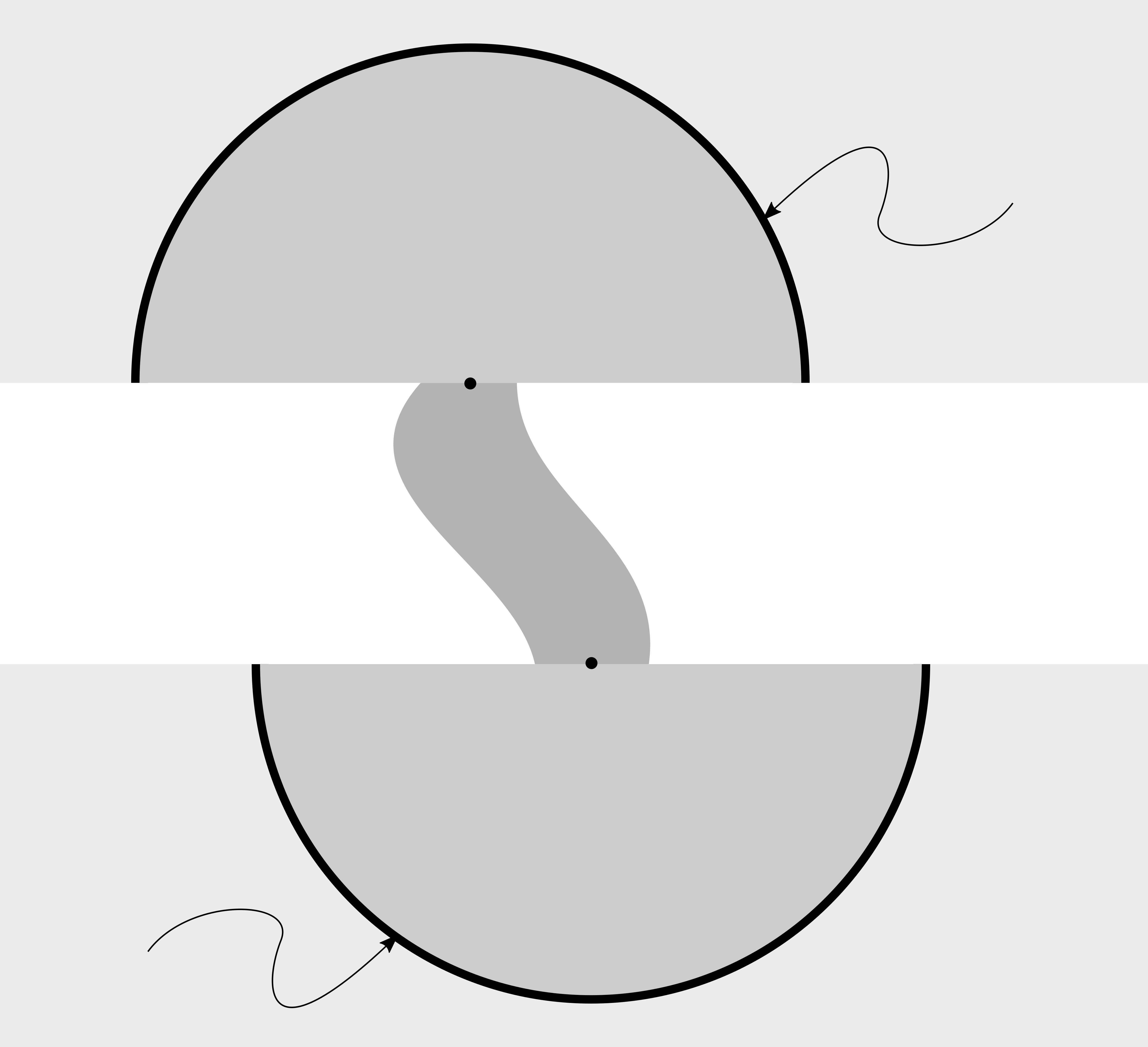} 
			\put(-110,20){$B\ke^-$}\put(-140,179){$B\ke^+$}
			\put(-130,100){$T\ke$}
			\put(-140,135){$x\ke^+$}\put(-110,65){$x\ke^-$}
			\put(-30,165){$S\ke^+$}\put(-205,10){$S\ke^-$}
			\put(-13,186){$\Xi\e^+$}
			\put(-13,10){$\Xi\e^-$}
		\end{picture}
		\caption{The set $G\ke\ceq {\rm int}(\overline{T\ke\cup  B\ke^+ \cup B\ke^-})$ }
		\label{fig:Gke}
	\end{figure}

	We consider the following boundary value problem in $G\ke$:
	\begin{align}\label{BVP:U:1}
		\Delta U\ke =0&\text{ in }G\ke,\\\label{BVP:U:2}
		U\ke =1&\text{ on }S\ke^+,\\\label{BVP:U:3}
		U\ke =0&\text{ on }S\ke^-,\\\label{BVP:U:4}
		\ds{\partial U\ke\over\partial\nu}=0&\text{ on }\partial G\ke \setminus(S\ke^+\cup S\ke^-),
	\end{align} 
	where ${\partial\over\partial\nu }$ is the derivative along the outward pointing normal to $\partial G\ke$. Let $U\ke (x)$ be the solution to \eqref{BVP:U:1}--\eqref{BVP:U:4}.
	Then we define   $\mathscr{C}\ke$ as follows,
	\begin{gather}
		\label{CUU}
		\mathscr{C}\ke=\|\nabla U\ke\|^2_{\L(G\ke)}.
	\end{gather}
	Equivalently, the quantities $\mathscr{C}\ke$ can be defined by
	\begin{gather}\label{Cke}
		\mathscr{C}\ke= \inf_{U\in \mathscr{U}\ke}\|\nabla U\|^2_{\L(G\ke)},
	\end{gather}
	where $\mathscr{U}\ke=\left\{U\in  \H^1({G\ke}):\ U\restr_{S\ke^+}=1,\ U\restr_{S\ke^-}=0\right\}$.

	Our last assumption is as follows:
	there exists  a real-valued function ${\mu}\in C^{1}(\Gamma)\cap\W^{1,\infty}(\Gamma)$ and
	$\kappa\e>0$ 
	with $\kappa\e\to 0$
	as $\eps\to 0$
	such that 
	\begin{multline}
		\label{assump:main}  
		\forall g\in \H^{2}(O\setminus\Gamma)\ \forall h\in \H^{1}(O\setminus\Gamma):\\ 
		\left|\suml_{k\in\Z^{n-1}}\mathscr{C}\ke
		\left(\la g \ra_{B\ke^+}-\la g \ra_{B\ke^-}\right)
		\left(\la \overline{h} \ra_{B\ke^+}-\la \overline{h} \ra_{B\ke^-}\right)-({\mu} [g],[h])_{\L(\Gamma)}\right| \leq
		\kappa\e\| g\|_{\H^{2}(O\setminus\Gamma )}\|h\|_{\H^{1}(O\setminus\Gamma )}.
	\end{multline}
	Here 
	\begin{itemize}
		\item by $\la v\ra_{B\ke^\pm}$ we denote  the mean value of a function $v$ in $B\ke^\pm$, 
		i.e.
		\begin{equation*}
			\la v\ra_{B\ke^\pm}\ceq { |B\ke^\pm|^{-1}}\int_{B\ke^\pm}v(x)\d x,\ |B\ke^\pm|\text{ is the volume of }B\ke^\pm,
		\end{equation*}
		\item the set $O$ is the $L/2$-neighborhood of $\Gamma$, i.e.
		\begin{gather}\label{O}
			O\ceq 
			\left\{x=(x',x_n)\in\R^n:\ |x_n|< {L}/{2}\right\},
		\end{gather} 
		\item  by $[f] $ we denote the jump of $f$  across $\Gamma$, namely,
		\begin{gather}\label{jump}
			[f] \ceq  f^+-f^- .
		\end{gather}
		where    $f^+$ (resp., $f^-$)
		stands for the trace of $f\restr_{\Omega^+}\in \H^1(\Omega^+)$ (resp., $f\restr_{\Omega^-}\in \H^1(\Omega^-)$) on $\Gamma$;
		it is well-known that $f^\pm\in\H^{1/2}(\Gamma)$ provided $f\in\H^1(\Omega\setminus\Gamma)$.
	\end{itemize}

        \begin{remark}
        For any $g,h\in\H^1(O\setminus\Gamma)$ one has the estimate
\begin{gather}\label{sumCke:bound}
\left|\suml_{k\in\Z^{n-1}}\mathscr{C}\ke
		\left(\la g \ra_{B\ke^+}-\la g \ra_{B\ke^-}\right)
		\left(\la \overline{h} \ra_{B\ke^+}-\la \overline{h} \ra_{B\ke^-}\right)\right|\leq 
  C\|g\|_{\H^1(O\setminus\Gamma)}\|h\|_{\H^1(O\setminus\Gamma)},
  \end{gather}
    where the constant $C$ is independent  of $\eps$;
    \eqref{sumCke:bound} follows  easily from the assumptions \eqref{assump:1}, \eqref{assump:3} and the estimates \eqref{UV}, \eqref{nablaV}, \eqref{abs:fke}, which will be proven later on.
	In Section~\ref{sec:7} we discuss examples  for which  the sum standing in the left-hand-side of \eqref{sumCke:bound} is not only bounded uniformly in $\eps$, but also converges (as $\eps\to 0$) to the integral $({\mu} [g],[h])_{\L(\Gamma)}$ with some explicitly calculated function
   $\mu$, and, moreover, the estimate \eqref{assump:main} is fulfilled.
\end{remark}

	\subsection{The operators $\A\e$ and $\A$} 
	
	In the Hilbert space   $\L(\Omega\e)$ 
	we  define the sesquilinear form  
	\begin{equation}
		\label{ae}
		\a\e[u,v]=
		(\nabla u,\nabla v)_{\L(\Omega\e)}=
		\ds\int_{\Omega\e}
		\suml_{i=1}^n {\partial u\over\partial x_i} {\partial \overline{ v}\over\partial x_i} \d x,\quad
		\dom(\a\e)= \H^1(\Omega\e).
	\end{equation} 
	The form $\a\e$
	is symmetric, densely defined, closed, and positive, whence, by virtue of the first representation theorem \cite[Chapter 6, Theorem 2.1]{Ka66}, there exists the unique self-adjoint and positive operator $\A\e$ associated with $\a\e$,
	i.e. $\dom(\A\e)\subset\dom(\a\e)$ and
	\begin{gather*}
		\forall u\in
		\dom(\A\e),\ \forall  v\in \dom(\a\e):\ (\A\e u,v)_{\L(\Omega\e)}= \a\e[u,v].
	\end{gather*} 
	The operator $\A\e$ is called the Neumann Laplacian on $\Omega\e$. Our goal is to describe the behaviour its resolvent $(\A\e+\Id)^{-1}$  as $\eps\to 0$.
	\smallskip

	Next, we introduce the  anticipating limiting operator $\mathcal{A}$. 
	We define the sesquilinear form $\a$ in  the space $\L(\Omega)$ via
	\begin{align}  
		\a[f,g]=(\nabla f , \nabla g )_{\L(\Omega^+)} +(\nabla f  , \nabla g )_{\L(\Omega^-)} 
		+({\mu} [f] ,[g])_{\L(\Gamma)},
		\quad
		\dom(\a)=\H^1(\Omega\setminus\Gamma).\label{aga}
	\end{align}
	Here ${\mu}(x)$ is a function standing in the assumption \eqref{assump:main}, by
	$[\cdot]$ we denote the jump of a function  across $\Gamma$, see \eqref{jump}.
	This form is symmetric, densely defined, closed, and non-negative 
	(see, e.g., \cite[Proposition 3.1]{BEL14}). 
	We denote by $\A$ the associated self-adjoint operator.
	
	It is easy to show (see, e.g., \cite[Theorem~3.3]{BEL14}) that
	the operator $\A$ acts as follows (below in \eqref{Aga:action} and \eqref{Aga:domain}, the notations $f^\pm$
	stand not only for the traces of $f$ from the upper and lower sides of $\Gamma$ as in \eqref{jump},
	but also for $f\restriction_{\Omega^\pm}$): 
	\begin{gather}\label{Aga:action} 
		(\A f)\restr_{\Omega^\pm} = -\Delta  f^\pm,
	\end{gather}
	and its domain is given by  
	\begin{gather}\label{Aga:domain} 
		f\in\dom(\A)\ \Leftrightarrow\
		\begin{cases}
			f^\pm\in\H^1(\Omega^\pm),\\[1ex]
			\Delta  f^\pm\in \L(\Omega^\pm),\\[1ex]\ds
			{\partial  f^\pm \over\partial   \nu^\pm}\in \L(\partial\Omega ),\quad
			{\partial  f^\pm \over\partial   \nu^\pm}\restriction_\Gamma= 
			\mp{\mu}[f] ,\quad
			{\partial  f^\pm \over\partial   \nu^\pm}\restriction_{\partial\Omega^\pm\setminus\Gamma}= 
			0.
		\end{cases}
	\end{gather}
	where  ${\partial   f^\pm\over\partial   \nu^\pm}$ stands for the trace on $\Gamma$ of the derivative of $f^\pm$  
	along $\nu^\pm$ -- the outward (with respect to $\Omega^\pm$) pointing normal to $\partial\Omega^\pm$; in fact, one has ${\partial\over \partial\nu^\pm}=\mp{\partial\over \partial x_n}$ on $\Gamma$.
	
	\begin{remark}
		In fact, to define the form $\a$ and the associated operator $\A$ it is enough to assume that $\mu$ is a real-valued function from $\LL^\infty(\Gamma)$;
		the properties \eqref{Aga:action}--\eqref{Aga:domain}
		remain true. However, since we deal with more regular $\mu$
		(recall that in \eqref{assump:main} we assume  ${\mu}$ to be in $C^{1}(\Gamma)\cap\W^{1,\infty}(\Gamma)$),
		the functions from $\dom(\A)$ possesses even more regularity -- see~\eqref{H2est}. 
	\end{remark}

	\subsection{Main results}
	
	We define  the number  $\eta\e>0$   as follows,
	\begin{align*}
		\eta\e &\ceq \max\left\{\sup_{k\in\Z^{n-1}}\eta\ke^+,\,\sup_{k\in\Z^{n-1}}\eta\ke^-\right\},\text{\quad where\quad }
		\eta\ke^\pm\ceq  
		\begin{cases}
			d\ke^\pm\over \rho\ke,&n\ge 5,
			\\[1mm] { d\ke^\pm\over \rho\ke} \left|\ln { d\ke^\pm\over \rho\ke} \right|,&n=4,
			\\[1mm] 
			\left({ d\ke^\pm\over \rho\ke} \right)^{1/2},&n=3,
			\\[1mm] 
			\left|\ln { d\ke^\pm\over \rho\ke} \right|^{-1/2},&n=2.
		\end{cases}
	\end{align*}
	Due to  \eqref{assump:3+}, $\eta\e\to 0$ as $\eps\to 0$. We also set (cf.~Remark~\ref{rem:chi})
	$$
	\chi\e\ceq \min\{\eta\e,\,\rho\e^{1/2}\}.
	$$

	We introduce the operators $\J\e:\L(\Omega)\to\L(\Omega\e)$, $\wt\J\e:\L(\Omega\e)\to\L(\Omega )$ via
	\begin{gather}\label{Je:wtJe}
		\J\e f\ceq f\restriction_{\Omega\e},\qquad
		(\wt\J\e u)(x)\ceq
		\begin{cases}
			u(x),&x\in\Omega\e,\\
			0,&x\in \Sigma\e.
		\end{cases}
	\end{gather} 
	\smallskip
	
	In the following,  by $C,C_1,C_2,\dots$ we denote generic positive constants being independent of $\eps$; note that these constants may   vary from line to line.
	
	We are now in position to formulate the main results of this work.
	
	\begin{theorem}\label{th1}
		One has  
		\begin{align} 
			\label{th1:est:1}
			\forall f\in\L(\Omega):\quad &\|(\A\e+\Id)^{-1} \J\e f - \J\e(\A+\Id)^{-1} f \|_{ \L(\Omega\e)}\leq 
			C\sigma\e\|f\|_{\L(\Omega)},
		\end{align}
		where 
		\begin{gather}\label{sigma:e}\sigma\e\ceq \max\left\{
			\eps^{1/2},\,
			\zeta\e,\,
			\kappa\e,\,
			\chi\e
			\right\}.
		\end{gather}
		Moreover, one has the estimates  
		\begin{align} 
			\label{th1:est:2}
			\forall f\in\L(\Omega\e):\quad &\|\wt\J\e(\A\e+\Id)^{-1}f   -  (\A+\Id)^{-1} \wt\J\e f\|_{  \L(\Omega )}\leq 
			C\sigma\e\|f\|_{\L(\Omega\e)},
			\\
			\label{th1:est:3}
			\forall f\in\L(\Omega\e):\quad &\|(\A\e+\Id)^{-1} f  -   \J\e(\A+\Id)^{-1}  \wt\J\e f\|_{ \L(\Omega\e )}\leq 
			C\sigma\e\|f\|_{\L(\Omega\e)},
			\\
			\label{th1:est:4}
			\forall f\in\L(\Omega):\quad &\|\wt\J\e(\A\e+\Id)^{-1}   \J\e f  -  (\A+\Id)^{-1}f \|_{  \L(\Omega )}\leq 
			C\sigma\e\|f\|_{\L(\Omega)},
		\end{align}
		
	\end{theorem}

	\begin{remark}\label{rem:chi}
		Let us assume (additionally to \eqref{assump:3}) that
		\begin{gather}\label{gamma:ge:C}
			\forall\eps\in (0,\eps_0]\ \forall k\in\Z^{n-1}:\quad 0<C\leq \gamma\ke^\pm.
		\end{gather}
		In this case, we get easily
		\begin{gather}\notag
			\rho\ke^{1/2}\leq C\eta\ke^{\pm}\cdot
			\begin{cases}
				\rho\ke^{(n-4)/(2n-4)},&n\ge 5,\\
				|\ln (d\ke^\pm/\rho\ke)|^{-1},&n=4,
			\end{cases}
		\end{gather}
		Hence, if $n\ge 4$ and \eqref{gamma:ge:C} holds, then 
		$\chi\e=\rho\e^{1/2}$ (provided $\eps$ is small enough).
		However, if $d\ke^\pm$ are very small (in particular, \eqref{gamma:ge:C} is violated), then we get
		$\chi\e=\eta\e$. If $n=2,3$ and \eqref{gamma:ge:C} holds, then $\eta\ke^\pm$ and $\rho\ke^{1/2}$
		are of the same order, i.e.,
		$ C_1\rho\ke^{1/2}\leq \eta\ke^\pm\leq C_2\rho\ke^{1/2}$.
	\end{remark}

	The second theorem establishes the closeness of resolvents 
	$(\A\e+\Id)^{-1}$ and $(\A+\Id)^{-1}$ 
	in the $L^2\to H^1$ operator norm.  
	Here we need to employ   special correctors.
	We introduce the operators 
	$\mathscr{K}^\pm\e:\L(\Omega)\to \H^1(\Omega\e^\pm)$ and the operator 
	$\mathscr{K}\e^T:\L(\Omega)\to \H^1(\cup_{k\in\Z^{n-1}}T\ke)$ via 
	\begingroup
	\allowdisplaybreaks 
	\begin{align*} 
		(\mathscr{K}\e^+ f)(x)& =
		\begin{cases}
			\left(\la g\ra_{B\ke^-}- \la g\ra_{B\ke^+}\right)\left(1-U\ke(x)\right),&x\in B\ke^+ ,\\
			0,&x\in\ds\Omega\e^+\setminus\left(\cupl_{k\in\Z^{n-1}} B\ke^+\right),
		\end{cases} 
		\\[2mm]
		(\mathscr{K}\e^- f)(x)& =
		\begin{cases}
			\left(\la g\ra_{B\ke^+}- \la g\ra_{B\ke^-}\right) U\ke(x) ,&x\in B\ke^- ,\\
			0,&x\in\ds\Omega\e^-\setminus\left(\cupl_{k\in\Z^{n-1}} B\ke^-\right),
		\end{cases}
		\\[2mm]
		(\mathscr{K}\e^T f )(x)&= \la g \ra_{B\ke^+} U\ke (x) + \la g \ra_{B\ke^-} (1-U\ke(x)),\ x\in T\ke . 
	\end{align*}\endgroup
	where $$g\ceq (\A+\Id)^{-1}f.$$ 
	
	\begin{theorem}\label{th2}
		One has the estimates
		\begin{align}
			\label{th2:est:1}
			&\|(\A\e+\Id)^{-1} \J\e f  - \J\e(\A+\Id)^{-1} f   -\mathscr{K}^\pm\e  f \|_{ \H^1(\Omega\e^\pm)} \leq 
			C\sigma\e\|f\|_{\L(\Omega)},
			\\
			\label{th2:est:2}
			&\|(\A\e+\Id)^{-1}\J\e f  - \mathscr{K}\e^T  f  \|_{  \H^1(\cup_{k\in\Z^{n-1}}T\ke)} \leq 
			C\sigma\e\|f\|_{\L(\Omega)}.
		\end{align}
		The correcting operators $\mathscr{K}\e^\pm $ and $\mathscr{K}^T\e$ obey the following asymptotic      
		properties as $\eps\to 0$: 
		\begin{gather}\label{K:prop1}
			\|\mathscr{K}\e^\pm f\|_{\L(\Omega\e^\pm)}\leq  C\sup_{k\in\Z^{n-1}}(\rho\ke^{1/2}\gamma\ke^\pm)\|f\|_{\L(\Omega)},\\
			\label{K:prop2}
			\|\mathscr{K}\e^T f\|_{\L(\cup_{k\in\Z^{n-1}}T\ke)}\leq  
			C\zeta\e\gamma\e^{1/2}
			\|f\|_{\L(\Omega)},
			\\[1mm]\notag
			\|\nabla (\mathscr{K}\e^+ f)\|^2_{\L(\Omega^+\e)}+
			\|\nabla (\mathscr{K}\e^- f)\|^2_{\L(\Omega^-\e)}+
			\|\nabla (\mathscr{K}\e^T f)\|^2_{\L(\cup_{k\in\Z^{n-1}}T\ke)}\\
			=
			\|{\mu}^{1/2}[g]\|^2_{\L(\Gamma)}+\mathcal{O} (\kappa\e)\|f\|^2_{\L(\Omega)},
			\text{ where }g=(\A+\Id)^{-1}f. \label{K:prop3}
		\end{gather}
		
	\end{theorem}
	
	\begin{remark}
		It is easy to show (see \eqref{Ake:prop:1}--\eqref{Ake:prop:1+}) that $\sup_{k\in\Z^{n-1}}(\rho\ke^{1/2}\gamma\ke^\pm)\leq C\chi\e$,
		whence the right hand side in \eqref{K:prop1} is estimated by $C\sigma\e \|f\|_{\L(\Omega)}$.
		Similarly, due to \eqref{assump:3}, the right hand side in \eqref{K:prop2} is also estimated by $C\sigma\e \|f\|_{\L(\Omega)}$.
		That is why the terms $\mathscr{K}\e^\pm $ and $\mathscr{K}\e^T$
		do not contribute to $L^2\to L^2$ operator estimates, but (cf.~\eqref{K:prop3}) they cannot be omitted in the $L^2\to H^1$ estimates.
	\end{remark}

	Finally, as a consequence of Theorem~\ref{th1} (though, not an immediate consequence) we deduce a convergence of spectra.	Recall that
	for closed sets $X,Y\subset\R$  the \emph{Hausdorff distance} $\mathrm{dist}_{\rm H} (X,Y)$ is given by
	\begin{gather*}
		\mathrm{dist}_{\rm H} (X,Y)\ceq\max\left\{\sup_{x\in X} \inf_{y\in Y}|x-y|,\,\sup_{y\in Y} \inf_{x\in X}|y-x|\right\}.
	\end{gather*}
	Since the spectra of the non-negative, unbounded, self-adjoint operators $\A\e$ and $\A$ are
	closed unbounded subsets of $[0, \infty)$, it is reasonable to introduce the weighted metric
	$\wt{\mathrm{dist}_{\rm H}} (X,Y)$ by
	\begin{gather}\label{dH:weighted}
		\wt{\mathrm{dist}_{\rm H}} (X,Y)\ceq \mathrm{dist}_{\rm H}(\overline{(X+1)^{-1}},\overline{(Y+1)^{-1}}),\quad X,Y\subset [0,\infty),
	\end{gather}
	where for $Z\in [0,\infty)$ we denote $\overline{(Z+1)^{-1}}=\overline{\{x\in\R: \exists z\in Z\text{ such that }x=(z+1)^{-1}\}}$.
	With the metric $\wt{\mathrm{dist}_{\rm H}}$, two spectra can be close even if they differ significantly at high
	spectral parameter. Note that by spectral mapping theorem one has
	$$\wt{\mathrm{dist}_{\rm H}}(\sigma({\A\e}),\sigma({\A}))=\mathrm{dist}_{\rm H}(\sigma((\A\e+\Id)^{-1}),\sigma((\A+\Id)^{-1})).$$

	\begin{theorem}\label{th3}
		One has the estimate
		\begin{gather}\label{th3:est}
			\wt{\mathrm{dist}_{\rm H}}(\sigma({\A\e}),\sigma({\A}))\leq 
			C\sigma\e .
		\end{gather}
	\end{theorem}
	
	It follows from Theorem~\ref{th3} that
	$\wt{\mathrm{dist}_{\rm H}}(\sigma({\A\e}),\sigma({\A}))\to 0$ as $\eps\to 0,$
	whence, by virtue of \cite[Lemma~A.2]{HN99}, we immediately obtain the following
	corollary.
	
	\begin{corollary}
		One has the following properties:
		\begin{align*}
			&
			\forall\lambda_0\in \sigma(\A)\ \exists (\lambda\e)_{\eps>0}\text{ with }
			\lambda\e\in \sigma(\A\e)\text{ such that }\lambda\e\to \lambda_0\text{ as }\eps\to 0,\text{ and}
			\\\label{Haus:2}
			&\forall \lambda_0\in\R\setminus \sigma(\A)\   
			\exists\delta>0\ \exists\wt\eps>0
			\text{ such that }\sigma(\A\e)\cap (\lambda_0-\delta,\lambda_0+\delta)=\emptyset
			\text{ for }\eps<\wt\eps.
		\end{align*}     
	\end{corollary}

	\section{Abstract tools}
	\label{sec:3}
	
	In this section we present   abstract results from
	\cite{AP21,KP21,P06} describing resolvent and spectral convergence of operators acting 	in varying Hilbert spaces. These results will be further utilized for the proof of Theorems~\ref{th1},\,\ref{th2},\,\ref{th3}.\smallskip

	Let $(\HS\e)_{\eps>0}$ be a family of Hilbert spaces, 
	$(\A\e)_{\eps>0}$ be a family of non-negative, self-adjoint, unbounded operators acting in these spaces, $(\a\e)_{\eps>0}$ be a family of associated 
	sesquilinear forms. Also, let $\HS$ be a Hilbert space, $\A $ be a non-negative, self-adjoint, unbounded operator in $\HS$, and $\a$ be the associated 
	sesquilinear form.  
	We  define the Hilbert spaces $\HS^1\e$ and $\HS^1$ via
	\begin{equation}\label{H1spaces}
		\begin{array}{ll}
			\HS^1\e=\dom(\a\e),&
			\|u\|_{\HS^1\e}^2=\a\e[u,u]+\|u\|^2_{\HS\e},\\[2mm]
			\HS^1=\dom(\a),&
			\|f\|_{\HS^1}^2=\a[f,f]+\|f\|^2_{\HS},
		\end{array}
	\end{equation}
	and the Hilbert space $\HS^2$ by
	\begin{gather}\label{H2space}
		\HS^2=\dom(\A ),\quad 
		\|f\|_{\HS^2}=\|{(\A+\Id)f}\|_{\HS}.
	\end{gather}
	
	In the following, the notation $\|\cdot\|_{X\to Y}$ stands for the norm of a bounded linear operator acting between Hilbert spaces $X$ and $Y$.
	
	It is widely known (see, e.g., \cite[Theorem~VI.3.6]{Ka66} or \cite[Theorem~VIII.25]{RS72}) that convergence (in a suitable sense) of sesquilinear forms with \textit{common domain} implies norm resolvent convergence of the associated operators. 	In \cite{P06}  O.~Post established several versions of the above fact  to the setting of varying spaces.

	\begin{theorem}[{\cite[Theorem~A.5]{P06}}]
		\label{thA1}
		Let 
		\begin{gather}\label{JJJJ}
			\J\e \colon \HS\to  \HS\e,\quad {\wt\J\e }\colon {\HS\e}\to \HS,\quad
			{ \J\e ^1} \colon {\HS^1} \to {\HS^1\e},\quad
			{\wt\J\e ^{1}} \colon {\HS\e^1}\to {\HS^1}
		\end{gather}
		be linear operators satisfying
		\begin{align} 
			\label{thA1:0}   
			|(u,\J\e f)_{\HS\e} - (\wt\J\e u,f)_{\HS}|&\leq \delta\e\|f\|_{\HS}\|u\|_{\HS\e},&& \forall     f\in\HS,\, u\in\HS\e,
			\\	
			\label{thA1:1}
			\|\J\e^1 f-\J\e f\|_{\HS\e}&\leq \delta\e\|f\|_{\HS^1 },&& \forall f\in \HS^1 ,
			\\[1mm] 
			\label{thA1:2}
			\|\wt\J\e^1 u - \wt\J\e u \|_{\HS}&\leq 
			\delta\e\|u\|_{ \HS\e^1},&&  \forall u\in \HS^1\e, 
			\\[1mm]
			\label{thA1:3}
			|\a\e[ u,\J^1\e f]-\a[\wt\J^{1}\e u,f] |&\leq 
			\delta\e\|f\|_{\HS^2 }\|u\|_{\HS^1\e},&& \forall f\in \HS^2 ,\ u\in \HS^1\e  
		\end{align}
		with some  $\delta\e\geq 0$.
		Then one has the estimate
		\begin{gather}\label{thA1:est}
			\|(\A\e+\Id)^{-1}\J\e -\J\e (\A+\Id)^{-1}\|_{\HS\to\HS\e} \leq 4\delta\e.
		\end{gather}
	\end{theorem} 
	
	\begin{remark}
		In  applications the operators $\J\e$ and $\wt\J\e$ usually appear in a natural way --
		in our case $\J\e$  and $\wt\J\e$ are defined  by \eqref{Je:wtJe}. The properties \eqref{thA1:1}--\eqref{thA1:2} suggest that the other two operators $\J^1\e$ and $\wt\J^1\e$ should be constructed as ``almost'' restrictions of $\J\e$ and $\wt\J\e$ to $\HS^1$ and $\HS^1\e$, respectively.
	\end{remark}
	
	\begin{theorem}[{\cite[Theorem~A.10]{P06}}]
		\label{thA2}
		Let $\J\e \colon \HS\to  \HS\e$ and ${\wt\J\e }\colon {\HS\e}\to \HS$ be          linear operators satisfying \eqref{thA1:0} and
		\begin{align} 
			\label{thA2:1} 
			\|f-\wt\J\e \J\e f\|_{\HS}\leq \delta\e\|f\|_{\HS^1},&\quad \forall f\in\HS^1,
			\\ \label{thA2:2} 
			\|u-\J\e \wt\J\e u\|_{\HS\e}\leq \delta\e\|u\|_{\HS\e^1},&\quad \forall          u\in\HS^1\e,
			\\
			\label{thA2:3} 
			\|\J\e f\|_{\HS }\leq 2\|f\|_{\HS},&\quad \forall f\in\HS,
			\\ \label{thA2:4} 
			\|\wt\J\e u\|_{\HS }\leq 2\|u\|_{\HS\e},&\quad \forall          u\in\HS\e,
		\end{align}
		with some  $\delta\e\ge 0$, and satisfying
		\begin{gather*}   
			\|(\A\e+\Id)^{-1}\J\e f -\J\e (\A+\Id)^{-1} f\|_{ \HS\e} \leq  \delta'\e\|f\|_{\HS}, \quad \forall     f\in\HS, 
		\end{gather*}
		with some $\delta'\e\geq 0$.
		Then one has the estimates
		\begin{gather*} 
			\|\wt\J\e(\A\e+\Id)^{-1} - (\A+\Id)^{-1}\wt\J\e\|_{\HS\e\to\HS} \leq C(\delta\e+\delta\e'),\\
			\|\wt\J\e(\A\e+\Id)^{-1}\J\e - (\A+\Id)^{-1}\|_{\HS\to\HS } \leq C(\delta\e+\delta\e'),\\
			\|(\A\e+\Id)^{-1} -\J\e (\A+\Id)^{-1}\wt\J\e\|_{\HS\e\to\HS\e} \leq C(\delta\e+\delta\e'),
		\end{gather*}
		where the constant $C$ is absolute.
	\end{theorem} 
	
	In \cite{AP21},  C.~Ann\'{e} and O.~Post  extended Theorem~\ref{thA1}  
	to a (suitably sandwiched) resolvent difference regarded  as an operator from 
	$\HS$ to $\HS\e^1$.

	\begin{theorem}[{\cite[Proposition~2.5]{AP21}}]
		\label{thA3}
		Let $\J\e \colon \HS\to  \HS\e$, ${\wt\J\e }\colon {\HS\e}\to \HS$,
		${ \J\e ^1} \colon {\HS^1} \to {\HS^1\e}$, ${\wt\J\e ^{1}} \colon {\HS\e^1}\to {\HS^1}$ be linear operators satisfying  
		\eqref{thA1:0}--\eqref{thA1:3}, \eqref{thA2:1}, \eqref{thA2:2} and 
		\begin{gather}\label{thA3:1}
			\|\J\e f\|_{\HS\e}\leq (1+\delta\e)\|f\|_{\HS}
		\end{gather}
		with some  $\delta\e\geq 0$.
		Then one has the estimate
		\begin{gather}\label{thA3:est}
			\|(\A\e+\Id)^{-1}\J\e -\J\e^1(\A+\Id)^{-1}\|_{\HS\to\HS^1\e} \leq 6\delta\e.
		\end{gather}
	\end{theorem} 
	
	\begin{remark}\label{rem:AP21}
		Above, we formulated Theorem~\ref{thA2} in the form it was given in \cite{AP21}.
		However, tracing the proof of  Proposition~2.5 from \cite{AP21}, one can  easily observe that neither of 
		the estimates \eqref{thA2:1}, \eqref{thA2:2}, \eqref{thA3:1} are  used for the      derivation of \eqref{thA3:est}.   
	\end{remark}

	The last theorem gives the estimate for the distance between the spectra of $ \A\e $ and 
	$ \A $
	in the weighted Hausdorff metrics $ \wt{\mathrm{dist}_{\rm H}}(\cdot,\cdot)$ (see~\eqref{dH:weighted}).
	It is well-known, that norm convergence of bounded self-adjoint operators in a fixed Hilbert
	space implies Hausdorff convergence of spectra of the underlying resolvents.
	More precisely, if the spaces $\HS\e$ and $\HS$   coincide, 
	then (see~\cite[Lemma A.1]{HN99}) one has the estimate
	$$
	\wt{\mathrm{dist}_{\rm H}}(\sigma(\A\e), \sigma(\A)) \leq 
	\|(\A\e+\Id)^{-1}-(\A+\Id)^{-1}\|_{\HS\to\HS}.
	$$
	The theorem below is an analogue of this result for the case of operators $\A\e$ and $\A$ 
	acting in different Hilbert spaces $\HS\e$ and $\HS$.

	\begin{theorem}[{\cite[Theorem~3.4]{KP21}}]\label{thA4}
		Let  
		$\J\e \colon\HS\to {\HS\e}$, $\wt\J\e \colon\HS\e\to {\HS} $ be  linear bounded operators satisfying 
		\begin{align*}
			\|(\A\e+\Id)^{-1}\J\e - \J\e (\A +\Id)^{-1} \|_{\HS\to \HS\e}&\leq \tau\e , 
			\\
			\|\wt\J\e (\A\e+\Id)^{-1} - (\A +\Id)^{-1}\wt\J\e  \|_{\HS\e\to  \HS}&\leq \wt\tau\e ,
		\end{align*}
		and, moreover,
		\begin{align}
			\label{thA4:3}
			\|f\|^2_{ \HS}&\leq \mu\e \|\J\e  f\|^2_{\HS\e}+\nu\e \,  \a[f,f],\quad \forall f\in \dom( \a),\\
			\label{thA4:4}
			\|u\|^2_{\HS\e}&\leq \wt\mu\e \|\wt\J\e  u\|^2_{\HS}+\wt\nu\e \,  
			\a\e[u,u],\quad \forall u\in \dom( \a\e)
		\end{align}
		for some positive constants $\tau\e ,\,\mu\e ,\,\nu\e ,\, \wt\tau\e ,\,\wt\mu\e ,\,\wt\nu\e  $.
		Then   one has
		\begin{gather}\label{thA4:est}
			\wt{\mathrm{dist}_{\rm H}}(\sigma(\A\e), \sigma(\A))\leq 
			\max
			\left\{
			{\nu\e\over 2}+\sqrt{{\nu\e^2\over 4}+\tau\e^2\mu\e},\,
			{\wt\nu\e\over 2}+\sqrt{{\wt\nu\e^2\over 4}+\wt\tau\e^2\wt\mu\e}
			\right\}.
		\end{gather}
	\end{theorem} 
	
	\begin{remark}
		\label{rem:tau}
		To be precise, in \cite{KP21} the obtained estimate reads as follows:
		\begin{gather}\label{thA4:est:initial}
			\forall t,\wt t\in (0,1):\quad
			\wt{\mathrm{dist}_{\rm H}}(\sigma(\A\e), \sigma(\A))
			\leq 
			\max\left\{
			\tau\e\sqrt{\mu\e\over t},\,{\nu\e\over 1-t },\,
			\wt\tau\e\sqrt{\wt\mu\e\over \wt t},\,{\wt\nu\e\over 1-t }
			\right\}.
		\end{gather}
		Obviously, the optimal choice of the parameters $t$ and $\wt t$ in \eqref{thA4:est:initial}  is the one for which
		\begin{gather}\label{opt:t}
			\tau\e\sqrt{\mu\e\over t}={\nu\e\over 1-t }\text{\quad and\quad }
			\wt\tau\e\sqrt{\wt\mu\e\over \wt t}={\wt\nu\e\over 1-t }.
		\end{gather}
		Finding  $t,\wt t\in (0,1)$ satisfying \eqref{opt:t} and inserting them to \eqref{thA4:est:initial}, one get easily \eqref{thA4:est}.
	\end{remark}
	
	\section{Auxiliary estimates}
	\label{sec:4}
	In this section we collect several useful estimates which will be used further in the proofs of the main
	results.
	
	In the following, we denote by $\la u\ra _{E}$  the mean value of a function $u$ in the domain $E$, i.e.
	\begin{equation*}
		\la u\ra_{E}=|E|^{-1}\int_{E}u(x)\d x,
	\end{equation*}
	where $|E|$ stands for the volume of $E$. 
	We keep the same notation for the mean value of a function $u$ on
	an $(n-1)$-dimensional compact surface $S$, i.e.
	\begin{equation*}
		\la u\ra_{S}=|S|^{-1}\int_{S}u\,\d s,
	\end{equation*}
	where $\d s$ is the density of the surface measure on $S$,
	$|S|=\int_S \d s$ stands for the  area of $S$.

	The first lemma demonstrates how some standard functional inequalities
	are influenced by a domain rescaling.

	\begin{lemma} 
		\label{lemma:PFTS} 
		Let $E\subset\R^n$ be a bounded Lipschitz domain, $S$ be a relatively open subset of 
		$\partial E$.              Let $\delta>0$, $z\in\R^n$, and $E_\delta\ceq \delta E+z$, $S_\delta\ceq \delta S+z$ (evidently, $S_\delta$ is a relatively open subset of $\partial E_\delta$).
		Then one has the estimates:
		\begin{align}\label{Poincare:De}
			\forall v\in \mathsf{H}^1(E_\delta):& \quad
			\|v-\la v\ra_{E_\delta}\|_{\L (E_\delta)}^2\leq  C\delta^2 \|\nabla v\|_{\L(E_\delta)}^2,
			\\\label{Friedrichs:De}
			\forall v\in \mathsf{H}^1(E_\delta)\text{ with }u\restr_{S_\delta}=0:& \quad
			\|v \|_{\L (E_\delta)}^2\leq  C\delta^2 \|\nabla v\|_{\L(E_\delta)}^2,
			\\\label{trace:De}
			\forall v\in \mathsf{H}^1(E_\delta):& \quad 
			\|v \|^2_{\L(E_\delta)} \leq
			C\left(\delta\|v \|_{\L(S_\delta)}^2+\delta^2\|\nabla v\|^2_{\L(E_\delta)}\right),
			\\\label{traceSigma:De}
			\forall v\in\H^1(E_\delta):& \quad \|v\|^2_{\L(S_\delta)}\leq 
			C\left(\delta^{-1}\|v\|^2_{\L(E_\delta)}+\delta\|\nabla v\|^2_{\L(E_\delta)}\right),
			\\\label{meandiff:est}\forall v\in\H^1(E_\delta):& \quad
			\left|\la v\ra_{S_\delta}-\la v\ra_{E_\delta}\right|^2\leq 
			C\delta^{2-n}\|\nabla v\|_{\L(E_\delta)}^2,
		\end{align}
		where the constant $C$ in \eqref{Poincare:De} depends only on the set $E$,
		the constants $C$ in \eqref{Friedrichs:De}--\eqref{meandiff:est} depend only on the sets $E$ and $S$.
		Furthermore, for each $p$ satisfying 
		\begin{gather}
			\label{p}
			p\in \left[1, \frac{2n}{n-4}\right]\text{\; as\; }n\geq 5,\quad
			p\in [1,\infty)\text{\; as\; }n=4,\quad
			p\in [1,\infty]\text{\; as\; }n=2,3,
		\end{gather} 
		one has the estimate
		\begin{gather}\label{Sobolev:De}
			\forall v\in \mathsf{H}^2(E_\delta):\ 
			\|v-\la v\ra_{E_\delta}\|_{\LL^p(E_\delta)}\leq  C_{p} \delta^{n/p+(2-n)/2}\|v\|_{\mathsf{H}^2(E_\delta)},
		\end{gather}
		where 
		the constant $C_{p}$ depends only on $E$ and on $p$; 
		here for $p=\infty$ we set $1/p=0$. 
		
	\end{lemma}
	
	\begin{remark}
		The point $z\in\R^n$ in Lemma~\ref{lemma:PFTS} plays absolutely no role for the estimates we obtained; the same holds for the subsequent Lemmas \ref{lemma:ring}, \ref{lemma:U}, \ref{lemma:V}. We have introduced this  $z$ simply because later on we will apply these lemmas for sets which are not necessarily centered at the origin.
	\end{remark}
	
	\begin{proof}[Proof of Lemma~\ref{lemma:PFTS}]
		One has the following standard Poincar\'e-type inequalities:
		\begin{align}\label{Poincare} 
			\forall v\in \mathsf{H}^1(E):& \quad
			\|v-\la v\ra_E\|_{\L(E)}^2\leq  C_{Neu} \|\nabla v\|_{\L(E)}^2,\\\label{Friedrichs}
			\forall v\in \mathsf{H}^1(E)\text{ with }u\restr_{S}=0:& \quad
			\|v \|_{\L (E)}^2\leq  C_{Dir} \|\nabla v\|_{\L(E)}^2,\\
			\label{trace}
			\forall v\in \mathsf{H}^1(E):&\quad
			\|v \|^2_{\L(E)} \leq
			C_{Rob}\left(\|v \|_{\L(S)}^2+\|\nabla v\|^2_{\L(E)}\right).   
		\end{align}
		Here $C_{Neu}^{-1}$, $C_{Dir}^{-1}$, $C_{Rob}^{-1}$ are the smallest non-zero 
		eigenvalues of the  Laplacian on  $E$ subject to the Neumann boundary conditions on $E\setminus S$
		and the Neumann boundary conditions, the Dirichlet boundary conditions, and the Robin boundary conditions 
		${\partial  v\over\partial   \nu} +v=0$ on $\partial E$, respectively; here $\nu$ is the  unit normal  pointed          outwards of $E$. Making the change of variables 
		$$E\ni y= (x-z)\delta^{-1}\text{ with }x\in E_\delta,$$ 
		we reduce \eqref{Poincare}--\eqref{trace} to \eqref{Poincare:De}--\eqref{trace:De}.
		We  also have the standard trace inequality
		\begin{gather*}
			\forall v\in\H^1(E):\ \|v\|^2_{\L(S)}\leq C\|v\|^2_{\H^1(E)}
		\end{gather*}
		(here $C$ depends only on   $E$ and $S$), which after the above change of variables reduces to
		the estimate \eqref{traceSigma:De}.
		
		Using the Cauchy-Schwarz inequality, \eqref{Poincare:De}, \eqref{traceSigma:De}, we get the  estimate \eqref{meandiff:est}:		
		\begin{align*}
			\left|\la v\ra_{S_\delta}-\la v\ra_{E_\delta}\right|^2&\leq
			\left|\la v-\la v\ra_{E_\delta}\ra_{S_\delta}\right|^2
			\leq |S_\delta|^{-1}\|v-\la v\ra_{E_\delta} \|^2_{\L(S_\delta)}
			\\
			&\le C|S_\delta|^{-1}\left(\delta^{-1}\|v-\la v\ra_{E_\delta} \|^2_{\L(E_\delta)}+
			\delta\|\nabla v\|_{\L(E_\delta)}^2\right)
			\leq 
			C_1\delta^{2-n}\|\nabla v\|_{\L(E_\delta)}^2.
		\end{align*}

		Finally, by the Sobolev embedding theorem  \cite[Theorem~5.4 and Remark~5.5(6)]{Ad75} 
		the space $\H^2(E)$ is embedded continuously into  
		$\LL^{p}(E)$ provided $p$ satisfies \eqref{p}. Hence, for each $p$ satisfying \eqref{p} there is a constant $C_p>0$ (depending on $E$) such that  
		\begin{gather}\label{Sobolev:D}
			\forall v\in \mathsf{H}^2(E):\ 
			\|v-\la v\ra_{E}\|_{\LL^p(E )}\leq  C_{p} \|v-\la v\ra_{E}\|_{\mathsf{H}^2(E )}.
		\end{gather}
		Again, making the  change of variables, we reduce  \eqref{Sobolev:D} to the  estimate 
		\begin{align}\notag
			\forall v\in\mathsf{H}^2(E_\delta):\quad   \| v-\la v\ra_{E_\delta}\|_{\mathsf{L}^p(E_\delta)}
			&\leq   C_{p}\delta^{n/p} \Bigg(\delta^{-n}\|v-\la v\ra_{E_\delta}\|^2_{\L(E_\delta)}   +\delta^{2-n}\|\nabla v\|^2_{\L(E_\delta)}
			\\
			&\label{Sobolev:De:1}  +\delta^{4-n}\suml_{k,l=1}^n\left\|{\partial^2v\over \partial x_k\partial x_l}\right\|^2_{\L(E_\delta)}
			\Bigg)^{1/2}.
		\end{align}
		Combining \eqref{Sobolev:De:1} and already proven \eqref{Poincare:De}, we arrive at the desired estimates \eqref{Sobolev:De}.
		The lemma is proven.
	\end{proof}
	
	As we can see above, the constant $C_p$ in \eqref{Sobolev:De} is the norm of the
	Sobolev embedding $\mathsf{L}^p(E)\hookrightarrow \H^2(E)$. In the four-dimensional case,
	this embedding takes place for $1\le p<\infty$, but not for $p=\infty$; consequently, $C_p$ 
	blows up as $p\to\infty$. The lemma below provides the estimate on $C_p$ showing that 
	it grows at most linearly in $p$ as $p\to \infty$.

	\begin{lemma}[{\cite[Lemma~4.3]{KP22}}]\label{lemma:c4p}
		In the  case $n=4$ the constant $C_{p}$ in \eqref{Sobolev:De} satisfies
		\begin{gather*}
			C_{p}\leq C p\quad\text{for }n=4,
		\end{gather*}
		where the constant $C>0$ depends only on the set $E$.
	\end{lemma}
	
	\begin{remark}
		The lemma above was proven in \cite{KP22} for $E=[0,1]^n$; for arbitrary $E$ the proof is repeated verbatim. 
		In fact, for $p\ge 4$ the constant $C_{p}$ has the form $C_{p}=C C_p' C''$, where
		\begin{itemize}
			\item $C$ is the norm of the Calderon extension operator \cite[Theorem 4.32]{Ad75}, that is  a linear bounded operator
			$\mathscr{C} : \H^2(E) \to  \H^2(\R^4)$ such that $(\mathscr{C} f )(x) = f (x)$ a.e. in $E$, 
			
			\item $C'_p$ is the norm of the Sobolev embedding $\W^{1,4}(\R^4)\hookrightarrow \mathsf{L}^p(\R^4)$, $p\in [4, \infty)$,
			
			\item $C''$ is the norm of the Sobolev embedding $\mathsf{H}^2(\R^4)\hookrightarrow\W^{1,4}(\R^4)$.
		\end{itemize}
		For the constant $C'_p$ one has the  estimate 
		\begin{gather}\label{Cpprime}
			C'_p\leq Cp , 
		\end{gather}
		where the constant $C$ is absolute;
		the proof of \eqref{Cpprime} is contained in the proof of {\cite[Lemma~4.3]{KP18}}.
	\end{remark}
	
	\begin{lemma}\label{lemma:cylinder}
		Let $S$ be a domain in $\R^{n-1}$, and
		$E_j\ceq\{x=(x',x_n)\in\R^n:\  x'\in S,\ x_n\in (a_j,b_j)\}$, $j=1,2$, where
		$(a_2,b_2)$ is a bounded interval, $(a_1,b_1)\subseteq (a_2,b_2)$. 
		Then one has the estimate
		\begin{gather}\label{lemma:cylinder:est}
			\forall v\in\H^1(E_2):\
			\|v\|_{\L(E_1)}^2 \leq 
			2(b_1-a_1)\left(
			(b_2-a_2)^{-1}\|v\|_{\L(E_2)}^2+(b_2-a_2)\|\nabla v\|_{\L(E_2)}^2\right).
		\end{gather}
	\end{lemma}

	\begin{proof}
		Let $v\in C^\infty(\overline{E_2})$. 
		Let $x'\in S$,  $y\in (a_1,b_1)$, $z\in (a_2,b_2)$.
		One has
		\begin{gather}\label{FTC}
			v(x',y)=v(x',z)+\int_{z}^{y} 
			{\partial v\over \partial x_n}(x',\tau)\d\tau.
		\end{gather}
		Squaring \eqref{FTC} and then applying the Cauchy-Schwarz inequality, we obtain
		\begin{align} \notag
			|v(x',y)|^2 &\leq  
			2|v(x',z)|^2+2|y-z|\left|\int_z^y\left|{\partial v\over \partial x_n}              
			(x',\tau)\right|^2\d\tau\right| \\\label{FTC+}
			&\leq 2|v(x',z)|^2+2(b_2-a_2)
			\int_{a_2}^{b_2}\left|\nabla v(x',\tau)\right|^2\d\tau.
		\end{align}
		Dividing the above inequality by $b_2-a_2$, and then
		integrating it over $x'\in S$, $y\in (a_1,b_1)$ and 
		$z\in (a_2,b_2)$, 
		we arrive at the desired estimate \eqref{lemma:cylinder:est}; 
		by the standard density arguments it holds for each $v\in \H^1(E_2)$. The lemma is proven.
	\end{proof}

	Recall that the function $\GG(t)$ is given by \eqref{Gt}.
	
	\begin{lemma}\label{lemma:ring}
		Let $z\in\R^n$. Let $S$ be a relatively open subset of $\B(1,z)$ 
  (a unit sphere with center at $z$), and let
		the domain $K$ be a union of rays starting at $z$ and crossing $S$ 
		(i.e., $K$ is an unbounded cone with apex at $z$ and base $S$). Let $\varkappa>1$, $\ell> 0$
		be constants;  in the case $n=2$ we  assume $\ell<1$.
		Then for any $a_1,a_2>0$ satisfying 
		\begin{gather}\label{a1a2}
			\varkappa a_1\le a_2\le \ell 
		\end{gather}
		one has the estimate
		\begin{multline}\label{lemma:ring:est}
			\forall v\in \H^1(E_2):\\
			\|v\|^2_{\L(E_1)}\leq
			Ca_1^n\left( a_2^{-n}\|v\|^2_{\L(E_2)}+
			\GG(a_1)\|\nabla v\|^2_{\L(E_2)}\right),\		\text{where } E_j\ceq \B(a_j,z)\cap K.
		\end{multline}   
		The constant $C$ in  \eqref{lemma:ring:est} depends only on the set $S$,  the constants 
		$\ell$, $\varkappa$ and the dimension $n$.
	\end{lemma}
	
	\begin{remark}
		Later on (see the estimates \eqref{JJ:est1}, \eqref{I1:2}, \eqref{alternative}, \eqref{prop:est:3}), we will use the above lemma for $S$ being a half of a unit sphere, that is the corresponding $K$ is a half-space and $E_1$ and $E_2$ are  concentric half-balls lying in $K$.
	\end{remark}
	
	\begin{proof}[Proof of Lemma~\ref{lemma:ring}]
		It is enough to prove \eqref{lemma:ring:est} only for $v\in  {C}^\infty(\overline{E_2})$.
		Let $S_1\ceq \partial   \B(a_1,z)\cap K $. One has $S_1=a_1 S + (1-a_1)z$ and $E_1=a_1 E+(1-a_1)z$, where $E=\B(1,z)\cap K$, 
		whence by Lemma~\ref{lemma:PFTS} (namely, we use the inequality \eqref{trace:De}), we get 
		\begin{align}\label{lm:Fest:est1}
			\|v\|_{\L(E_1)}^2\leq 
			C\left(a_1\|v\|^2_{\L(S_1)}+
			a_1^2\|\nabla v\|^2_{\L(E_1)}\right), 
		\end{align}
		where the constant $C$ depends only on $S$.
		Next, we introduce  the spherical coordinate $(r,\Theta)$ in $\overline{E_2\setminus E_1}$, where $r\in [a_1,a_2]$ stands for the radial coordinate,
		and $\Theta\in S\subset \partial\B(1,z)$ -- for the angular coordinates.
		Let $R\in (a_1,a_2)$, $\Theta\in S$. One has:
		$$
		v(x)=v(y)-\int_{a_1}^{R}{\partial
			v \over\partial r}(\xi(\tau))\d\tau,\text{ where }
		x=(a_1,\Theta),\ y=(R,\Theta),\ \xi(\tau)=(\tau,\Theta).
		$$
		Hence
		\begin{align}\label{lm:Fest:est2}
			|v(x)|^2&\le 2|v(y)|^2+2\left|\int_{a_1}^{R}{\partial
				v(\xi(\tau))\over\partial \tau}\d\tau\right|^2
			\leq
			2|v(y)|^2+
			2M \int_{a_1}^{a_2}\left|\nabla 
			v(\xi(\tau)) \right|^2\tau^{n-1}\d\tau,
		\end{align}
		where
		$M\ceq \ds\int_{a_1}^{a_2}\tau^{1-n}\d\tau.$
		Integrating \eqref{lm:Fest:est2} with respect to $x\in \partial E_1$ and $y\in E_2\setminus\overline{E_1}$, we get
		\begin{equation}\label{lm:Fest:est3}
			N\|v\|^2_{\L(S_1)}
			\leq Ca_1^{n-1}\left( \|v\|^2_{\L(E_2\setminus \overline{E_1})}
			+
			M N\|\nabla v\|^2_{\L(E_2\setminus \overline{E_1})}\right),
		\end{equation}
		where $N\ceq \ds \int_{a_1}^{a_2} R^{n-1}\d R $ and the constant $C$ depends on  $S$.
		One has
		\begin{gather}\label{M:eps}
			M\leq C\cdot \GG(a_1), 
		\end{gather} where the constant $C$ depends only on the dimension $n$; 
		in the case $n=2$ we took into account that 
		$\ln a_1<\ln a_2\le\ln \ell<0$.
		Furthermore, since $a_1<\varkappa a_2$ with $\varkappa\in (0,1)$, one has
		\begin{gather}\label{N:eps}
			N\geq Ca_2^n
		\end{gather}
		(the constant $C$ in \eqref{N:eps} depends on $\varkappa$ and $n$).   
		Evidently, the desired estimate \eqref{lemma:ring:est} follows from
		\eqref{lm:Fest:est1}, \eqref{lm:Fest:est3}--\eqref{N:eps} (in the case $n=2$ we take into account that
		$a_1^2\leq |\ln\ell|^{-1}a_1^2|\ln a_1|$). The lemma is proven.
	\end{proof}

	The last two lemmas will be used later on to get the
	pointwise estimate \eqref{dUke:est:1} on the partial derivatives of the solution $U\ke$ to the problem \eqref{BVP:U:1}--\eqref{BVP:U:4}. In fact, to get \eqref{dUke:est:1} we will use only Lemma~\ref{lemma:V}, while Lemma~\ref{lemma:U}
	plays an auxiliary role.
	
	\begin{lemma}\label{lemma:U}
		Let $\ell>0$, $\varkappa>1$; for $n=2$ we assume that $\ell<1$. 
		Let $a_1,a_2>0$ satisfy \eqref{a1a2}.
		Let $z\in\R^n$, $E_j\ceq \B(a_j,z)$, $j=1,2$, and
		let the function $U:\overline{E_2\setminus E_1}\to \R$ satisfy
		\begin{gather}\label{U:prop}
			\Delta U=0\text{ in }E_2\setminus\overline{E_1},
			\quad
			U=0\text{ on }\partial E_2,
			\quad
			0\le U\le 1\text{ in }\partial E_1.
		\end{gather}      
		Then one has the  pointwise estimate
		\begin{align}
			\label{lemma:dU:est}
			\forall j\in \{1,\dots,n\}:\quad
			\left|{\partial U  \over\partial x_j}(x)\right|\leq 
			C{a_2^{1-n} \over \GG(a_1) - \GG(a_2)},&& x\in 
			E_2\setminus  \overline{\B(\tau^{-1} a_2, z)},  
		\end{align}
		where   $\tau\in (1,\varkappa)$.
		The constant $C$ is \eqref{lemma:dU:est} depends only on $\ell$, $\varkappa$, $\tau$, $n$.
		
	\end{lemma}
	
	\begin{proof}
		First, we establish the estimate
		\begin{gather}
			\label{lemma:U:est}
			0\leq U(x)\leq 
			\wt U(x)\ceq {\GG(|x-z|)- \GG(a_2) \over \GG(a_1) - \GG(a_2)},\ x\in E_2\setminus\overline{E_1}.
		\end{gather}
		By virtue of the maximum principle (see, e.g., \cite[Section~6.4]{Ev98}), one has
		\begin{gather}\label{U>0}
			0\le U\text{ in }E_2\setminus\overline{E_1}.
		\end{gather}
		Next,  we  define the function $W(x)\ceq U(x)-\wt U(x)$. 
		Note that 
		\begin{align*}
			\Delta \wt U  =0\text{ in }\R^n\setminus\{z\},
			\quad
			\wt U  =1\text{ on }\partial E_1,
			\quad
			\wt U  =0\text{ on }\partial E_2,
		\end{align*}
		whence the function $W(x)$ satisfies
		\begin{align*}
			\Delta W =0\text{ in }E_2\setminus\overline{E_1},
			\quad
			W  \le 0\text{ on }\partial E_1,
			\quad
			W= 0\text{ on }\partial E_2.
		\end{align*}
		Again applying the maximum principle we get
		\begin{gather}\label{V<0}
			W(x)\leq 0\text{ for }x\in E_2\setminus\overline{E_1}.
		\end{gather}
		From \eqref{U>0} and \eqref{V<0} we conclude the desired inequalities \eqref{lemma:U:est}.

		Now, we proceed to the proof of \eqref{lemma:dU:est}.
		Let us fix $x\in E_2\setminus \overline{\B(\tau^{-1} a_2, z)}$. We set
		$$  \iota_x \ceq a_2^{-1}|x-z |,\quad \omega_x \ceq \min\left\{\tau^{-1}-\varkappa^{-1};\,1-\iota_x\right\}.$$ 
		One has
		$\iota_x\in (\tau^{-1},1)$, moreover,
		\begin{gather}	\label{l:enclo}
			\B(\omega_x a_2,x)\subset E_2\setminus\overline{E_1}.
		\end{gather}
		Since the function $U$ is harmonic in $E_2\setminus\overline{E_1}$,   its partial derivatives ${\partial U\over\partial x_j}$ are harmonic functions as well. Applying the mean value theorem for harmonic functions (taking into account \eqref{l:enclo}) and  then integrating by parts, we obtain
		\begin{align}\notag
			{\partial U\over\partial x_j}(x)&= 
			{|\B(\omega_x a_2,x)|^{-1}}\int_{\B(\omega_x a_2,x)}{\partial U\over\partial x_j}(y) \d y
			\\\label{mvt} &={|\B(\omega_x a_2,x)|^{-1}}\int_{\partial \B(\omega_x a_2,x)}U(y)\nu_j(y) \d s_y.
		\end{align}
		Here 
		$\d s_y$ is the density of the surface measure  on the sphere ${\partial \B(\omega_x a_2,x)}$ and  
		$\nu_j(y)$ is the $j$-th component of the outward pointing unit normal on $\partial \B(\omega_x a_2,x)$.
		Using the already proven estimate \eqref{lemma:U:est} and taking into account that $|\nu_j(y)|\leq 1$, we deduce from  \eqref{mvt}:
		\begin{align} \label{dU:prelimest}
			\left|{\partial U\over \partial x_j}(x)\right| 
			&\leq C(\omega_x a_2)^{-1}(\GG(a_1)-\GG(a_2))^{-1}\left(\max_{y\in \partial \B(\omega_x a_2,x)}\GG(|y-z|)-\GG(a_2)\right).
		\end{align}
		($C$ above depends only on the dimension $n$).
		The maximum $\max_{y\in \partial \B(\omega_x a_2,x)}\GG(|y-z|)$ is attained at the point on $\partial  \B(\omega_x a_2,x)$  lying on the interval connecting $x$ and the origin $z$, that is
		$\max_{y\in \partial \B(\omega_x a_2,x)}\GG(|y-z|)=\GG((\iota_x-\omega_x)a_2)$. Hence \eqref{dU:prelimest} becomes 
		\begin{align*}
			\left|{\partial U\over \partial x_j}(x)\right| \leq C a_2^{1-n}(\GG(a_1)-\GG(a_2))^{-1}\vartheta_x.
		\end{align*}
		where 
		$$
		\vartheta_x \ceq 
		{1\over \omega_x}\cdot
		\begin{cases}
			(\iota_x-\omega_x)^{2-n}-1 ,& n\ge 3,
			\\-\ln(\iota_x-\omega_x),
			& n=2.
		\end{cases}
		$$
		It is easy to show that the quantity $\vartheta_x$ is   bounded
		on the set $$\left\{ (\iota_x,\omega_x):\ \iota_x\in (\tau^{-1},1),\ \omega_x \ceq \min\left\{\tau^{-1}-\varkappa^{-1},\,1-\iota_x\right\}\right\}$$
		by a constant which depends only on $\varkappa$, $\tau$ and $n$, whence we immediately conclude the desired estimate \eqref{lemma:dU:est}.
		The lemma is proven.  
	\end{proof}

	\begin{lemma}\label{lemma:V}
		Let $\ell>0$, $\varkappa>1$ with $\ell<1$ for $n=2$. 
		Let $a_1,a_2>0$ satisfy \eqref{a1a2}. Let $z=(z',z_n)\in\R^n$.
		We set 
		$$
		F_j\ceq \B(a_j,z)\cap \{x=(x',x_n)\in\R^n:\ x_n>z_n\},\ j=1,2, \quad
		\Gamma_z\ceq  \{x=(x',x_n)\in\R^n:\ x_n=z_n\}.$$
		Let the function  $V:F_2\setminus \overline{F_1}\to \R$ satisfy
		\begin{gather*}
			\Delta V=0\text{ in }F_2\setminus\overline{F_1},
			\quad
			V=0\text{ on }\partial F_2\setminus\Gamma_z,
			\quad
			0\le V\le 1\text{ on }\partial F_1\setminus\Gamma_z,
			\quad
			{\partial V\over \partial \nu}=0\text{ on }\Gamma_z\cap \partial (F_2\setminus F_1),
		\end{gather*}
		where ${\partial\over\partial \nu}$ stands for the outward normal derivative 
		(in fact,   ${\partial\over \partial \nu }=-{\partial\over \partial x_n}$ on 
		$\Gamma_z\cap \partial (F_2\setminus F_1)$).     
		
		Then one has the  pointwise estimate
		\begin{align}
			\label{lemma:dV:est}
			\forall j\in \{1,\dots,n\}:\quad
			\left|{\partial V  \over\partial x_j}(x)\right|\leq 
			C{a_2^{1-n} \over \GG(a_1) - \GG(a_2)},\quad x\in 
			F_2\setminus  \overline{\B(\tau^{-1} a_2, z)},\ \tau\in (1,\varkappa)  .
		\end{align}
		The constant $C$ is \eqref{lemma:dV:est} depends only on $\ell$, $\varkappa$, $\tau$, $n$.

	\end{lemma}

	\begin{proof}
		We denote $E_j\ceq \B(a_j,z)$ and 
		introduce the function $U: (E_2\setminus\overline{E_1})\setminus\Gamma\to \R$,
		$$
		U(x',x_n)=
		\begin{cases}
			V(x',x_n),&x_n>z_n,\\   
			V(x',-x_n),&x_n<z_n,   
		\end{cases}\text{\quad where }
		x=(x',x_n)\in (E_2\setminus\overline{E_1})\setminus\Gamma_z.
		$$
		The traces of $U$ from the top and the bottom of $\Gamma_z\cap \partial (F_2\setminus F_1)$ coincide, furthermore, 
		the traces of the derivatives ${\partial U\over \partial x_n}$ from the top and the bottom of $\Gamma_z\cap \partial (F_2\setminus F_1)$
		are zero. Hence, the function $U$ satisfies \eqref{U:prop}; consequently, the estimates  \eqref{lemma:dU:est} are valid. 
		But, since $U=V$ on $F_2\setminus\overline{F_1}$, we immediately arrive at    \eqref{lemma:dV:est}. The lemma is proven.
	\end{proof}

	\section{Proof of the main results}
	\label{sec:5}	
	
	The proof is  based on the abstract 
	theorems from \cite{P06,AP21,KP21}   presented in Section~\ref{sec:3}.
	To apply them, 
	we will construct suitable operators as in \eqref{JJJJ} 
	satisfying the assumptions \eqref{thA1:0}--\eqref{thA1:3}, \eqref{thA2:1}--\eqref{thA2:4}, \eqref{thA3:1} with some $\delta\e\to 0$, and
	the assumptions \eqref{thA4:3}--\eqref{thA4:4}   
	with some $\nu\e\to 0$, $\wt\nu\e\to 0$  and $0\leq \mu\e\leq C$, 
	$0\leq \wt\mu\e\leq C$.

	\subsection{Preliminaries}\label{subsec:5:1}
	
	In this subsection
	we collect a plenty of auxiliary numbers, sets and functions, which 
	will be used below in the proofs of the main results.
	
	We denote:
	\begin{align*} 
		\wt d\ke^\pm&\ceq
		\begin{cases}
			2 d\ke^\pm ,&n\ge 3 , 
			\\
			(\rho\ke d\ke^\pm)^{1/2} ,&n=2, 
		\end{cases}\\ 
		\wt\eta\ke^\pm&\ceq
		\begin{cases}
			\max\left\{(\wt d\ke^\pm)^{n/2} \rho\ke^{-n/2},\,\wt d\ke^\pm\right\}  ,&n\ge 3,
			\\[2mm]
			\max\left\{\wt d\ke^\pm \rho\ke^{-1},\,\wt d\ke^\pm  |\ln \wt d\ke^\pm |^{1/2}\right\}   , &n= 2.
		\end{cases}
	\end{align*}
	Later on we will need the following  simple inequalities
	\begin{align}\label{eta3}
		\wt\eta\ke^\pm&\le C\eta\ke^\pm,\\\label{eta3+}
		\wt\eta\ke^\pm&\le C\rho\ke^{1/2},
		\\
		\label{Ake:prop:2}
		\rho\ke(\gamma\ke^\pm)^{1/2}&\leq C\eta\ke^\pm,\\\label{Ake:prop:2+}
		\rho\ke(\gamma\ke^\pm)^{1/2}&\le C\rho\ke^{1/2},
		\\
		\label{Ake:prop:1}
		\rho\ke^{1/2}\gamma\ke^\pm&\leq C\eta\ke^\pm, \\\label{Ake:prop:1+}
		\rho\ke^{1/2}\gamma\ke^\pm&\le C\rho\ke^{1/2},
	\end{align}
	which follows easily from \eqref{assump:3},  \eqref{eps0:1+}--\eqref{eps0:4}
	(in particular, to prove the most non-obvious estimates \eqref{Ake:prop:2} and \eqref{Ake:prop:1} we use the equalities
	\begin{gather*}
		\rho\ke^{1/2}\gamma\ke^\pm=\eta\ke^\pm (\gamma\ke^\pm)^{1/2} \ell\ke^\pm,\quad
		\rho\ke(\gamma\ke^\pm)^{1/2}=\eta\ke^\pm (\rho\ke^\pm)^{1/2} \ell\ke^\pm,\quad
	\end{gather*}
	where 
	\begin{gather*}
		\ell\ke^\pm\ceq 
		\begin{cases}
			\left({d\ke^\pm\over\rho\ke}\right)^{n-4\over 2},&n\ge 5,\\
			\left|\ln {d\ke^\pm\over\rho\ke}\right|^{-1},&n=4,\\
			1,&n=3,\\
			\left|1-{\ln \rho\ke }/\ln d\ke^\pm\right|^{1/2},&n=2.
		\end{cases}
	\end{gather*}
	and the (rough) estimate $\ell^\pm\ke\leq C$).

	\begin{remark}\label{rem:lke} 
		The  estimates \eqref{eta3}--\eqref{Ake:prop:1} are  rough. In fact, 		 they all  (except \eqref{Ake:prop:1} for $n=2,3$) hold with $\ll$ instead of $\le $. 
		However, such an improvement  plays no role for our purposes, cf.~Remark~\ref{rem:improve}.
	\end{remark}

	Recall that the sets $O\e$,  $O$, $\Xi\e^\pm$, $\Xi^\pm$, $B\ke^\pm$ are defined by \eqref{Oe}, \eqref{O}, \eqref{Xie}, \eqref{Xi0},  \eqref{BSG}, respectively, and the number $L>0$ is given in \eqref{Gamma:dist}.
	Next, we introduce the   sets
	\allowdisplaybreaks 
	\begin{align}\label{Pke}
		P\ke^\pm&\ceq \B(d\ke,x^\pm\ke)\cap{\Xi}\e^\pm,
		\\\label{wtPke}
		\wt P\ke^\pm&\ceq \B(\wt d\ke,x^\pm\ke)\cap{\Xi}\e^\pm,
		\\\label{whPke}
		\wh P\ke^\pm&\ceq  \B(\rho\ke/2,x^\pm\ke)\cap{\Xi}\e^\pm,
		\\\label{wtRke}
		\wt R\ke^{\pm} &\ceq \left\{x=(x',x_n)\in\R^n: 
		|x'-(x\ke^\pm)'|_{\R^{n-1}}<\rho\ke,\  
		\eps<\pm x_n< \eps+\rho\ke,
		\right\},\\ \label{Rke}
		R\ke^{\pm} &\ceq \left\{x=(x',x_n)\in\R^n: 
		|x'-(x\ke^\pm)'|_{\R^{n-1}}<\rho\ke,\  
		\eps<\pm x_n< \eps+L/4
		\right\},
		\\\label{Oepm}
		O\e^\pm&\ceq O\e\cap \Xi^\pm,
		\\\label{Opm}
		O^\pm& \ceq  O\cap \Xi^\pm,\\
		\label{Pi}
		\Pi\e^\pm&\ceq  \{x=(x',x_n)\in\R^n:\   \eps<\pm x_n<\eps+L/2\}
	\end{align}
	(here $|\cdot|_{\R^{n-1}}$ stands for the Euclidean distance in $\R^{n-1}$).
	Taking into account \eqref{assump:1}, \eqref{eps0:last:1}, \eqref{eps0:1+}, \eqref{eps0:2}, we conclude the following properties
	\begin{gather}\label{enclo1}
		P\ke^\pm\subset \wt P\ke^\pm\subset \wh P\ke^\pm\subset B\ke^\pm\subset \wt R\ke^\pm
		\subset R\ke^\pm, \\\label{enclo2}
		R\ke^\pm\cap R_{j,\eps}^\pm=\emptyset\text{ as } k\not=j,
		\qquad
		\cup_{k\in\Z^{n-1}} R\ke^\pm\subset O^\pm\setminus \overline{O^\pm\e}\subset\Pi\e^\pm\subset\Omega\e^\pm.
	\end{gather} 
	
	Finally, we define the following cut-off functions
	$\phi\ke^\pm$ and $\psi\ke^\pm$:
	\begin{itemize}
		\setlength\itemsep{0.3em}

		\item for $n\geq 3$ the functions $\phi^\pm\ke $ are given by
		\begin{gather*}
			\phi\ke^\pm\ceq \phi\left({|x-x\ke^\pm|\over d\ke^\pm}\right),
		\end{gather*}
		where $\phi:\R\to\R$ is a smooth function
		satisfying
		\begin{gather}\label{phi:prop}
			0\le\phi\le 1,\quad \phi(x)=1\text{ as }x\le 1,\quad \phi(x)=0\text{ as }x\ge 2,
		\end{gather}
		
		\item for $n=2$ the  functions $\phi\ke^\pm$ are given by
		\begin{equation*}
			\phi\ke^\pm(x)\ceq  \
			\begin{cases}
				1,&|x-x\ke^\pm|\le d\ke,
				\\
				\dfrac{\ln|x-x\ke^\pm|-\ln \sqrt{\rho\ke d\ke^\pm}}{\ln d\ke^\pm -\ln \sqrt{\rho\ke d\ke^\pm}},&
				|x-x\ke^\pm|\in (d\ke^\pm,\sqrt{\rho\ke d\ke^\pm}),\\
				0,& |x-x\ke^\pm|\geq \sqrt{\rho\ke d\ke^\pm},
			\end{cases}
		\end{equation*}
		
		\item $\psi\ke^\pm\ceq \phi\left({2|x-x\ke^\pm|\over \rho\ke}\right)$,
		where again $\phi:\R\to\R$ is a smooth function
		satisfying \eqref{phi:prop}.

	\end{itemize}
	The functions obey (in particular) the following properties:
	\begin{gather}
		\label{phipsi:1}
		\phi^\pm\ke\restriction_{ P^\pm\ke}=1,\quad 	\supp\phi^\pm\ke\cap \Xi\e^\pm\subset \overline{\wt P\ke^\pm},
		\\\label{phipsi:2}
		\psi^\pm\ke\restriction_{\wh P^\pm\ke}=1,\quad\supp\psi^\pm\ke \cap \Xi^\pm\e\subset \overline{B\ke^\pm},
		\\
		\label{phipsi:3}
		(\supp(\psi\ke^+-1)\cup\supp(\nabla \psi\ke^\pm)\cup\supp(\Delta \psi\ke^\pm))\cap \Xi\e^\pm \subset 
		\overline{B\ke^\pm\setminus\wh P\ke^\pm},
		\\
		\label{phipsi:4}
		|\phi^\pm\ke(x)|\le1,\\ 
		\label{phipsi:5}
		|\psi\ke^\pm (x)|\leq 1,
		\\ 
		\label{phipsi:6}
		|\nabla \psi\ke^\pm (x)|\leq C\rho\ke^{-1},\quad 
		|\Delta \psi\ke^\pm (x)|\leq C\rho\ke^{-2},
		\\\label{phipsi:7}
		{\partial\psi\ke^\pm\over\partial x_n}=0\text{ for }x\in\Gamma\ke^\pm,
		\\
		\label{phipsi:8}
		\begin{array}{cl}
			\Delta((1-U\ke) \psi\ke^+)=-2\nabla U\ke \cdot \nabla\psi\ke^++(1-U\ke) \Delta \psi\ke^+,&x\in B\ke^+ ,\\[1mm]
			\Delta( U\ke \psi\ke^-)= 2\nabla U\ke \cdot \nabla\psi\ke^-+ U\ke  \Delta \psi\ke^-,&x\in B\ke^-
		\end{array}
	\end{gather}
	(recall that $U\ke$ is a solution to \eqref{BVP:U:1}--\eqref{BVP:U:4}).
	
	\subsection{The operators   $\J\e^1$ and $\wt\J\e^1$}
	
	We denote
	$$
	\HS\e\ceq \L(\Omega\e),\quad \HS\ceq \L(\Omega).
	$$
	Next, using the forms $\a\e$ \eqref{ae} and $\a$ \eqref{aga} and the associated with them operators $\A\e$ and $\A$, 
	we introduce the spaces $\HS\e^1$, $\HS^1$, $\HS^2$ (cf. \eqref{H1spaces}, \eqref{H2space}):
	\begin{gather*} 
		\begin{array}{lll}
			\HS^1\e=\H^1(\Omega\e),
			&
			\ds\|u\|^2_{\HS^1\e}
			&=\|u \|^2_{\H^1(\Omega\e)},
			\\[1.5ex]
			\HS^1=\H^1(\Omega\setminus\Gamma),
			&
			\ds\|f\|^2_{\HS^1}&=
			\|\nabla f \|^2_{\L(\Omega^+)}+
			\|\nabla f \|^2_{\L(\Omega^-)} 
			+\|{\mu}^{1/2}[f]\|^2_{\L(\Gamma)}+\|f \|^2_{\HS},
			\\[1.5ex]
			\HS^2=\dom(\A),
			&\ds
			\|f\|_{\HS^2}^2
			&
			=\|-\Delta f  +f\|^2_{\L(\Omega^+)}+\|-\Delta f  +f  \|^2_{\L(\Omega^-)}
		\end{array}
	\end{gather*} 
	(recall that $[f]$ stands for the jump of $f$ across $\Gamma$, see \eqref{jump}).

	We have already introduced the operator   $\J\e:\HS\to\HS\e$   and $\wt\J\e:\HS\e\to\HS$, see \eqref{Je:wtJe}.
	Next, we define the operator  
	$\J\e^1:\HS^1\to\HS^1\e $ by
	\begin{gather*}
		(\J\e^1 f)(x)=
		\begin{cases}
			f(x), & x\in  \Omega\e \setminus \left(\cup_{k\in\Z^{n-1}}G\ke\right),\\[2mm]
			\ds f(x)+  (\la f \ra_{B\ke^+}-f(x) ) \phi\ke^+(x)
			\\
			\qquad\,+\ds   (\la f \ra_{B\ke^-} - \la f \ra_{B\ke^+} )
			(1-U\ke(x))\psi\ke^+(x),&x\in B\ke^+,\ k\in\Z^{n-1},
			\\[2mm]
			\ds f(x)+  (\la f \ra_{B\ke^-}-f(x) ) \phi\ke^-(x)
			\\
			\qquad\,+\ds   (\la f \ra_{B\ke^+} - \la f \ra_{B\ke^-} )U\ke (x)\psi\ke^-(x),&x\in B\ke^-,\ k\in\Z^{n-1},
			\\[2mm]
			\la f \ra_{B\ke^+} U\ke (x) + \la f \ra_{B\ke^-} (1-U\ke(x)),& x\in T\ke,\ k\in\Z^{n-1}
		\end{cases}
	\end{gather*}
	(recall that the sets $G\ke$ are defined in \eqref{BSG}).
	Using \eqref{phipsi:1}--\eqref{phipsi:2}, one can easily show that
	the function $\J\e^1 f$ indeed belongs to $\H^1(\Omega\e)$.

	Recall that the sets   $\Pi\e^\pm$ are defined by  \eqref{Pi}.
	It is well-known  (see, e.g., \cite[Chapt.~IV]{Ad75})  that there exist  
	linear bounded extension operators $\mathscr{E}\e^\pm:\H^1(\Pi^\pm\e)\to \H^1(\R^n)$, i.e. the bounded operators satisfying  
	\begin{gather}
		\label{Eprop}
		\forall u\in \H^1(\Pi\e^\pm):\quad
		(\mathscr{E}\e^\pm u)\restriction_{\Pi\e^\pm}=u.
	\end{gather}
	Since $\Pi\e^\pm$ are congruent to the sets $O^\pm$ (recall that they defined by \eqref{Opm}), we can obviously choose these operators in such a way that their norms
	\begin{gather}
		\label{Eprop:moreover}
		\|\mathscr{E}^\pm\e  \|_{\H^1(\Pi\e^\pm)\to \H^1(\R^n)} \text{ are independent of }\eps.
	\end{gather}
	Then  we define  
	$\wt\J\e^1:\H^1(\Omega\e)\to\H^1(\Omega\setminus\Gamma)$ as follows,
	\begin{gather*}
		(\wt\J\e^1 u)(x)\ceq
		\begin{cases}
			u(x),&x\in\Omega\e^+\cup\Omega\e^-,\\
			(\mathscr{E}^+\e u\e^+)(x),&x\in  O\e^+,\\
			(\mathscr{E}^-\e u\e^-)(x),&x\in O\e^-,
		\end{cases}	
		\text{\quad where }u\e^\pm\ceq u\restr_{\Pi\e^\pm}
	\end{gather*}
	(recall that $O^\pm\e$ are defined by \eqref{Oepm} and we have $\Pi\e^\pm\subset \Omega\e^\pm$).
	Using \eqref{Eprop}--\eqref{Eprop:moreover},  we get
	\begin{gather}\label{wtJprop}
		(\wt\J\e^1 u)\restr_{\Omega\e^+\cup\Omega\e^-}=u,\quad
		\|\wt\J\e^1 u\|_{\H^1(\Omega^\pm)}\leq C\|u\|_{\H^1(\Omega^\pm\e)},\quad
		\forall u\in \H^1(\Omega\e).
	\end{gather}
	
	\subsection{Verification of  the assumptions \eqref{thA1:0}--\eqref{thA1:3}}
	
	Evidently, one has
	\begin{align}\label{JJ}  
		(u,\J\e f)_{\HS\e} =(\wt\J\e u,f)_{\HS},\quad \forall f\in\HS,\, u\in\HS\e.
	\end{align}
	Thus, \eqref{thA1:0} is fulfilled with any $\delta\e\ge 0$.  
	
	Before to proceed to the verification of the remaining assumptions \eqref{thA1:1}--\eqref{thA1:3}, we establish an estimate on the volume  of the passage  $T\ke$. Recall that $\gamma\ke^\pm$, $\zeta\e$ are given in \eqref{assump:3}, \eqref{non:con}.
	
	\begin{lemma}\label{lemma:Tke:est}
		One has  
		\begin{gather}\label{Tke:est}
			\forall\eps\in(0,\eps_0]\ \forall k\in\Z^{n-1}:\quad |T\ke|\leq C\zeta\e^2\rho\ke^{n-1}\max\left\{\ga\ke^+,\ga\ke^-\right\}.
		\end{gather}	
	\end{lemma}	
	
	\begin{proof}
		Let us fix $k\in\Z^{n-1}$ and introduce the function $w\e\in\H^1(\Omega\e)$ via
		\begin{gather*}
			w\e (x)
			\ceq 
			\begin{cases}
				1,&x\in T\ke\cup P\ke^+\cup P\ke^- ,
				\\[2mm]
				\dfrac{
					\GG(|x-x\ke^\pm|)-\GG(\rho\ke^\pm)}{\GG(d\ke^\pm)-\GG(\rho\ke)},&
				x\in B\ke^\pm\setminus P\ke^\pm,\\
				0,& x\in \Omega\e\setminus (B\ke^+\cup B\ke^-\cup T\ke).
			\end{cases}
		\end{gather*} 
		(recall that $P\ke^\pm$ are defined by \eqref{Pke}).
		One has:
		\begin{align}\label{we:est1}
			\|w\e\|^2_{\L(\cup_{j\in\Z^{n-1}}T_{j,\eps})}&=|T\ke|,\\
			\|w\e\|^2_{\L(\Omega\e)}&=|T\ke|+\|w\e\|^2_{B\ke^+}+\|w\e\|^2_{B\ke^+},\\
			\label{we:est2} 
			\|\nabla w\e\|^2_{\L(\Omega\e)}&=
			\|\nabla w\e\|^2_{\L(B\ke^+\setminus\overline{P\ke^+})}+\|\nabla w\e\|^2_{\L(B\ke^-\setminus\overline{P\ke^-})}.
		\end{align}
		Simple calculations yield
		\begin{align}\label{we:est3}
			\|\nabla w\e\|^2_{\L(B\ke^\pm\setminus\overline{P\ke^\pm})}
			\leq  C
			\big((\GG(d\ke^\pm)-\GG(\rho\ke))^{-1}  
			\le C_1
			(\GG(d\ke^\pm))^{-1}
		\end{align}
		(to get the last estimate we take into account \eqref{eps0:2}--\eqref{eps0:4}).
		Furthermore, by virtue of Lemma~\ref{lemma:PFTS} (namely, we use the inequality  \eqref{Friedrichs:De} for 
		$E_\delta=B\ke^\pm$, $S_\delta=S\ke^\pm$) and taking into account that $\rho\ke\leq C$ (see \eqref{eps0:1+}), we obtain:
		\begin{gather}\label{we:Friedrichs}
			\|w\e\|^2_{\L(B\ke^\pm)}\leq C\rho\ke^2\|\nabla w\e\|^2_{\L(B\ke^\pm)}\leq C_1  \|\nabla w\e\|^2_{\L(B\ke^\pm)}=\|\nabla w\e\|^2_{\L(B\ke^\pm\setminus\overline{P\ke^\pm})}.
		\end{gather}
		Using \eqref{non:con}, \eqref{we:est1}--\eqref{we:Friedrichs}, we obtain
		\begin{gather*}
			|T\ke|\leq \zeta\e^2|T\ke|  +C\zeta\e^2\left( 
			(\GG(d\ke^+))^{-1}+(\GG(d\ke^-))^{-1}\right) ,
		\end{gather*}
		whence, taking into account \eqref{eps0:3}
		and the definition of $\gamma\ke^\pm$, we arrive at the desired estimate 
		\eqref{Tke:est}. The lemma is proven.
	\end{proof}

	Let us now proceed to the  verification of the assumption \eqref{thA1:1}.

	\begin{lemma}\label{lemma:A1}
		One has
		\begin{gather}
			\label{lemma:A1:est}
			\forall f\in \HS^1:\quad 
			\|\J\e^1 f - \J\e f \|_{\HS\e}\leq  
			C  
			\max \left\{\chi\e,\, \zeta\e,\,  \eps^{1/2}  \right\}\|f\|_{\HS^1}.
		\end{gather}
	\end{lemma}

	\begin{proof} 
		
		Let $f\in \HS^1=\H^1(\Omega\setminus\Gamma)$. 
		Taking into account   \eqref{phipsi:1},   \eqref{phipsi:2},  \eqref{phipsi:4}, \eqref{phipsi:5},
		we get
		\begin{align}\notag
			\|\J\e^1 f - \J\e f \|^2_{\HS\e} 
			&\leq 
			2\suml_{k\in\Z^{n-1}}\| \la f \ra_{B\ke^+}-f \|^2_{\L(\wt P\ke^+)}
			+4\suml_{k\in\Z^{n-1}}(|\la f \ra_{B\ke^+}|^2+|\la f \ra_{B\ke^-}|^2)\|U\ke-1\|^2_{\L(B\ke^+)}
			\\\notag
			&+
			2\suml_{k\in\Z^{n-1}}\| \la f \ra_{B\ke^-}-f \|^2_{\L(\wt P\ke^-)}
			+4\suml_{k\in\Z^{n-1}}(|\la f \ra_{B\ke^+}|^2+|\la f \ra_{B\ke^-}|^2)\|U\ke \|^2_{\L(B\ke^-)}
			\\\notag
			&+ 
			3\suml_{k\in\Z^{n-1}}|\la f \ra_{B\ke^+}|^2\|U\ke\|^2_{\L(T\ke)}+
			3\suml_{k\in\Z^{n-1}}|\la f \ra_{B\ke^-}|^2\|U\ke-1\|^2_{\L(T\ke)}
			\\
			\label{JJ:start}
			&+3\suml_{k\in\Z^{n-1}}\|f\|^2_{\L(T\ke)}
		\end{align}
		(recall that the sets $\wt P\ke^\pm$ are defined by \eqref{wtPke}).

		By Lemma~\ref{lemma:ring}, 
		which we apply for $z=x\ke^\pm$, $E_1=\wt P\ke^\pm$, $E_2=B^\pm\ke$, $\varkappa={2}$, $\ell=\rho\e$, $a_1=\wt d\ke^\pm$, $a_2=\rho\ke$, we obtain
		\begin{align}\notag
			\suml_{k\in\Z^{n-1}}\| \la f\ra_{B\ke^\pm} -f \|^2_{\L(\wt P\ke^\pm)} & \leq
			C \suml_{k\in\Z^{n-1}}(\wt\eta\ke)^2\|\la f\ra_{B\ke^\pm} -f\|^2_{\H^1(B\ke^\pm)}
			\\ 
			&\leq 
			C_1 \sup_{k\in\Z^{n-1}}(\wt\eta\ke)^2 \suml_{k\in\Z^{n-1}}\| f\|^2_{\H^1(B\ke^\pm)}\leq
			C_1 \sup_{k\in\Z^{n-1}}(\wt\eta\ke)^2 \|f\|^2_{\HS^1}.
			\label{JJ:est1} 
		\end{align}
		(here we also use  the Cauchy-Schwarz inequality 
		$\|\la f\ra_{B\ke^\pm}\|^2_{\H^1(B\ke^\pm)}=|\la f\ra_{B\ke^\pm}|^2|B\ke^\pm|\leq \|f\|^2_{\L(B\ke^\pm)}$).

		To estimate the second and the fourth terms in the right-hand-side of \eqref{JJ:start}
		we introduce auxiliary functions $V^\pm\ke\in \H^1(G\ke)$ via
		\begin{align*}
			V\ke^+ (x)
			&\ceq 
			\begin{cases}
				\dfrac{\GG(|x-x\ke^+|)-\GG(d\ke^+)}{\GG(\rho\ke)-\GG(d\ke^+)},&
				x\in B\ke^+\setminus P\ke^+,\\
				0,& x\in P\ke^+\cup B\ke^-\cup T\ke,
			\end{cases}
			\\
			V\ke^- (x)
			&\ceq 
			\begin{cases}
				\dfrac{\GG(|x-x\ke^-|)-\GG(\rho\ke)}{\GG(d\ke^-)-\GG(\rho\ke)},&
				x\in B\ke^-\setminus P\ke^-,\\
				1,& x\in B\ke^+\cup T\ke\cup P\ke^-.
			\end{cases}
		\end{align*}  
		One has $V\ke^\pm\restr_{S\ke^+}=1$ and $V\ke^\pm\restr_{S\ke^-}=0$, whence, due to
		\eqref{CUU}--\eqref{Cke}, we get
		\begin{gather}\label{UV}
			\capty\ke=\|\nabla U\ke \|^2_{\L(G\ke)}\leq \|\nabla V^\pm\ke\|^2_{\L(G\ke)}.
		\end{gather}
		Direct computations yield (taking into account \eqref{eps0:2}--\eqref{eps0:4})
		\begin{gather}\label{nablaV}
			\|\nabla V^\pm\ke\|^2_{\L(G\ke)} \leq 
			C((\GG(d\ke^\pm))^{-1}-(\GG(\rho\ke))^{-1})^{-1}\leq
			C_1(\GG(d\ke^\pm))^{-1}= C_1\gamma\ke^\pm \rho\ke^{n-1}.
		\end{gather}
		Since $U\ke-1=0$ on $S\ke^+$ (resp., $U\ke=0$ on $S\ke^-$), 
		we obtain by virtue of Lemma~\ref{lemma:PFTS} (namely, we use the inequality  \eqref{Friedrichs:De} for $z=x\ke^\pm$,
		$E_\delta=B\ke^\pm$, $S_\delta=S\ke^\pm$):
		\begin{gather}\label{U:Poincare}
			\|U\ke -1 \|^2_{\L(B\ke^+)}\leq C \rho\ke^2\|\nabla U\ke \|^2_{\L(B\ke^+)},\quad 
			\|U\ke \|^2_{\L(B\ke^-)}\leq C \rho\ke^2\|\nabla U\ke \|^2_{\L(B\ke^-)}.
		\end{gather}
		Combining \eqref{UV}--\eqref{U:Poincare}, we arrive at the estimates
		\begin{gather}\label{UL2}
			\|U\ke -1 \|^2_{\L(B\ke^+)}+ 
			\|U\ke \|^2_{\L(B\ke^-)}\leq C\rho\ke^{n+1}\min\left\{\ga\ke^+,\,  \ga\ke^+\right\}.
		\end{gather}
		Recall that the sets $\wt R\ke^\pm,\,R\ke^\pm$ are defined in \eqref{wtRke}--\eqref{Rke}.
		Using   $B\ke^\pm\subset\wt R\ke^\pm$ (see \eqref{enclo1})
		and then applying Lemma~\ref{lemma:cylinder} (which we apply for $E_1=\wt R\ke^\pm$ and $E_2=R\ke^\pm$),
		we obtain:
		\begin{gather}
			\label{Rke:est}
			\|f\|_{\L(B\ke^\pm)}^2\leq   \|f\|_{\L(\wt R\ke^\pm)}^2\leq C \rho\ke  \|f\|_{\H^1(R\ke^\pm)}^2.
		\end{gather}
		Combining the Cauchy-Schwarz inequality and  \eqref{Rke:est}, we get
		\begin{gather}\label{abs:fke}
			|\la f \ra_{B\ke^\pm}|^2
			\leq 
			\rho\ke^{-n}\|f\|^2_{\L(B\ke^\pm)}\leq
			C \rho\ke^{1-n} \|f\|^2_{\H^1(R\ke^\pm)}.
		\end{gather}
		From \eqref{UL2}--\eqref{abs:fke} (taking into account \eqref{enclo2}), we deduce the estimate for the second and the fourth         terms in the right-hand-side of \eqref{JJ:start}:
		\begin{multline}\label{JJ:est2} 
			\ds\suml_{k\in\Z^{n-1}}\left(|\la f \ra_{B\ke^-}|^2+|\la f \ra_{B\ke^+}|^2\right) 
			\left(\|U\ke-1\|^2_{\L(B\ke^+)}+\|U\ke\|^2_{\L(B\ke^-)}\right)
			\\\leq
			C\min\left\{\sup_{k\in\Z^{n-1}}(\ga\ke^- \rho\ke ^2),\,\sup_{k\in\Z^{n-1}}(\ga\ke^+ \rho\ke ^2)\right\}\|f\|^2_{\HS^1}.
		\end{multline}
		
		The maximum principle with  the Hopf lemma  \cite[Section~6.4]{Ev98} and  
		\eqref{BVP:U:1}--\eqref{BVP:U:4}
		imply 
		\begin{gather}\label{U01}
			0\leq U\ke(x) \le 1,\ x\in G\ke.
		\end{gather}
		From \eqref{Tke:est}, \eqref{abs:fke}, \eqref{U01}, 
		we deduce the estimate
		\begin{align}  \notag
			\suml_{k\in\Z^{n-1}}|\la f \ra_{B\ke^-}|^2\|U\ke-1\|^2_{\L(T\ke)}
			&+
			\suml_{k\in\Z^{n-1}}|\la f \ra_{B\ke^+}|^2\|U\ke\|^2_{\L(T\ke)}
			\\\notag
			& \leq 
			C \suml_{k\in\Z^{n-1}}\rho\ke^{1-n}|T\ke| \|f\|^2_{\H^1(R\ke^+\cup R\ke^-)}\\ 
			& \leq C \zeta\e^2\gamma\e \|f\|^2_{\HS^1}.
			\label{JJ:est3}
		\end{align}
		
		Finally, by virtue of Lemma~\ref{lemma:cylinder} (which we apply for $E_1=O^\pm\e$ and $E_2=O^\pm$), we get  
		\begin{gather}\label{Se:est}
			\|f\|^2_{\L(O\e^\pm)}\leq 
			C\eps\|f\|^2_{\H^1(O^\pm)},
		\end{gather} 
		whence, we conclude	
		\begin{gather}\label{JJ:est4} 
			\suml_{k\in\Z^{n-1}}\|f\|^2_{\L(T\ke)}\leq
			\|f\|^2_{\L(O\e^-)}+\|f\|^2_{\L(O\e^+)}
			\leq C  \eps  \left(\|f\|^2_{\H^1(O^+)}+\|f\|^2_{\H^1(O^-)}\right)
			\leq C  \eps \|f\|^2_{\HS^1}.
		\end{gather}

		Combining \eqref{JJ:start}, \eqref{JJ:est1}, \eqref{JJ:est2}, \eqref{JJ:est3}, \eqref{JJ:est4}
		and taking into account \eqref{assump:3}, \eqref{eta3}--\eqref{Ake:prop:2+}, 
		we arrive at the desired estimate \eqref{lemma:A1:est}. The lemma is proven.
		
	\end{proof}
	
	\begin{remark}\label{rem:improve}
		The estimate \eqref{lemma:A1:est}
		is rough; in fact, it holds with   $o(\chi\e)$ instead of $\chi\e$.
		We worsened the final estimate 
		on the last step of the proof, when we use
		the inequalities \eqref{eta3}--\eqref{Ake:prop:2+} (see Remark~\ref{rem:lke}); also the estimate \eqref{JJ:est1} can be improved by using the Poincar\'e inequality \eqref{Poincare:De}.
		However, at the end of the day this improvement plays no
		role 
		due to the fact that  
		$\chi\e$ pops up later on in the   estimate \eqref{lemma:A3:est}, where it cannot be omitted
		(at least, so far we do not see how to do this).
		Similarly, the estimate \eqref{lemma:A1:est} can be improved by replacing $\zeta\e$ by 
		$\zeta\e\gamma\e^{1/2}$ (see \eqref{JJ:est3}), however this improvement also plays no role, since
		$\zeta\e$ appears later on in the  estimate \eqref{lemma:A2:est}. 
	\end{remark}

	Next, we verify the assumption \eqref{thA1:2}.
	
	\begin{lemma}\label{lemma:A2}
		One has
		\begin{gather}\label{lemma:A2:est}
			\forall u\in \HS^1\e:\quad \|\wt\J\e^1 u - \wt\J\e u \|_{\HS}\leq C 
			\max\left\{  \eps^{1/2},\, \zeta\e\right\}\|u\|_{ \HS\e^1}. 
		\end{gather}
	\end{lemma}
	
	\begin{proof} 
		Let $u\in \H^1(\Omega\e)$.  
		Using \eqref{non:con}, \eqref{wtJprop}, \eqref{Se:est}, we get
		\begin{align*}
			\|\wt\J\e^1 u - \wt\J\e u \|_{\HS}^2&=
			\sum_{k\in\Z^{n-1}}\|\wt\J\e^1 u -   u \|_{\L(T\ke)}^2
			\leq 
			2\sum_{k\in\Z^{n-1}}\|\wt\J\e^1 u\|^2_{\L(T\ke)}+2\sum_{k\in\Z^{n-1}}\|u\|^2_{\L(T\ke)}
			\\
			&\le 2\|\wt\J\e^1u\|^2_{\L( O\e^+)}+2\|\wt\J\e^1 u\|^2_{\L(O\e^-)}+
			2\sum_{k\in\Z^{n-1}}\|u\|^2_{\L(T\ke)}
			\\
			&\leq
			C\left(\eps\|\wt\J\e^1u\|^2_{\H^1(O^+)}+ \eps\|\wt\J\e^1u\|^2_{\H^1(O^-)}
			+ \zeta\e^2\|u\|^2_{\H^1(\Omega\e)}\right)
			\\
			&\leq C_1\max\left\{ {\eps},\, \zeta\e^2\right\}\|u\|^2_{\HS^1\e}.
		\end{align*}
		The lemma is proven.    
	\end{proof}

	We are now proceeding to the verification of   \eqref{thA1:3}.
	For $f\in \HS^2 ,\ u\in \HS^1\e $ we have
	\begin{gather}\label{aa}
		\a\e[u,\J^1\e f]-\a [\wt\J^{1}\e u,f]=
		I^{1,-}\e+I^{1,+}\e+I^{2,-}\e+I^{2,-}\e+I^{3}\e+I^4\e+I^5\e,
	\end{gather}
	where
	\begin{align*}
		I^{1,\pm}\e &\ceq\sum_{k\in\Z^{n-1}}\left(\nabla u, {\nabla((\la f\ra_{B\ke^\pm}-f)\phi\ke^\pm)}\right)_{\L(B\ke^\pm)},
		\\
		I^{2,+}\e &\ceq\sum_{k\in\Z^{n-1}}\left(\nabla u,{\nabla((\la f\ra_{B\ke^-}-\la f\ra_{B\ke^+})(1-U\ke)\psi\ke^+)}\right)_{\L(B\ke^+)},
		\\
		I^{2,-}\e &\ceq\sum_{k\in\Z^{n-1}}\left(\nabla u,{\nabla((\la f\ra_{B\ke^+}-\la f\ra_{B\ke^-})U\ke\psi\ke^-)}\right)_{\L(B\ke^-)},
		\\
		I^3\e &\ceq \sum_{k\in\Z^{n-1}}\left(\nabla u,\nabla( \la f\ra_{B\ke^+}U\ke  + \la f\ra_{B\ke^-}(1-U\ke) )\right)_{\L(T\ke)}.
		\\
		I^4\e &\ceq - \left(\nabla (\wt\J\e^1 u), {\nabla f}\right)_{\L(O\e^+)}
		- \left(\nabla (\wt\J\e^1 u), {\nabla f}\right)_{\L(O\e^-)},
		\\
		I^5\e &\ceq -({\mu}[\wt\J\e^1u],[f])_{\L(\Gamma)}. 
	\end{align*}
	Below we analyse each of the above terms $I\e^{k,\pm}$, $k=1,2$ and $I\e^{k}$, $k=3,4,5$.

	\begin{lemma}\label{lemma:I1:est} 
		One has:
		\begin{gather}\label{I1:est} 
			|I^{1,\pm}\e|\leq C\chi\e \|f\|_{\HS^2}\|u\|_{\HS^1\e} .
		\end{gather}
	\end{lemma}

	\begin{proof}Recall that we assume the function ${\mu}$  to be in $C^{1}(\Gamma)\cap\W^{1,\infty}(\Gamma)$.
		This assumptions implies
		\begin{gather}\label{H2est}
			\forall f\in\dom(\A):\quad f\restriction_{O^\pm} \in \H^2(O^\pm)
			\text{\quad and \quad} \|f  \|_{\mathsf{H}^2(O^\pm)}\leq
			C \|\A f +f \| _{\L(\Omega^\pm)} ,
		\end{gather}
		where the constant $C$ is independent of $f$.
		For the proof see, e.g., { \cite[Lemma~3.7]{Kh23} }.

		Taking into account \eqref{phipsi:1} and \eqref{phipsi:4}, 
		we obtain using the Cauchy-Schwarz inequality:
		\begin{gather}  
			|I_{\eps}^{1,\pm}|\leq 
			\left\{ 
			\Bigl( {\suml_{k\in\Z^{n-1}}\|\nabla f  \|_{\L(\wt P\ke^\pm)}^2\Bigr)^{1/2}}  + \Bigl(
			{\suml_{k\in\Z^{n-1}}\|( f  - \la f \ra_{B^\pm\ke})\nabla \phi\ke^\pm\|_{\L(\wt P\ke^\pm)}^2\Bigr)^{1/2}} \right\}\|u\|_{\HS\e^1}.\label{I1:1}
		\end{gather} 
		
		By Lemma~\ref{lemma:ring}, 
		which we apply for $z=x\ke^\pm$, $E_1=\wt P\ke^\pm$, $E_2=B^\pm\ke$, $\varkappa={2}$, $\ell=\rho\e$, $a_1=\wt d\ke^\pm$, $a_2=\rho\ke$), one obtains
		\begin{gather}\label{I1:2} 
			\|\nabla f \|_{\L(\wt P\ke^\pm)}^2\leq  
			C(\wt\eta\ke^\pm)^2\|f \| _{\H^2(B\ke^\pm)}^2 .
		\end{gather}
		Using  \eqref{eta3}, \eqref{H2est}, \eqref{I1:2},  we arrive at the estimate
		\begin{gather}\label{I1:3}
			\left(\suml_{k\in\Z^{n-1}}\|\nabla f \|^2_{\L(\wt P\ke^\pm)}\right)^{1/2}\leq 
			C\sup_{k\in\Z^{n-1}}\wt\eta\ke\|f \|_{\H^2(O^\pm)}\leq
                C_1\sup_{k\in\Z^{n-1}}\wt\eta\ke\|f \|_{\HS^2 }
		\end{gather}
		
		To estimate the second term in the right-hand-side of \eqref{I1:1},
		we define the numbers $\mathbf{p},\, \mathbf{q}$:
		\begin{align*} 
			&\mathbf{p}=\frac {2n}{n-4} \text{\, if }n\geq 5,\quad
			&&\mathbf{p}=2|\ln d\ke^\pm -\ln \rho\ke |\text{\, if }n=4,\quad
			&&\mathbf{p}=\infty\text{\, if }n=2,3,\\ 
			&\mathbf{q}={n\over 2}\text{\, if }n\geq 5,\quad 
			&&\mathbf{q}={2\over 1-|\ln d\ke^\pm -\ln \rho\ke |^{-1}}\text{\, if }n=4,\quad 
			&&\mathbf{q}=2\text{\, if }n=2,3 
		\end{align*}
		(due to \eqref{eps0:2},   we have $|\ln d\ke^\pm -\ln \rho\ke |\ge  \ln 4>1$, hence  $\mathbf{q}$ is well-defined for $n=4$).
		One has $\mathbf{p},\mathbf{q}\in[2,\infty]$ and $\mathbf{p}^{-1}+ \mathbf{q}^{-1}=2^{-1}$. 
		Using $\wt P\ke^\pm\subset B\ke^\pm$ and   the H\"older inequality, we obtain
		\begin{align} 
			\|(f  -\la f \ra_{B^\pm\ke})\nabla  \phi\ke^\pm\|^2_{\L(\wt P\ke^\pm)}\leq 
			\|f  - \la f \ra_{B^\pm\ke}\|^2_{\LL^\mathbf{p}(B\ke^\pm)} \|\nabla  \phi\ke^\pm\|^2_{\LL^\mathbf{q}(B\ke^\pm)}.\label{Hoelder}
		\end{align}  
		By virtue of Lemma~\ref{lemma:PFTS} (namely, we use \eqref{Sobolev:De} for $E_\delta=B\ke^\pm$ and Lemma \ref{lemma:c4p}, 
		we get  
		\begin{gather}\label{pmax+}
			\|f -\la f \ra_{B^\pm\ke}\|_{\LL^\mathbf{p}(B^\pm\ke)}\leq C
			\|f  \|_{\mathsf{H}^2(B^\pm\ke)}
			\begin{cases}
				\rho\ke^{-1}&n\geq 5,\\
				|\ln d\ke^\pm - \ln \rho\ke |\cdot\rho\ke^{2|\ln d\ke^\pm - \ln \rho\ke |^{-1}-1}&n= 4,\\
				\rho\ke^{-1/2}&n= 3\\
				1&n= 2.
			\end{cases}
		\end{gather} 
		Furthermore, simple calculations yield
		\begin{gather}
			\label{wtphi-est}
			\|\nabla \phi\ke^\pm\|_{\LL^\mathbf{q}(B\ke^\pm)}\leq 
			C\begin{cases}
				d\ke^\pm,&n\ge 5,\\
				(d\ke^\pm)^{1-2|\ln {d\ke^\pm}-\ln{\rho\ke}|^{-1}},&n=4,\\
				(d\ke^\pm)^{1/2},&n=3,\\
				|\ln {d\ke^\pm}-\ln{\rho\ke}|^{-1/2},&n=2.
			\end{cases} 
		\end{gather} 
		Using \eqref{Hoelder}--\eqref{wtphi-est} and the equality
		$\left({d\ke^\pm\over\rho\ke}\right)^{-2|\ln d\ke^\pm - \ln\rho\ke|^{-1}}=\exp(2)$,
		we get
		\begin{gather}\hspace{-2mm}
			\|(f -\la f \ra_{B^\pm\ke})\nabla \phi\ke^\pm\|_{\L(\wt P\ke^\pm)}\leq  C \eta\ke^\pm
			\|f \|_{\H^2(B\ke^\pm)},
			\label{I1:4}
		\end{gather} 
		Taking into account \eqref{H2est}, we deduce from \eqref{I1:4}:
		\begin{gather}\label{I1:5}
			\left(\suml_{k\in\Z^{n-1}}\|(f -\la f \ra_{B^\pm\ke})\nabla \phi\ke^\pm\|_{\L(\wt P\ke^\pm)}^2\right)^{1/2}\leq  C \sup_{k\in\Z^{n-1}}\eta\ke^\pm
			\|f\|_{\HS^2}.
		\end{gather}   
		Combining \eqref{I1:1}, \eqref{I1:3}, \eqref{I1:5} and taking into account \eqref{eta3}, we obtain
		\begin{gather}\label{I1:est:eta} 
			|I^{1,\pm}\e|\leq C\eta\e \|f\|_{\HS^2}\|u\|_{\HS^1\e} .
		\end{gather}
		
		Finally, we estimate  the second term in the right-hand-side of \eqref{I1:1}
		in another way. One has for $n\ge 3$:
		\begin{align}\notag
			&\suml_{k\in\Z^{n-1}}\|( f  - \la f \ra_{B^\pm\ke})\nabla \phi\ke^\pm\|_{\L(\wt P\ke^\pm)}^2 \leq 
			C(d\ke^\pm)^{-2}\suml_{k\in\Z^{n-1}}\|  f  - \la f \ra_{B^\pm\ke} \|_{\L(\wt P\ke^\pm)}^2\\ &\qquad\leq 
			C_1 \suml_{k\in\Z^{n-1}}\left( {(d\ke^\pm)^{n-2} \rho\ke^{-n}} 
			\|f  - \la f \ra_{B^\pm\ke}\|^2_{\L(B\ke^\pm)}+\|\nabla f  \|^2_{\L(B\ke^\pm)}\right)\notag\\\label{alternative}
			&\qquad\leq 
			C_2 \suml_{k\in\Z^{n-1}}\left(  (d\ke^\pm/\rho\ke)^{n-2} +1 \right)\|\nabla f  \|^2_{\L(B\ke^\pm)}
			\leq 
			C_3 \rho\e\suml_{k\in\Z^{n-1}} \| f  \|^2_{\H^2(R\ke^\pm)}\leq
			C_4 \rho\e\|f\|^2_{\HS^2}.
		\end{align}
		Here, to get the first inequality we use $|\nabla \phi\ke^\pm |\le C(d\e^\pm)^{-1}$ ($n\ge 3$), for the second one we use Lemma~\ref{lemma:ring} 
		(which we apply for $z=x\ke^\pm$, $E_1=\wt P\ke^\pm$, $E_2=B^\pm\ke$, $\varkappa={2}$, $\ell=\rho\e$, $a_1=\wt d\ke^\pm$, $a_2=\rho\ke$), 
		for the third one we utilize \eqref{Poincare:De} (with $E_\delta=B\ke^+$), 
		for the fourth one we use \eqref{Rke:est}, and  the last one
		follows from \eqref{H2est}.
		Combining \eqref{I1:1}, \eqref{I1:3}, \eqref{alternative} and taking into account \eqref{eta3+},
		we arrive at the estimate
		\begin{gather}\label{I1:est:rho} 
			n\ge 3:\quad |I^{1,\pm}\e|\leq C\rho\e^{1/2} \|f\|_{\HS^2}\|u\|_{\HS^1\e} .
		\end{gather}
		Taking into account that
		$\eta\ke^\pm \le C\rho\ke^{1/2}$ as $n=2$
		(this fact follows easily from \eqref{assump:3} and \eqref{eps0:1}--\eqref{eps0:4}),
		we obtain from \eqref{I1:est:eta}, \eqref{I1:est:rho} the desired estimate \eqref{I1:est}.
		The lemma is proven.        
	\end{proof}

	To proceed further, we need the following simple lemma below.
	
	\begin{lemma}
		One has:
		\begin{gather}\label{Cke:prop}
			\mathscr{C}\ke=\left(1, {\partial U\ke\over \partial x_n}\right)_{\L(D\ke^+)}
			=\left(1, {\partial U\ke\over \partial x_n}\right)_{\L(D\ke^-)}.
		\end{gather}
	\end{lemma}
	
	\begin{proof}
		We denote by  ${\partial\over\partial \nu}$ the exterior (with respect to $G\ke$) normal derivative 
		on $\partial G\ke$. Integrating by parts and using \eqref{BVP:U:1}--\eqref{BVP:U:4}, we obtain
		\begin{gather}\label{Cke:prop:1}
			\mathscr{C}\ke=\|\nabla U\ke\|^2_{\L(G\ke)}= 
			\left(1, {\partial U\ke\over \partial\nu}\right)_{\L(S\ke^+)}.
		\end{gather}
		On the other hand, taking into account that $\Delta U\ke=0$ in $B\ke^+$  
		and ${\partial U\ke\over \partial \nu}=0$ on $\partial B\ke^+\setminus (S\ke^+ \cup D\ke^+)$, 
		where ${\partial\over\partial \nu}$ is the exterior (with respect to $B\ke^+$) normal derivative 
		on $\partial B\ke^+$,     
		we obtain again integrating by parts:
		\begin{gather}\label{Cke:prop:2}
			\left(1, {\partial U\ke\over \partial\nu}\right)_{\L(S\ke^+)}=
			-\left(1, {\partial U\ke\over \partial\nu}\right)_{\L(D\ke^+)}
		\end{gather} 
		Taking into account that ${\partial \over \partial \nu}=-{\partial \over \partial x_n}$ on 
		$\partial B\ke^+\setminus S\ke^+$, we deduce from \eqref{Cke:prop:1}--\eqref{Cke:prop:2}
		the first equality in \eqref{Cke:prop}. The second equality is obtained similarly (using the function $1-U\ke$ instead of $U\ke$).
		The lemma is proven.
	\end{proof}
	
	Now, we are ready to investigate the term   $I^{2,+}\e$. 
	Taking into 
	account that $\psi\ke^+\in C^\infty(\R^n)$ and the properties (cf. \eqref{BVP:U:4}, \eqref{phipsi:2},  \eqref{phipsi:7})
	\begin{gather*}
		\psi\ke^+=1\text{ on }D\ke^+,\quad \psi\ke^+=0\text{ on }S\ke^+,\\
		{\partial U\ke\over\partial \nu}=0\text{ on }
		\partial B\ke^+\setminus (D\ke^+\cup S\ke^+),\quad
		{\partial \psi\ke^+\over\partial \nu}=0\text{ on }
		\partial B\ke^+ ,
	\end{gather*}
	where  ${\partial\over\partial \nu}$ is the exterior (with respect to $B\ke^+$) normal derivative 
	on $\partial B\ke^+$, we obtain
	\begin{align}\notag
		I^{2,+}\e 
		&= -\sum_{k\in\Z^{n-1}}\left(\la \overline{f}\ra_{B\ke^-}-\la \overline{f}\ra_{B\ke^+}\right) 
		(u, {\Delta((1-U\ke)\psi\ke^+)})_{\L(B\ke^+)} + Q^{1,+}\e
		\\\notag
		&=-\sum_{k\in\Z^{n-1}}\left(\la \overline{f}\ra_{B\ke^-}-\la \overline{f}\ra_{B\ke^+}\right)\la u\ra_{B\ke^+} (1, {\Delta((1-U\ke)\psi\ke^+)})_{\L(B\ke^+)} + Q^{1,+}\e+Q^{2,+}\e
		\\\notag
		&= -\sum_{k\in\Z^{n-1}}\left(\la \overline{f}\ra_{B\ke^-}-\la \overline{f}\ra_{B\ke^+}\right)\la u\ra_{B\ke^+}
		\left(1, {\partial U\ke\over \partial x_n}\right)_{\L(D\ke^+)} + Q^{1,+}\e+Q^{2,+}\e
		\\
		&= 
		Q^{1,+}\e+Q^{2,+}\e+Q^{3,\eps},\label{I2:start+}
	\end{align}
	where
	\begin{align*}
		Q^{1,+}\e &\ceq \sum_{k\in\Z^{n-1}} \left(\la \overline{f}\ra_{B\ke^-}-\la \overline{f}\ra_{B\ke^+}\right) \left(u, {\partial U\ke \over\partial x_n}
		\right)_{\L(D\ke^+)},
		\\
		Q^{2,+}\e &\ceq \sum_{k\in\Z^{n-1}}\left(\la \overline{f}\ra_{B\ke^+}-\la \overline{f}\ra_{B\ke^-}\right)\left(u- \la u\ra_{B\ke^+},{\Delta( (1-U \ke)\psi\ke^+)}\right)_{\L(B\ke^+)},
		\\
		Q^{3,+}\e &\ceq\sum_{k\in\Z^{n-1}} \mathscr{C}\ke \la u\ra_{B\ke^+}  \left(\la \overline{f}\ra_{B\ke^+}-\la \overline{f}\ra_{B\ke^-}\right) .
	\end{align*}
	(on the last step in \eqref{I2:start+} we use the equality \eqref{Cke:prop}).
	Similarly, we get 
	\begin{align} 
		I^{2,-}\e 
		= 
		Q^{1,-}\e+Q^{2,-}\e+Q^{3,-}\e,\label{I2:start-}
	\end{align}
	where
	\begin{align*}
		Q^{1,-}\e & \ceq 
		\sum_{k\in\Z^{n-1}} \left(\la \overline{f}\ra_{B\ke^+}-\la \overline{f}\ra_{B\ke^-}\right) \left(u, {\partial U\ke \over\partial x_n}\right)_{\L(D\ke^-)},\\
		Q^{2,-}\e &\ceq \sum_{k\in\Z^{n-1}}\left(\la \overline{f}\ra_{B\ke^-}-\la \overline{f}\ra_{B\ke^+}\right)\left(u- \la u\ra_{B\ke^-},{\Delta( U \ke\psi\ke^+)}\right)_{\L(B\ke^-)},
		\\   
		Q^{3,-}\e &\ceq\sum_{k\in\Z^{n-1}} \mathscr{C}\ke \la u\ra_{B\ke^-}  \left(\la \overline{f}\ra_{B\ke^-}-\la \overline{f}\ra_{B\ke^+}\right) .
	\end{align*}   
	
	The terms $Q^{2,\pm}\e$ tend to zero, namely one has the following estimate.
	
	\begin{lemma}        
		One has:
		\begin{gather}
			\label{Q2:est:pm} 
			|Q^{2,\pm}\e|\le  C\chi\e\|f\|_{\HS^1}\|u\|_{\HS^1\e}.
		\end{gather}
	\end{lemma}
	
	\begin{proof}
		Via the Cauchy-Schwarz inequality, we have $|Q^{2,+}\e|\le Q^{2,1,+}\e\cdot Q^{2,2,+}\e$, where
		\begin{align*}
			Q^{2,1,+}\e&\ceq 
			\left(2\sum_{k\in\Z^{n-1}}\rho\ke(\gamma\ke^+  )^{-2}
			\left(|\la  {f}\ra_{B\ke^+}|^2+|\la  {f}\ra_{B\ke^-}|^2\right) 
			\|\Delta ((1-U\ke )\psi\ke^+)\|^2_{\L(B\ke^+)}\right)^{1/2}\\
			Q^{2,2,+}\e&\ceq 
			\left( \sum_{k\in\Z^{n-1}}\rho\ke^{-1}(\gamma\ke^+  )^{ 2}\|u -
			\la u \ra_{B\ke^+}\|^2_{\L(B\ke^+)}\right)^{1/2}. 
		\end{align*} 
		
		First, we estimate the factor $Q^{2,1,+}$.  
		Let us look closely at the function $ 1-U\ke$ being restricted to the set
		$B\ke^+\setminus\overline{\wh P\ke^+}$ (recall that the set $\wh P\ke^+$ is defined in \eqref{whPke}). This function is harmonic on this set, it vanishes on $S\ke^+$,  its normal derivative vanishes on $\partial (B\ke^+\setminus\overline{\wh P\ke^+})\cap \Gamma\e^+$;
		furthermore (see \eqref{U01}) one has
		$$
		0\le 1 - U\ke(x)\le 1,\ x\in G\ke.
		$$
		Then, by Lemma~\ref{lemma:V} (being applied for $z=x\ke^+$, $F_1= P\ke^+$, $F_2=B\ke^+$, $\varkappa=4$, $\ell=\rho\e$, $a_1=d\ke^+$, $a_2=\rho\ke$, $\tau=2$), we get the estimate
		\begin{align}
			\label{dUke:est:1}
			\forall j\in \{1,\dots,n\}:\quad
			\left|{\partial U\ke  \over\partial x_j}(x)\right|\leq 
			C{\rho\ke^{1-n} \over \GG(d\ke^+) - \GG(\rho\ke)}\leq C_1\gamma\e^+,&& x\in 
			B\ke^+\setminus\overline{\wh P\ke^+}
		\end{align}
		(the last inequality is due to \eqref{eps0:2}--\eqref{eps0:4}).
		Furthermore, by virtue of Lemma~\ref{lemma:PFTS} (namely, we use the inequality   \eqref{Friedrichs:De} for $z=x\ke^+$,
		$E_\delta=B\ke^+\setminus \overline{\wh P\ke^+}$ and $S_\delta=S\ke^+$), one has
		\begin{gather}\label{Uke:Frie}
			\|1- U\ke \|^2_{\L (B\ke^+\setminus \overline{\wh P\ke^+})}\leq
			C\rho\ke^{2}\|\nabla U\ke \|^2_{\L(B\ke^+\setminus \overline{\wh P\ke^+})}.
		\end{gather}
		Using  \eqref{dUke:est:1}, \eqref{Uke:Frie} and the properties
		\eqref{phipsi:3}, \eqref{phipsi:6}, \eqref{phipsi:8},    
		we obtain:
		\begin{multline} 
			\|\Delta((1-U\ke)\psi\ke^+)\|^2_{\L(B\ke^+)} 
			\leq
			C_1 \left(\rho\ke^{ -2}\|\nabla U\ke \|^2_{\L(B\ke^+\setminus \overline{\wh P\ke^+})}+
			\rho\ke^{ -4}\|U\ke -1\|^2_{\L(B\ke^+\setminus \overline{\wh P\ke^+})}\right)
			\\
			\leq C_2\rho\ke^{ -2}\|\nabla U\ke\|^2_{\L (B\ke^+\setminus \overline{\wh P\ke^+})}
			\leq C_3\rho\ke^{n -2}\|\nabla U\ke\|^2_{\LL^\infty(B\ke^+\setminus \overline{\wh P\ke^+})} 
			\leq C_4\rho\ke^{n-2}(\gamma\ke^+)^2.\label{est:div}
		\end{multline}
		Combining \eqref{abs:fke} and \eqref{est:div},
		we obtain 
		\begin{align} \notag
			Q\e^{2,1,+}& \leq C\left(\sum_{k\in\Z^{n-1}}\rho\ke^{n-1} \left(|\la f\ra_{B\ke^-}|^2+|\la f\ra_{B\ke^+}|^2\right) \right)^{1/2}
			\\
			&\leq C_1\left(\sum_{k\in\Z^{n-1}}  \left(\|f\|_{\L(R\ke^-)}^2+
			\|f\|_{\L(R\ke^+)}^2\right) \right)^{1/2}
			\le C_1\|f\|_{\HS^1}.\label{Q2:term1}
		\end{align}
		
		Finally, using  Lemma~\ref{lemma:PFTS} (namely, we use the inequality  \eqref{Poincare:De} for $z=x\ke^+$, $E_\delta=B\ke^+$), we estimate the second factor $Q^{2,2,+}\e$ as follows:
		\begin{align} 
			Q\e^{2,2,+}\leq 
			C\left(\sum_{k\in\Z^{n-1}}\rho\ke(\gamma\ke^+ )^{2}\|\nabla u\|^2_{\L(B\ke^+)}\right)^{1/2}
			\leq
			C\sup_{k\in\Z^{n-1}}(\rho\ke^{1/2}\gamma\ke^+ ) \|u\|_{\HS^1}. 
			\label{Q2:term2} 
		\end{align}
		The   estimate \eqref{Q2:est:pm} for $Q\e^{2,+}$ follows from  \eqref{Q2:term1}, \eqref{Q2:term2}, \eqref{Ake:prop:1}, \eqref{Ake:prop:1+}. For $Q\e^{2,-}$ the proof is similar. The lemma is proven.
	\end{proof}

	Using \eqref{wtJprop}, \eqref{H2est}, \eqref{Se:est}, we get
	\begin{align}\notag
		|I^{4}\e| &\leq
		\|\nabla (\wt\J\e^1 u)\|_{\L(O\e^+\cup O\e^-)}\|{\nabla f}\|_{\L(O\e^+\cup O\e^-)}
		\leq
		C\eps^{1/2} \|\nabla (\wt\J\e^1 u)\|_{\L(O\e^+\cup O\e^-)}\|f\|_{\H^2(O^+\cup O^-)}
		\\
		&\leq C_1\eps^{1/2} \|u\|_{\H^1(\Omega^+\e\cup\Omega^-\e)}\| f\|_{\H^2(O^+\cup O^-)}
		\leq C_2\eps^{1/2} \|u\|_{\HS^1\e}\|f\|_{\HS^2}.\label{I4:est}
	\end{align}
	Taking into account that $\Delta U\ke=0$ on $G\ke$, we get easily via integration by parts:
	\begin{gather}
		\label{QQI}
		Q^{1,+}\e+Q^{1,-}\e+I^3\e=0
	\end{gather}
	Finally, using \eqref{assump:main}, \eqref{wtJprop}, \eqref{H2est} and taking into account that
	$\wt\J\e^1 u = u$ on $B\ke^\pm$, we obtain
	\begin{gather}\label{Q3Q3I5}
		|Q^{3,+}+ Q^{3,-}+I^{5}\e|\leq \kappa\e\|f\|_{\H^2(O^+\cup O^-)}\|\wt\J^1\e u\|_{\H^1(O^+\cup O^-)}\leq C\kappa\e\|f\|_{\HS^2}\|u\|_{\HS^1\e}.
	\end{gather}    
	
	Combining \eqref{aa}, \eqref{I1:est}, \eqref{I2:start+}--\eqref{Q2:est:pm}, \eqref{I4:est}--\eqref{Q3Q3I5}, we deduce the  lemma below.

	\begin{lemma}\label{lemma:A3}
		One has  
		\begin{multline}\label{lemma:A3:est}
			\forall f\in \HS^2 ,\ u\in \HS^1\e:\quad
			\left|\a\e[ u,\J^1\e f]-\a[\wt\J^{1}\e u,f]\right|
			\leq 
			C
			\max\left\{\chi\e,\,
			\eps^{1/2},\
			\kappa\e
			\right\} 
			\|f\|_{\HS^2 }\|u\|_{\HS^1\e}.
		\end{multline}
	\end{lemma}
	
	\subsection{End of the proofs of the main results}
	
	\subsubsection{Proof of Theorem~\ref{th1}}   
	It follows from \eqref{JJ} and Lemmas~\ref{lemma:A1}, \ref{lemma:A2},         
	\ref{lemma:A3} that conditions 
	\eqref{thA1:0}--\eqref{thA1:3} hold with 
	$
	\delta\e=
	C\max\left\{
	\eps^{1/2},\,
	\zeta\e,\, \kappa\e,\,
	\chi\e
	\right\},
	$
	whence, by Theorem~\ref{thA1}, we conclude the estimate \eqref{th1:est:1}.
	Furthermore, for any $f\in\HS $ and $u\in \HS\e $ we have
	\begin{gather}\label{finalest:1}
		\|\J\e f\|_{\HS\e}\le \|f\|_{\HS },\quad
		\|\wt\J\e u\|_{\HS}=\|u\|_{\HS\e},\quad 
		\|u -\J\e \wt\J\e u\|_{\HS\e}=0,
	\end{gather}
	and for any $f\in\HS^1$ we derive, using 
	\eqref{Se:est}:
	\begin{align} 
		\|f -\wt \J\e \J\e f\|^2_{\HS } =\|f\|^2_{\L(\Sigma\e)}\leq 
		\|f\|^2_{\L(O\e^+\cup O\e^-)} 
		\leq 
		C\eps \|f\|^2_{\H^1(O^+\cup O^-)} \leq 
		C\eps\|f\|^2_{\HS^1}. \label{finalest:2}
	\end{align}
	Then, by virtue of Theorem~\ref{thA2},  \eqref{th1:est:1} (which is already proven above), \eqref{finalest:1}, \eqref{finalest:2} imply the remaining estimates \eqref{th1:est:2}--\eqref{th1:est:4}.
	Theorem~\ref{th1} is proven.

	\subsubsection{Proof of Theorem~\ref{th2}} 
	
	By Theorem~\ref{thA3},   
	\eqref{JJ}, \eqref{lemma:A1:est}, \eqref{lemma:A2:est}, \eqref{lemma:A3:est},
	\eqref{finalest:1}, \eqref{finalest:2}
	imply  
	\begin{gather}\label{finalest:H1}
		\|(\A\e+\Id)^{-1}\J\e -\J\e^1 (\A+\Id)^{-1}\|_{\HS\to\HS^1\e} \leq C\sigma\e.
	\end{gather}
	
	Now, let $f\in\HS$ and $g=(\A+\Id)^{-1} f$.
	One has:
	\begin{multline}
		((\A\e+\Id)^{-1} \J\e f  - \J\e(\A+\Id)^{-1} f)\restriction_{\Omega\e^+}   -\mathscr{K}^+\e f\\ 
		= 
		((\A\e+\Id)^{-1} \J\e f  - \J^1\e(\A+\Id)^{-1} f)\restriction_{\Omega\e^+}  +  \mathscr{V}^+\e g + \mathscr{W}^+\e g, \label{AAeq}
	\end{multline}
	where
	\begin{gather*}
		\mathscr{V}^+\e g\ceq \suml_{k\in\Z^{n-1}}\left(\la g \ra_{B\ke^+}-g(x)\right) \phi\ke^+(x),
		\\
		\mathscr{W}^+\e g\ceq   \suml_{k\in\Z^{n-1}}\left(\la g \ra_{B\ke^-} - \la g \ra_{B\ke^+}\right)
		\left(1-U\ke(x)\right)\left(\psi\ke^+(x)-1\right).
	\end{gather*}
	One has the estimates (see \eqref{JJ:est1}, \eqref{I1:3}, \eqref{I1:5}, \eqref{alternative}, \eqref{eta3}, \eqref{eta3+})
	\begin{align}\label{V1}
		\|\mathscr{V}^+\e g\|_{\L(\Omega\e^+)}^2&\leq 
		\sum_{k\in\Z^{n-1}}\|g-\la g \ra_{B\ke^+}\|_{\L(\wt P\ke^+)}^2
		\leq C \chi\e^2\|g\|^2_{\HS^1},
		\\\notag
		\|\nabla (\mathscr{V}^+\e g)\|_{\L(\Omega\e^+)}^2&\leq 
		\Bigl( {\suml_{k\in\Z^{n-1}}\|\nabla g  \|_{\L(\wt P\ke^+)}^2\Bigr)^{1/2}}  + \Bigl(
		{\suml_{k\in\Z^{n-1}}\|( g  - \la g \ra_{B^+\ke})\nabla \phi\ke^+\|_{\L(\wt P\ke^+)}^2\Bigr)^{1/2}}
		\\ &\leq C\chi\e^2\|g\|^2_{\HS^2}.\label{V2}
	\end{align}
	Furthermore, we have (see \eqref{JJ:est2} and \eqref{Ake:prop:2}--\eqref{Ake:prop:2+})
	\begin{gather}\label{W1}
		\|\mathscr{W}^+\e g\|_{\L(\Omega\e^+)}^2\leq
		2\suml_{k\in\Z^{n-1}}\left(|\la f \ra_{B\ke^-}|^2+|\la f \ra_{B\ke^+}|^2\right) 
		\|U\ke-1\|^2_{\L(B\ke^+)}\leq 
		C \chi\e^2\|g\|^2_{\HS^1}.
	\end{gather}
	Finally, using \eqref{phipsi:3}, \eqref{phipsi:5}, \eqref{phipsi:6},  \eqref{abs:fke} \eqref{dUke:est:1}, \eqref{Uke:Frie}, 
	we estimate the gradient of $\mathscr{W}^+\e g$:
	\begin{align}\notag
		\|\nabla(\mathscr{W}^+\e g)\|_{\L(\Omega\e^+)}^2&\leq
		C\sum_{k\in\Z^{n-1}}\left(
		|\la g\ra_{B\ke^+}|^2+|\la g\ra_{B\ke^-}|^2\right)
		\left(\| \nabla U\ke\|^2_{\L(B\ke^+\setminus\overline{\wh P\ke^+})}+\rho\ke^{-2}\|1-U\ke \|^2_{\L(B\ke^+\setminus\overline{\wh P\ke^+})}\right)
		\\\notag
		&\leq
		C_1\sum_{k\in\Z^{n-1}}\rho\ke^{1-n}\left(\|g\|^2_{\H^1(R\ke^+)}+\|g\|^2_{\H^1(R\ke^-)}\right)
		\|\nabla U\ke\|^2_{\L(B\ke^+\setminus\overline{\wh P\ke^+})} 
		\\\notag
		&\leq
		C_2\sum_{k\in\Z^{n-1}}\rho\ke\left(\|g\|^2_{\H^1(R\ke^+)}+\|g\|^2_{\H^1(R\ke^-)}\right)
		\|\nabla U\ke\|^2_{\LL^\infty(B\ke^+\setminus\overline{\wh P\ke^+})} 
		\\\notag
		&\leq
		C_3\sum_{k\in\Z^{n-1}}\rho\ke \left(\|g\|^2_{\H^1(R\ke^+)}+\|g\|^2_{\H^1(R\ke^-)}\right)
		(\ga\ke^+)^2
		\\\label{W2}
		&\leq C_4\sup_{k\in\Z^{n-1}}\rho\ke(\gamma\ke^+)^2 \|g\|^2_{\HS^1}\leq C_4\chi\e^2 \|g\|^2_{\HS^1}
	\end{align}
	(on the last step we use \eqref{Ake:prop:1}--\eqref{Ake:prop:1+}).
	Combining \eqref{V1}--\eqref{W2}   we arrive at the estimate
	\begin{gather}\label{finalest:VW}
		\|\mathscr{V}\e^+ g\|_{\H^1(\Omega\e^+)}+\|\mathscr{W}\e^+ g\|_{\H^1(\Omega\e^+)}
		\leq C\chi\e\|g\|_{\HS^2}=C\chi\e\|f\|_{\HS}.
	\end{gather}
	The desired   estimate \eqref{th2:est:1} with plus sign follows from \eqref{finalest:H1}, \eqref{AAeq}, \eqref{finalest:VW}; the proof of the minus sign part is similar.
	
	The estimate \eqref{th2:est:2} follows from \eqref{finalest:H1} and the fact
	that $\mathscr{K}^T\e f =(\J\e^1 (\A+\Id)^{-1} f)\restr_{\cup_{k\in\Z^{n-1}}T\ke}$.
	
	Using \eqref{UL2}, \eqref{abs:fke}, 
	we obtain the estimate \eqref{K:prop1} with plus sign:
	\begin{align*}\notag
		\|\mathscr{K}\e^+ f\|_{\L(\Omega\e^+)}^2&\leq
		2\sum_{k\in\Z^{n-1}}\left(|\la g\ra_{B\ke^-}|^2+|\la g\ra_{B\ke^+}|^2\right)
		\|1-U\ke\|^2_{\L(B\ke^+)}\\\notag
		&\leq C\sum_{k\in\Z^{n-1}}\rho\ke^{2}\gamma\ke\left(\|g\|^2_{\H^1(R\ke^-)}+\|g\|^2_{\H^1(R\ke^+)}\right)\\
		&\leq
		C\sup_{k\in\Z^{n-1}}(\rho\ke^{2}\gamma\ke)\|g\|^2_{\HS^1}\leq 
		C\sup_{k\in\Z^{n-1}}(\rho\ke^{2}\gamma\ke)\|f\|^2_{\HS}.
	\end{align*}
	The proof of \eqref{K:prop1} with minus sign is similar.
	Furthermore, by means of \eqref{abs:fke}, \eqref{U01}, \eqref{Tke:est}, we  
	deduce the estimate \eqref{K:prop2}:  
	\begin{multline*}
		\|\mathscr{K}\e^T f\|_{\L(\cup_{k\in\Z^{n-1}}T\ke)}^2
		\le 2\sum_{k\in\Z^{n-1}}\left(|\la g\ra_{B\ke^-}|^2+|\la g\ra_{B\ke^+}|^2\right)|T\ke|
		\\
		\leq C\zeta\e^2\gamma\e\sum_{k\in\Z^{n-1}}\left(\|g\|_{\H^1(R\ke^-)}^2+\|g\|_{\H^1(R\ke^+)}^2\right)
		\leq C\zeta\e^2\ga\e \|g\|_{\HS^1}^2\leq C\zeta\e^2\ga\e\|f\|_{\HS}^2.
	\end{multline*}
	
	Finally, taking into account the definition of $\mathscr{C}\ke$, we get  
	\begin{align*}
		\|\nabla (\mathscr{K}\e^+ f)\|^2_{\L(\Omega^+\e)}+
		\|\nabla (\mathscr{K}\e^- f)\|^2_{\L(\Omega^-\e)}+
		\|\nabla (\mathscr{K}\e^T f)\|^2_{\L(\cup_{k\in\Z^{n-1}}T\ke)}=
		\sum_{k\in\Z^{n-1}}\mathscr{C}\ke\left|\la g\ra_{B\ke^-}-\la g\ra_{B\ke^+}\right|^2,
	\end{align*}
	whence, by virtue of \eqref{assump:main} and \eqref{H2est}, we arrive at the desired equality \eqref{K:prop3}:
	\begin{gather*}
		\left|\|\nabla (\mathscr{K}\e^+ f)\|^2_{\L(\Omega^+\e)}+
		\|\nabla (\mathscr{K}\e^- f)\|^2_{\L(\Omega^-\e)}+
		\|\nabla (\mathscr{K}\e^T f)\|^2_{\L(\cup_{k\in\Z^{n-1}}T\ke)}-\|{\mu}^{1/2}[g]\|^2_{\L(\Gamma)}\right|\\
		\leq 
		\kappa\e\|g\|_{\H^2(O\setminus\Gamma)}^2\leq C\kappa\e\|g\|_{\HS^2}^2=C\kappa\e\|f\|_{\HS}^2.
	\end{gather*}
	Theorem~\ref{th2} is proven.
	
	\subsubsection{Proof of Theorem~\ref{th3}}   
	
	First, we return back to the estimate \eqref{Se:est}. 	
	It is easy to see that 
	\begin{gather}\label{Se:est:c}
		\text{the constant }C\text{ in \eqref{Se:est} equals }\max\{L,\,4 L^{-1}\}.
	\end{gather}
	Then, using \eqref{Se:est} and \eqref{Se:est:c}, we get
	\begin{align*}  
		\|f\|^2_{\HS}&=\|\J\e f\|^2_{\HS\e}+\|f\|^2_{\L(\Sigma\e)}
		\leq \|\J\e f\|^2_{\HS\e}+\|f\|^2_{\L(O^+\e)}+\|f\|^2_{\L(O^-\e)}
		\\ 
		&\leq\|\J\e f\|^2_{\HS\e}+\eps \max\{L,\,4 L^{-1}\}\big( \|f\|^2_{\H^1(\Omega^+\e)}+ \|f\|^2_{\H^1(\Omega^-\e)}\big)\\
		&\le
		\|\J\e f\|^2_{\HS\e}+\eps \max\{L,\,4 L^{-1}\} (\a[f,f]+\|f\|^2_{\HS}),
	\end{align*}
	whence, taking into account that $1-\eps \max\{L,\,4 L^{-1}\}\geq 1/2$ (this follows easily from \eqref{eps0:last:1}--\eqref{eps0:last:2}), 
	we conclude
	\begin{align} \notag
		\|f\|^2_{\HS}&\leq (1-\eps \max\{L,\,4 L^{-1}\})^{-1}\left(\|\J\e f\|^2_{\HS\e}+
		\eps \max\{L,\,4 L^{-1}\}\a[f,f]\right)\\
		&\leq \label{A3:est}
		2\|\J\e f\|^2_{\HS\e}+C\eps\a[f,f].
	\end{align}
	Moreover, one has
	\begin{align}\label{A4:est}
		\|u\|^2_{\HS\e}=\|\wt\J\e u\|^2_{\HS },
	\end{align}
	Thus conditions \eqref{thA4:3} and \eqref{thA4:4} hold
	with $\mu\e=2$, $\nu\e=C\eps$, $\wt\mu\e=1$, $\wt\nu\e=0$. 
	Then, by Theorem~\ref{thA4}, the properties  \eqref{th1:est:1}, \eqref{th1:est:2}, \eqref{A3:est}, \eqref{A4:est} imply the desired estimate  \eqref{th3:est}.
	Theorem~\ref{th3} is proven.

	\begin{remark} 
		\label{rem:intersection}
		Let us omit the assumptions \eqref{Gamma:dist} on $\Omega$, i.e.  the hyperplane $\Gamma$ 
		(and, consequently, the its $\eps$-neighborhood $O\e$)
		may intersect
		$\partial\Omega$. We construct the domain $\Omega\e$ similarly to \eqref{Omegae}
		with the natural additional assumption: the passages $T\ke$ must belong to $O\e\cap\Omega$
		and their total number  can be finite, i.e., everywhere $\Z^{n-1}$ should be replaced by 
		some $\mathbb{M}\e\subset \Z^{n-1}$ (possibly, with $\# \mathbb{M}\e<\infty$).
		We consider the operator $\A\e$ being associated with the form 
		\eqref{ae} and the operator $\A$ being associated with the form 
		\eqref{aga} with ${\mu}\in C^{1}(\Gamma\cap\Omega)\cap\W^{1,\infty}(\Gamma\cap\Omega)$ and 
		the integral over $\Gamma$ being replaced by the integral over $\Gamma\cap\Omega$.
		Assume that the
		assumptions \eqref{assump:0}--\eqref{non:con}, \eqref{assump:main} still hold
		(of course, with $\mathbb{M}\e$ instead of $\Z^{n-1}$ and other appropriate modifications).
		
		It is natural to ask, whether the main results of this work remain valid in this new setting.
		When trying to repeat the above proofs,   
		we  face a difficulty:  $\H^2$-regularity of the functions from $\dom(\A)$  (cf.~\eqref{H2est})
		may fail in a neighborhood of $\partial\Omega\cap\Gamma$ -- on this set  $\delta'$-interface conditions ``meet'' with the Neumann boundary conditions. In fact, this is essentially the only 
		place  where 
		the assumption \eqref{Gamma:dist} comes into play. 
		
		However, this difficulty can be bypassed if  we impose the following 
		additional restriction: the passages $T\ke$ 
		belong to the set $\wh O\e\ceq O\e\cap \{x\in\R^n:\ x'\in\wh\Gamma\},$ 
		where $\wh\Gamma $ is a relatively open  subset of $\Gamma\cap\Omega$ satisfying $\dist(\wh \Gamma,\partial\Omega)>0$ (and, hence, for sufficiently small $\eps$, we have $\wh O\e\subset \Omega$); 
		furthermore, we assume that \eqref{assump:main} holds with ${\mu}\in C^{1}(\Gamma\cap \Omega)\cap\W^{1,\infty}(\Gamma\cap\Omega)$ satisfying $\supp(\mu)\subset\wh\Gamma$ and  $O$ being replaced by 
		$\wh O \ceq O \cap\{x\in\R^n:\ x'\in\wh\Gamma\}$.
		In this case the functions from $\dom(\A)$  satisfy  \eqref{H2est} with $O^\pm$ being replaced by $\wh O^\pm\ceq \wh O\cap\Xi^\pm$. 
		With the above changes the main results remain the same, and  the proofs are similar to the case of $\Omega$ satisfying \eqref{Gamma:dist}.
	\end{remark}

	\section{Non-concentrating property}
	\label{sec:6}
	
	In this section we discuss several examples of passages 
	for which the non-concentrating property \eqref{non:con} is either fulfilled or not fulfilled.
	
	\subsection{Straight passages}
	\label{subsec:6:1}
	
	We start from the simplest case, when the passages are straight, i.e.
	$$T\ke =  \left\{x=(x',x_n)\in\R^n:\ |x_n|<\eps,\  x'\in \D\ke\right\},$$ 
	where $\left\{\D\ke,\ k\in\Z^{n-1}\right\}$ is a family of  domains in $\R^{n-1}$, that is $$D\ke^\pm=\{x\in\R^n:\ x'\in\D\ke,\ x_n=\pm\eps\};$$ 
	it goes without saying that the sets
	$\overline{\D\ke}$, $k\in\Z^{n-1}$ are pairwise disjoint.
	
	To demonstrate the fulfilment of \eqref{non:con},
	we define the set $\wh T\ke$ by
	$$
	\wh T\ke = \left\{x=(x',x_n)\in\R^n:\ |x_n|<L,\  x' \in \D\ke\right\}
	$$
	(recall that $L$ is given in \eqref{Gamma:dist}).
	Note that 
	\begin{gather}\label{whTke}
		T\ke\subset \wh T\ke\subset\Omega\e,\quad \wh T\ke\cap \wh T_{j,\eps}=\emptyset,\ k\not=j.
	\end{gather}
	By Lemma~\ref{lemma:cylinder} (which we apply for $E_1=T\ke$ and $E_2=\wh T\ke$) we have  the estimate 
	\begin{gather}\label{non:con+}
		\forall u\in \H^1(\wh T\ke):\  \|u\|_{\L(T\ke)}^2 \leq 
		C\eps\|u\|_{\H^1(\wh T\ke)}^2.
	\end{gather}
	It follows from \eqref{whTke} and \eqref{non:con+} that the desired estimate 
	\eqref{non:con} holds with
	$\zeta\e=C\eps^{1/2}$. 
	
	\begin{remark}\label{rem:noncon:opt}
		The  estimate \eqref{non:con} with $\zeta\e=C\eps^{1/2}$ obtained above is far from being optimal -- we still did not use the fact that
		the passages $T\ke$ are not only short (their lengths $\eps$ tend to $0$), but also very thin (cf.~\eqref{assump:3}). For example, for $n=2$ the property \eqref{non:con} holds  with
		\begin{gather*}
			\al\e=C \eps^{1/2}\max \left\{\eps^{1/2},\, {\sup}_{k\in\Z^{n-1}}(  d\ke|\ln d\ke|)^{1/2}\right\} ,
		\end{gather*}
		where $d\ke$ stands for the radius of the smallest ball containing $\D\ke$; this statement is a particular case of
		Proposition~\ref{prop:curved} below.
		However, this improvement will give us no benefits, since the convergence rate $\eps^{1/2}$  enters the expression for $\sigma\e$ anyway 
		(see Lemma~\ref{lemma:A2}).
	\end{remark}

	\subsection{Hourglass-like passages}
	\label{subsec:6:2}
	
	Assume that the passage $T\ke$ can be represented in the form
	\begin{gather}\label{Tke:form}
		T\ke=\mathrm{int}(\overline{F\ke^-\cup F\ke^+}), 
	\end{gather}
	where  the domains $F\ke^\pm$  are as follows:
	\begin{gather}\label{Fke}
		\begin{array}{l}
			F\ke^+=\left\{x=(x',x_n)\in\R^n:\ x'\in \D\ke^+,\ \mathcal{F}^+\ke(x')<x_n<\eps\right\},\\
			F\ke^-=\left\{x=(x',x_n)\in\R^n:\ x'\in \D\ke^-,\ -\eps<x_n<\mathcal{F}^-\ke(x') \right\}.
		\end{array}
	\end{gather}
	Here $\D\ke^\pm$ are bounded domains in $\R^{n-1}$ (that is 
	$D\ke^\pm=\{x=(x',x_n):\ x'\in \D\ke^\pm,\ x_n=\pm\eps\}$), $\mathcal{F}^\pm\ke: \D\ke^\pm\to [-\eps,\eps]$
	are some continuous and piecewise smooth functions.
	
	The name \emph{hourglass-like} that we use for such passages is chosen since first of all we keep in mind passages having geometry like on Figure~\ref{fig:passages}\,(a). However, not only hourglass looking passages obey the above assumption -- another (not hourglass looking) example obeying them is presented on Figure~\ref{fig:passages}\,(b). 
	
	Note that the result we formulate below also holds for the passages looking like 
	the one in Figure~\ref{fig:passages}\,(c) -- such kind passages are defined in the same way (via \eqref{Tke:form}--\eqref{Fke}), but with piecewise continuous functions $\mathcal{F}\ke^\pm$.

	\begin{figure}[h]
		\begin{picture}(420,105)
			\includegraphics[width=0.9\textwidth]{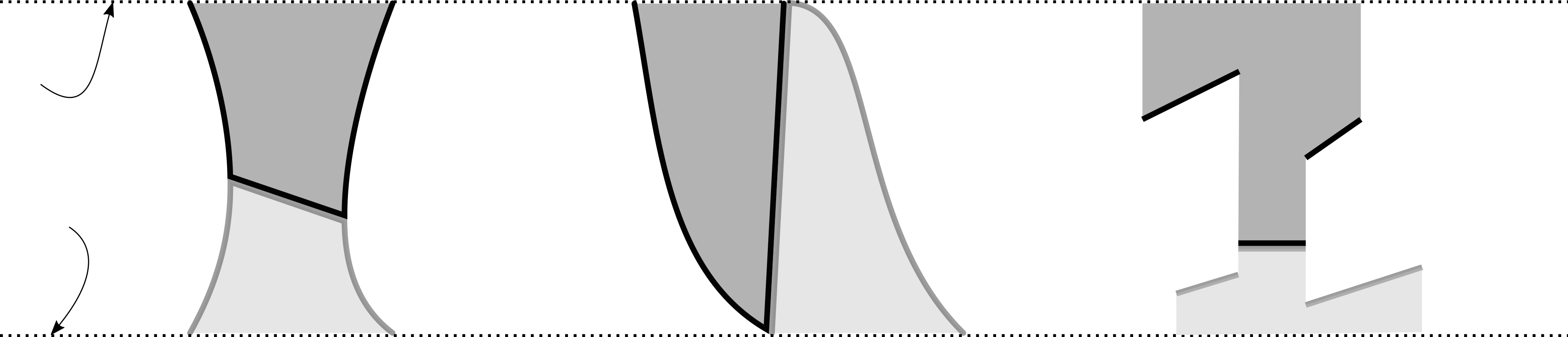}
			\put(-85,95){\text{(c)}}
			\put(-220,95){\text{(b)}}
			\put(-330,95){\text{(a)}}
			
			\put(-400,65){$\Gamma\e^+$}
			\put(-393,30){$\Gamma\e^-$} 
			
			\put(-333,55){$F\ke^+$} 
			\put(-333,15){$F\ke^-$} 
			
		\end{picture}
		\caption{Different types of hourglass-like passages. The sets $F\ke^+$ and $F\ke^-$ are painted in dark gray and  light gray, respectively. The bold black and bold grey
			curves correspond to the graphs of the functions $\mathcal{F}\ke^+$ and $\mathcal{F}\ke^-$, respectively. }\label{fig:passages}
	\end{figure}
	
	\begin{proposition}\label{prop:hourglass}
		For passages defined by \eqref{Tke:form}--\eqref{Fke} the estimate \eqref{non:con}
		holds with $\zeta\e=C\eps^{1/2}$.
	\end{proposition}
	
	\begin{remark}\label{rem:notoptimal}
		Similarly to  the case of straight passages, the estimate obtained in Proposition~\ref{prop:hourglass} 
		is not optimal (we did not use \eqref{assump:3} in its proof).
		However, as we already noticed (see remark~\ref{rem:noncon:opt}), at the end of the day the eventual improvement gives us no benefits.
	\end{remark} 
	
	\begin{proof}[Proof of Proposition~\ref{prop:hourglass}]
		We introduce the following sets:
		\begin{gather*}
			\wh F\ke^\pm \ceq 
			\mathrm{int}\left( \left\{x=(x',x_n)\in\R^n:\  \eps\le \pm x_n\le L,\ x'\in \overline{\D\ke^\pm}\right\}\cup \overline{F\ke^\pm}\right).
		\end{gather*}
		Note that
		\begin{gather}\label{Eke:prop}
			F\ke^\pm\subset \wh F\ke^\pm\subset\Omega\e,\quad \wh F\ke^\pm\cap 
			\wh F_{j,\eps}^\pm=\emptyset,\ k\not=j.
		\end{gather}
		
		Let $u\in C^1(\overline{\wh F\ke^+})$. For $x'\in \D\ke^+$, 
		$y\in (\mathcal{F}\ke^+(x'),\eps)$ and $z\in (\mathcal{F}\ke^+(x'),L)$, we have
		\begin{align*} 
			|u(x',y)|^2 & \leq 2|u(x',z)|^2+2\left(L-\mathcal{F}\ke^+(x')\right)
			\int_{\mathcal{F}\ke^+(x')}^{L}\left|\nabla u(x',\tau)\right|^2\d\tau\\
			& \leq 2|u(x',z)|^2+2(L+\eps)
			\int_{\mathcal{F}\ke^+(x')}^{L}\left|\nabla u(x',\tau)\right|^2\d\tau 
		\end{align*}
		(the proof is similar to the proof of the estimate \eqref{FTC+}).
		Integrating the above inequality over $y\in (\mathcal{F}\ke^+(x'),\eps)$ and $z\in (\mathcal{F}\ke^+(x'),L)$,  and then dividing by $L-\mathcal{F}\ke^+(x')$, 
		we get
		\begin{align*} 
			\int_{\mathcal{F}\ke^+(x')}^\eps |u(x',y)|^2 \d y&\leq 2
			\frac{\eps-\mathcal{F}\ke^+(x')}{L-\mathcal{F}\ke^+(x')}
			\int_{\mathcal{F}\ke^+(x')}^{L} |u(x',z)|^2\d z\\&+
			2(L+\eps)(\eps-\mathcal{F}\ke^+(x'))\int_{\mathcal{F}\ke^+(x')}^{L}\left|\nabla u(x',\tau)\right|^2\d\tau\\
			&\leq 
			\frac{4\eps}{L-\eps }
			\int_{\mathcal{F}\ke^+(x')}^{L} |u(x',z)|^2+
			4(L+\eps)\eps \int_{\mathcal{F}\ke^+(x')}^{L}\left|\nabla u(x',\tau)\right|^2\d\tau.
		\end{align*}
		Finally, we integrate the estimate above over $x'\in \D\ke^+$. We obtain (taking into account that $\eps\le\eps_0\le L/8$, see~\eqref{eps0:last:1}): 
		\begin{gather}\label{non:con2+}
			\|u\|^2_{\L(F\ke^+)}\leq
			C\eps\|\nabla u\|^2_{\H^1(\wh F\ke^+)}.
		\end{gather}
		By the density arguments \eqref{non:con2+} holds for any $u\in \H^1(\wh F\ke^+)$.
		Similarly,  $\forall u\in \H^1(\wh F\ke^-)$ we have
		\begin{gather}\label{non:con2-}
			\|u\|^2_{\L(F\ke^-)}\leq
			C\eps\|\nabla u\|^2_{\H^1(\wh F\ke^-)}.
		\end{gather}
		From \eqref{Tke:form}--\eqref{non:con2-},    
		we   conclude that \eqref{non:con} holds with
		$\zeta\e=C\eps^{1/2}$ Q.E.D.
	\end{proof}

	\subsection{Passages with moderately varying cross-section}
	\label{subsec:6:3}
	
	In this subsection we consider a class of passages  for which the assumption \eqref{non:con} still holds, but they are
	not of the form \eqref{Tke:form}--\eqref{Fke}. 
	To simplify the presentation, we restrict ourselves to the case $n=2$; for $n\ge 3$ one can easily derive similar results.
	
	Let $z\ke \in\R$,  $g\ke\in \W^{1,\infty}(-\eps,\eps)$, $h\ke\in \W^{1,\infty}(-\eps,\eps)$ such that 
	\begin{gather*}
		g\ke>0,\ h\ke>0,\quad 1/g\ke\in\LL^\infty(-\eps,\eps),\ 1/h\ke\in\LL^\infty(-\eps,\eps) .
	\end{gather*}
	We assume that the passage $T\ke$ has the form (see Figure~\ref{fig:passages2}\,(a))
	\begin{gather}\label{Tke:gh}
		T\ke =  \left\{x=(x_1,x_2)\in\R^2:\ |x_2|<\eps,\ -g\ke(x_2)<x_1-z\ke<h\ke(x_2)\right\},
	\end{gather} 
	that is $D\ke^\pm=\left\{x=(x_1,x_2)\in\R^2:\ -g\ke(\pm\eps)<x_1-z\ke<h\ke(\pm\eps),\ x_2=\pm \eps\right\}$, and
	\begin{gather}\label{dke:xke}
		d\ke^\pm={1\over 2}(h\ke(\pm \eps)+g\ke(\pm\eps)),\quad  
		x\ke^\pm=(z\ke+{h\ke(\pm\eps)-g\ke(\pm\eps)\over 2},\pm\eps)
	\end{gather}

	The proposition below demonstrates that the the non-concentrating property \eqref{non:con} holds
	provided the cross-sections of $T\ke$ (which is determined by the functions $g\ke$ and $h\ke$) satisfy  appropriate assumptions -- roughly speaking, the cross-section diameter should not change very rapidly along the passage.
	We denote\interdisplaylinepenalty=10000
	\begin{align}\notag
		G\e^+\ceq \eps^2\sup_{k\in\Z^{n-1}}\left(\|g\ke\|_{\LL^\infty(-\eps,\eps)}  \|1/g\ke\|_{\LL^\infty(-\eps,\eps)}(1+\|g\ke^\prime\|^2_{\mathsf{L}^\infty(-\eps,\eps)})\right)\\\label{Ge}
		+\,\eps\sup_{k\in\Z^{n-1}}\left(\|g\ke\|_{\LL^\infty(-\eps,\eps)}|\ln d\ke^+|\right),
		\\ \notag
		H\e^+\ceq \eps^2\sup_{k\in\Z^{n-1}}\left(\|h\ke\|_{\LL^\infty(-\eps,\eps)}  \|1/h\ke\|_{\LL^\infty(-\eps,\eps)}(1+\|h\ke^\prime\|^2_{\mathsf{L}^\infty(-\eps,\eps)})\right)\\\label{He}+\,\eps\sup_{k\in\Z^{n-1}}\left(\|h\ke\|_{\LL^\infty(-\eps,\eps)}|\ln d\ke^+|\right).
	\end{align}

	\begin{figure}[h]
		\begin{picture}(420,105)
			\includegraphics[width= 0.9\textwidth]{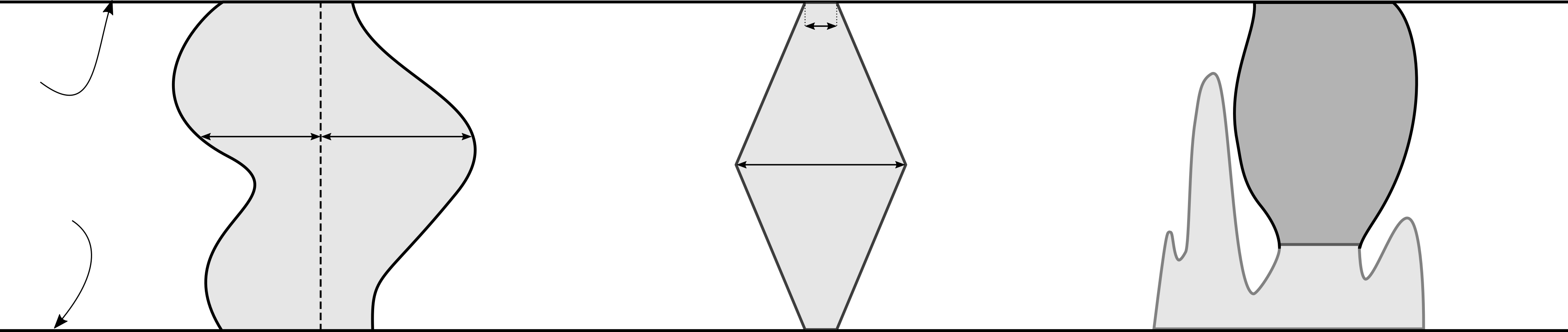}
			
			\put(-325,95){\text{(a)}}    
			\put(-195,95){\text{(b)}}    
			\put(-70,95){\text{(c)}}    
			
			\put(-395,68){$\Gamma\e^+$}
			\put(-390,33){$\Gamma\e^-$} 
			
			\put(-346,56){$_{g\ke(x_2)}$} 
			\put(-312,56){$_{h\ke(x_2)}$} 
			
			\put(-335,26){$_{T^L\ke}$} 
			\put(-312,26){$_{T^R\ke}$} 
			
			\put(-195,36){$_{2a\e}$} 
			\put(-195,72){$_{2d\e}$}
			
			\put(-70,50){$F\ke^+$}    
			\put(-70,9){$F\ke^-$}      
			
		\end{picture}
		\caption{(a) The passage of the form \eqref{Tke:gh}. The vertical dashed line connects the points $(z\ke,\eps)$ and $(z\ke,-\eps)$.  
			\\[1mm](b) The passage from Example~\ref{example:gh}.\\[1mm]
			(c) The hybrid passage described in Remark~\ref{rem:hyb}}\label{fig:passages2}
	\end{figure}
	
	\allowdisplaybreaks
	
	\begin{proposition}\label{prop:curved}
		Assume that   the assumptions \eqref{assump:1}, \eqref{assump:3} hold true,
		moreover,
		\begin{gather}
			\label{gh:prop}
			\exists C>0\ \forall \eps\in (0,\eps_0]\ \forall k\in\Z^{n-1}:\quad   \max\left\{{g\ke(\eps)\over h\ke(\eps)};\,{h\ke(\eps)\over g\ke(\eps)}\right\}\leq C,
			\\
			\label{GeHe}
			G\e^+\to 0,\ 
			H\e^+\to 0 \text{ as }\eps\to 0.
		\end{gather}
		Then \eqref{non:con} holds with 
		$$\zeta\e=C (G\e^++H\e^+)^{1/2}.$$
	\end{proposition}
	
	\begin{proof} 
		It follows easily from \eqref{dke:xke}, \eqref{gh:prop} that
		\begin{gather}\label{dgh}
			d\ke^+\leq C_1g\ke(\eps),\quad d\ke^+\leq C_2h\ke(\eps).
		\end{gather}
		
		We denote by $T\ke^L$ and $T\ke^R$ the subsets of $T\ke$ lying, respectively, left and right from the line 
		connecting the points $(z\ke,\eps)$ and $(z\ke,-\eps)$ (see  Figure~\ref{fig:passages2}\,(a)).
		Similarly, $D\ke^{+,L}$ and $D\ke^{+,R}$ are the subsets of $D\ke^+$ lying, respectively, left and right from the point $(z\ke,\eps)$.
		
		We introduce in $T\ke^R$ the new coordinates $(y_1,y_2)\in (0,1)\times (-\eps,\eps)$, which
		are related to the ``old'' Cartesian coordinates $(x_1,x_2)$ as follows,
		$$
		x_1=h\ke(y_2) y_1+z\ke,\quad x_2=y_2.
		$$
		Simple computations yield:
		\begin{gather}\label{normu}
			\|u\|^2_{\L(T\ke^R)}=\int_{-\eps}^\eps \int_{0}^1 |u(y_1,y_2)|^2 h\ke(y_2) \d y_1 \d y_2,\\
			\label{normu+}
			\|u\|^2_{\L(D\ke^{R,+})}= h\ke(\eps)\int_{0}^1 |u(y_1,\eps)|^2  \d y_1  ,\\\notag
			\|\nabla u\|^2_{\L(T\ke^R)}=\int_{-\eps}^\eps \int_{0}^1 \left({1+(h'(y_2)y_1)^2\over h(y_2)}\left|{\partial u\over \partial y_1}\right|^2+h\ke(y_2)\left|{\partial u\over \partial y_2}\right|^2\right.
			\\\label{normdu}\left.
			-2 h'(y_2)y_1\Re\left({\partial u\over \partial y_1}\overline{\partial u\over \partial y_2}\right)\right)   \d y_1 \d y_2.
		\end{gather}
		We re-write the expression \eqref{normdu} for $\|\nabla u\|^2_{\L(T\ke^R)}$ as follows:
		\begin{gather}\label{Me}
			\|\nabla u\|^2_{\L(T\ke^R)}=\int_{-\eps}^\eps \left(\int_{0}^1 {h\ke(y_2)\over 2+\|h\ke^\prime\|_{\LL^\infty(-\eps,\eps)}^2}\left|{\partial u\over \partial y_2}\right|^2+
			\la M\ke Du,Du\ra_{\C^2}
			\right)
			\d y_1\d y_2,\text{ where}
			\\   \notag
			Du\ceq \left(\begin{matrix}{\partial u\over \partial y_1}\\[2mm]{\partial u\over \partial y_2}\end{matrix}\right),\quad
			M\ke\ceq 
			\left(\begin{matrix}{1+(h\ke'(y_2)y_1)^2\over h\ke(y_2)} & -h\ke'(y_2)y_1\\-h\ke'(y_2)y_1 &  {h\ke(y_2) }
				{1+\|h\ke^\prime\|_{\LL^\infty(-\eps,\eps)}^2\over 2+
					\|h\ke^\prime\|_{\LL^\infty(-\eps,\eps)}^2}\end{matrix}\right).
		\end{gather}
		The top left entry of $M\ke$ is positive, moreover, taking into account that $|y_1|<1$, we get 
		$$
		\det M\ke= 
		{1+\|h\ke^\prime\|_{\LL^\infty(-\eps,\eps)}^2-(h'(y_2)y_1)^2\over 2+\|h\ke^\prime\|_{\LL^\infty(-\eps,\eps)}^2}
		\geq {1 \over 2+\|h\ke^\prime\|_{\LL^\infty(-\eps,\eps)}^2}
		>0
		$$
		Thus $\la M\ke Du,Du\ra_{\C^2}\ge 0$. Using this fact, we deduce the following estimate from \eqref{Me}:
		\begin{gather}\label{normdu:est}
			\int_{-\eps}^\eps \int_{0}^1  \left|{\partial u\over \partial y_2}\right|^2 \d y_1 \d y_2\le
			\|1/h\ke\|_{\LL^\infty(-\eps,\eps) }  \left(2+\|h\ke^\prime\|_{\mathsf{L}^\infty(-\eps,\eps)}^2\right)\|\nabla u\|^2_{\L(T\ke^R)} 
		\end{gather}
		Furthermore, from \eqref{normu} we obtain
		\begin{gather}\label{normu:est}
			\|u\|^2_{\L(T\ke^R)}\leq  \|h\ke\|_{\LL^\infty(-\eps,\eps) }\int_{-\eps}^\eps \int_{0}^1 |u(y_1,y_2)|^2  \d y_1 \d y_2.
		\end{gather}
		One has:
		\begin{gather}\label{y1y2:est}
			\int_{-\eps}^\eps \int_{0}^1  \left|u(y_1,y_2)\right|^2 \d y_1 \d y_2\le 
			4\eps\int_{0}^1  \left| u(y_1,\eps) \right|^2 \d y_1 +
			8\eps^2\int_{-\eps}^\eps \int_{0}^1  \left|{\partial u\over \partial y_2}(y_1,y_2)\right|^2 \d y_1 \d y_2 
		\end{gather}
		(the estimate \eqref{y1y2:est} follows easily from the 
		equality $\ds u(y_1,y_2)=u(y_1,\eps)-\int_{y_2}^{\eps}{\partial u\over\partial y_2}(y_1,\tau)\d\tau$).
		Combining \eqref{dgh}, \eqref{normu+}, \eqref{normdu:est}--\eqref{y1y2:est}, we arrive at the estimate
		\begin{align}\notag
			\|u\|_{\L(T\ke^R)}^2 
			&\leq C\|h\ke\|_{\LL^\infty(-\eps,\eps)} \Big({\eps(d\ke^+)^{-1}}\|u\|_{\L(D\ke^{R,+})}^2 
			\\\label{prop:est:R}
			&+ \eps^2   \|1/h\ke\|_{\LL^\infty(-\eps,\eps)}(1+\|h\ke^\prime\|^2_{\mathsf{L}^\infty(-\eps,\eps)})\big) \|\nabla u\|_{\L(T\ke^R )}^2\Big) .
		\end{align}
		Similarly, we obtain the estimate for the $\L$-norm of $u$ in $T\ke^L$:
		\begin{align}\notag
			\|u\|_{\L(T\ke^R)}^2 
			&\leq C\|g\ke\|_{\LL^\infty(-\eps,\eps)}\Big( {\eps(d\ke^+)^{-1}}\|u\|_{\L(D\ke^{L,+})}^2 
			\\\label{prop:est:L}
			&+ \eps^2  \|1/g\ke\|_{\LL^\infty(-\eps,\eps)}(1+\|g\ke^\prime\|^2_{\mathsf{L}^\infty(-\eps,\eps)})\big)\|\nabla u\|_{\L(T\ke^L )}^2\Big) .
		\end{align}
		
		Recall that the sets $P\ke^+$ are give by \eqref{Pke}.
		By Lemma~\ref{lemma:PFTS} (namely, we use the estimate \eqref{traceSigma:De} for 
		$E_\delta=P\ke^+$, $S_\delta=D\ke^+=\partial P\ke^+\cap\Gamma\e^+$), we get 
		\begin{gather}\label{prop:est:2}
			\|u\|_{\L(D\ke^+)}^2
			\leq C\left((d\ke^+)^{-1}\|u\|_{\L(P\ke^+)}^2+
			d\ke^+\|\nabla u\|_{\L(P\ke^+)}^2\right).
		\end{gather}
		Then, by virtue of Lemma~\ref{lemma:ring} (which we apply for $E_1=  P\ke^+$, $E_2=B^+\ke$, $\varkappa={4}$, $\ell=\rho\e$, $a_1=d\ke^+$, $a_2=\rho\ke$), we derive
		\begin{gather}\label{prop:est:3}
			\|u\|_{\L(P\ke^+)}^2 
			\leq C\left((d\ke^+)^2\rho\ke^{-2}\|u\|_{\L(B\ke^+)}^2+(d\ke^+)^2|\ln d\ke^+|\|\nabla u\|_{\L(B\ke^+)}^2\right).
		\end{gather}
		Finally (cf.~\eqref{Rke:est}), we have 
		\begin{gather}\label{prop:est:4}
			\|u\|_{\L(B\ke^+)}^2 
			\leq C\rho\ke \|u\|_{\H^1(R\ke^+)}^2.
		\end{gather}
		
		Combining \eqref{Ge}, \eqref{He}, \eqref{prop:est:R}--\eqref{prop:est:4} and taking into account that 
		$\rho\ke^{-1}\leq C|\ln d\ke^+|$ (this follows from \eqref{assump:3})  and 
		that the sets $T\ke\cup R\ke^+$, $k\in\Z^{n-1}$ are pairwise disjoint and belong to $\Omega\e$ (here we need \eqref{assump:1}),
		we arrive at the final estimate
		\begin{gather*}\label{prop:est:final}
			\|u\|_{\L(\cup_{k\in\Z^{n-1}}T\ke)}^2 \leq C(G\e+H\e)\|u\|^2_{\H^1(\cup_{k\in\Z^{n-1}}(T\ke\cup R\ke^+))}\leq 
			C(G\e+H\e)\|u\|^2_{\H^1(\Omega\e)}.
		\end{gather*}
		The proposition is proven.
	\end{proof}
	
	\begin{example}\label{example:gh}
		Let the
		functions $g\ke,\, h\ke$ be given by  
		\begin{gather}\label{h:ex}
			h\ke(t)=g\ke(t)=
			\begin{cases}
				{t-\eps\over \eps}(d\e-a\e)+{d\e},&0\le t\le \eps,\\
				{t+\eps\over \eps}(a\e-d\e)+{d\e},&-\eps\leq t\le 0,
			\end{cases}
		\end{gather}
		where $d\e,\,a\e>0$ and $d\e<a\e$; we have, in particular, $d\ke^+=d\ke^-=g\ke(\pm\eps)=h\ke(\pm\eps)=d\e$ and $g\ke(0)=h\ke(0)=a\e$ (see Figure~\ref{fig:passages2}(b)).
		Then  \eqref{gh:prop} holds with $C=1$, and  
		the assumption \eqref{GeHe} is fulfilled provided 
		\begin{gather*}
			a\e(a\e-d\e)^2d\e^{-1}\to 0, \ 
			\eps^2a\e d\e^{-1}\to 0,\ \eps a\e|\ln d\e|\to 0\text{ as }\eps\to 0.
		\end{gather*}
	\end{example}
	
	\begin{remark}\label{rem:GH}
		Assume that instead of \eqref{gh:prop} we have the condition
		\begin{gather*}
			\exists C>0\ \forall \eps\in (0,\eps_0]\ \forall k\in\Z^{n-1}:\quad   \max\left\{{g\ke(-\eps)\over h\ke(-\eps)};\,{h\ke(\eps)\over g\ke(\eps)}\right\}\leq C,
		\end{gather*}
		and also \eqref{GeHe} is replaced by 
		$$G\e^++H\e^-\to 0,$$
		where  $H\e^-$ is defined by \eqref{He}  with
		$d\ke^+$ being replaced by $d\ke^-$. 
		Then  \eqref{non:con} holds with 
		$\zeta\e=C (G\e^++H\e^-)^{1/2}.$
		The proof is similar to the one of Proposition~\ref{prop:curved}.
	\end{remark}
	
	\begin{remark}\label{rem:hyb}
		The results obtained in Subsections~\ref{subsec:6:2}-\ref{subsec:6:3} allows to treat also the hybrid passages consisting
		of two parts $F\ke^+$ and $F\ke^-$  (Figure~\ref{fig:passages2}\,(c)): 
		the part $F\ke^-$ has the form \eqref{Fke}, while    
		the upper part $F\ke^+$ is now given by
		$$
		F\ke^+ =  \left\{x=(x_1,x_2)\in\R^2:\ \xi\ke< x_2<\eps,\ -g\ke(x_2)<x_1-z\ke<h\ke(x_2)\right\},
		$$
		where $\xi\ke\in (-\eps,\eps)$, $z\ke \in\R$,  $g\ke,\,h\ke\in \W^{1,\infty}(\xi\ke,\eps)$, $1/g\ke,\,1/h\ke\in \LL^{\infty}(\xi\ke,\eps)$, $g\ke,\,h\ke>0$,  satisfying \eqref{gh:prop} and \eqref{GeHe} (with $\LL^{\infty}(\xi\ke,\eps)$ instead of $\LL^{\infty}(-\eps,\eps)$ in the definitions of $G\e$ and $H\e$); also assume that  \eqref{assump:1}, \eqref{assump:3} hold true.
		Tracing the proofs of Propositions~\ref{prop:hourglass} and \ref{prop:curved}, we conclude that
		for such hybrid passages the assumption \eqref{non:con} is fulfilled with 
		$\zeta\e=C\left(\eps+G\e+H\e\right)^{1/2}$ (in fact, this convergence rate can be improved, cf. Remark~\ref{rem:notoptimal}).

	\end{remark}
	
	\subsection{Passages not obeying non-concentrating property}
	\label{subsec:6:4}
	
	In this subsection we present examples of passages for which the property 
	\eqref{non:con} does not hold. Proposition~\ref{prop:hourglass} and \ref{prop:curved}
	suggest the way how to construct appropriate examples: passages should not be of a
	hourglass-like form and their cross-sections should vary rapidly.
	\smallskip
	
	In order to show that the property \eqref{non:con} does not hold,
	it is enough to find such a sequence $(v\e\in \H^1(\Omega\e))_{\eps\in (0,\eps_0]}$ that
	\begin{gather}\label{inf>0}
		\inf_{\eps\in (0,\eps_0]}\frac{\|v\e\|_{\L(\cup_{k\in\Z^{n-1}}T\ke)}}{\|v\e\|_{\H^1(\Omega\e)}}>0.
	\end{gather}
	
	\begin{figure}[h]
		\begin{picture}(320,125)
			\includegraphics[width= 0.7\textwidth]{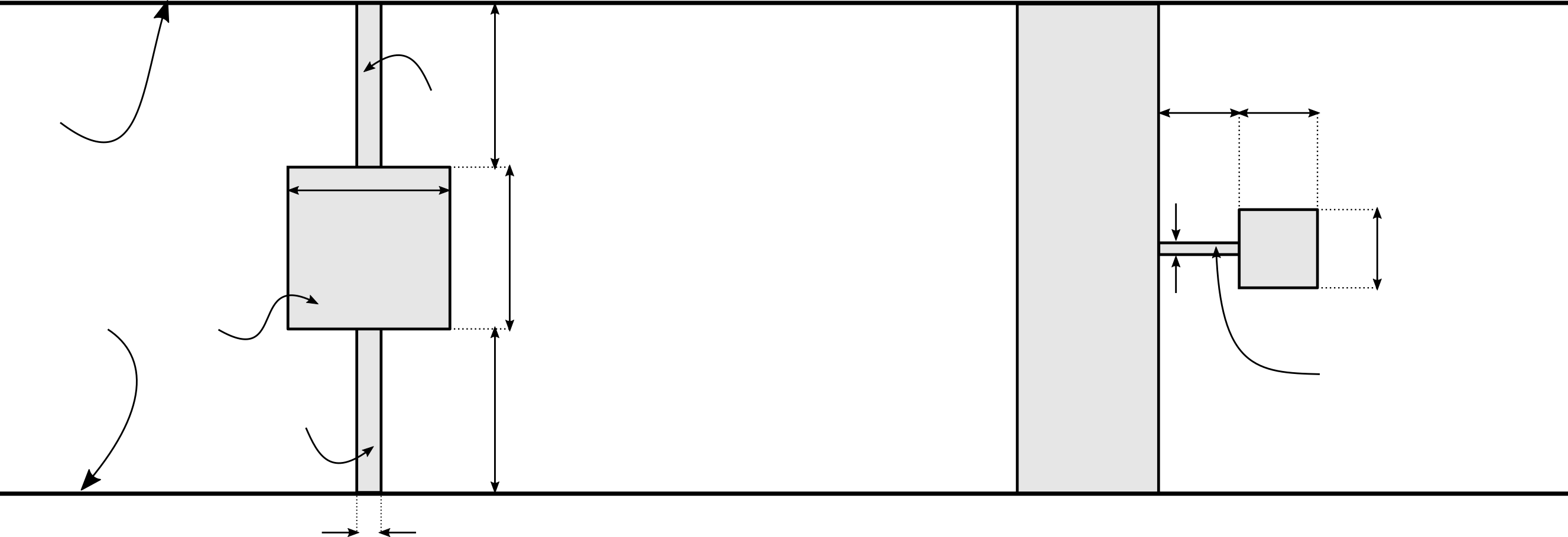}
			
			\put(-243,115){\text{(a)}}    
			\put(-105,115){\text{(b)}}    
			
			\put(-310,82){$\Gamma\e^+$}
			\put(-301,40){$\Gamma\e^-$}

			\put(-255,28){$_{T^-\e}$}
			\put(-228,84){$_{T^+\e}$} 
			\put(-275,45){$_{R\e}$} 
			
			\put(-209,55){$_{2\eps\over 3}$} 
			\put(-212,25){$_{2\eps\over 3}$} 
			\put(-212,88){$_{2\eps\over 3}$}
			\put(-241,64){$_{2\eps\over 3}$}
			
			\put(-241,-1){$_{_{2d\e}}$}

			\put(- 99,50){$^{T\e}$}
			\put(-62,50){$^{R\e}$}    
			\put(-49,32){$_{P\e}$} 
			\put(-37,51){$^{\xi\e}$} 
			\put(-62,84){$^{\xi\e}$} 
			\put(-75,84){$^{\xi\e}$} 
			\put(-80,73){$_{\xi\e^{3+\al}}$} 
			
		\end{picture}
		\caption{The examples of passages for which the non-concentrating property \eqref{non:con} is not fulfilled}\label{fig:passages3}
	\end{figure}
	
	As in the previous subsection, we restrict ourselves to the case $n=2$. For $n\ge 3$ one can easily construct appropriate examples by utilising the same idea.
	
	\begin{example}
		Let {one of the passages}, for example $T_{0,\eps}$, has the form
		\begin{gather}
			T_{0,\eps}=\mathrm{int}(\overline{T^{+}\e \cup R\e \cup T^{-}\e}),
		\end{gather}
		where 
		\begin{align*}
			T^-\e&\ceq\left\{x=(x_1,x_2)\in\R^2:\ |x_1|< d\e,\ -\eps< x_2 < -{\eps\over 3}\right\},
			\\ 
			T^+\e&\ceq\left\{x=(x_1,x_2)\in\R^2:\ |x_1|< d\e,\ {\eps\over 3}< x_2 < \eps\right\},
			\\
			R\e&\ceq\left\{x=(x_1,x_2)\in\R^2:\ |x_1|< {\eps\over 3},\ |x_2|< {\eps\over 3}\right\},
		\end{align*}
		with $d\e=\eps^{3+\al}$, $\al>0$; see Figure~\ref{fig:passages3}\,(a).
		The geometry of the remaining passages $T\ke$, $k\not=0$, plays no role.
		We define the function $v\e\in\H^1(\Omega\e)$ as follows,
		\begin{gather}
			v\e(x)=
			\begin{cases}
				1, & x\in R\e,\\
				{3\over 2\eps}(\eps-x_2), &x=(x_1,x_2)\in P\e^+,\\
				{3\over 2\eps}(\eps+x_2), &x=(x_1,x_2)\in P\e^-,\\
				0,&x\in \Omega\e\setminus T_{0,\eps}.
			\end{cases}
		\end{gather}
		Straightforward calculations  yield
		\begin{gather}\label{ve:norms}
			\|v\e\|_{\L(T_{0,\eps})}^2=
			{4\eps^2\over 9} + {8\eps^{4+\al}\over 9},\quad
			\|\nabla v\e\|_{\L(T_{0,\eps})}^2=  {6\eps^{2+\al}}. 
		\end{gather}
		From  \eqref{ve:norms} and the fact that 
		$\supp v\e\subset \overline{T_{0,\eps}}$, we easily conclude 
		\eqref{inf>0} Q.E.D.
		
	\end{example}
	
	The example above had one disadvantage -- we impose quite strong restrictions
	on the behaviour of the diameter $d\e$ of the passage necks as $\eps\to 0$ (namely, $d\e=\eps^{3+\al}$).
	The example below demonstrates that, in fact, any passage can be perturbed by 
	an arbitrarily small bump in such a way that   the fulfillment of \eqref{non:con} is broken.
	
	\begin{example} 
		Let $T_{0,\eps}$ consist of a straight passage with attached small
		bump; this bump consists of a square being connected to that passage via
		a narrow rectangle.  Namely, we assume that
		$$T_{0,\eps}=\mathrm{int}(\overline{T\e\cup P\e \cup R\e}),$$
		where         
		\begin{align*}
			T\e&\ceq\left\{x=(x_1,x_2)\in\R^2:\ |x_1|< d\e,\ -\eps< x_2 < \eps\right\},\\
			P\e&\ceq\left\{x=(x_1,x_2)\in\R^2:\ d\e<x_1< d\e+\xi\e,\ -(\xi\e)^{3+\al}/2 < x_2 < (\xi\e)^{3+\al}/2\right\},\\
			R\e&\ceq\left\{x=(x_1,x_2)\in\R^2:\ d\e+\xi\e<x_1< d\e+2\xi\e,\ -\xi\e/2 < x_2 < \xi\e/2\right\},
		\end{align*}
		where $\al>0$ and $\xi\e\ll\eps $. The passage $T_{0,\eps}$ is depicted on Figure~\ref{fig:passages3}\,(b).
		The remaining passages $T\ke$, $k\not=0$ can be choose arbitrarily.
		In fact, the geometry of the $T\e$ also plays no essential role: instead
		of the straight passage $T\e$ one can take a passage with an arbitrary geometry and then attach a small bump being congruent to $\overline{B\e\cup R\e}$.
		
		We define  $v\e\in\H^1(\Omega\e)$ by
		\begin{gather*}
			v\e(x)=
			\begin{cases}
				1, & x\in R\e,\\
				{x_1-d\e\over \xi\e}, &x=(x_1,x_2)\in P\e^+,\\
				0,&x\in \Omega\e\setminus (R\e\cup P\e).
			\end{cases}
		\end{gather*}
		One has 
		$\|v\e\|_{\L(R\e\cup P\e)}^2=(\xi\e)^2 +{1\over 3}(\xi\e)^{4+\al}$ and 
		$\|\nabla v\e\|_{\L(R\e\cup P\e)}^2=  (\eta\e)^{2+\al}$, 
		whence (taking into account that
		$\supp v\e\subset \overline{R\e\cup P\e}$), we obtain the desired property obtain
		\eqref{inf>0}.
	\end{example}
	
	\section{Calculation of the quantities $\mathscr{C}\ke$ and verification of the assumption \eqref{assump:main}}
	\label{sec:7}
	
	In this section we present several examples for which the asymptotic behaviour of the quantities  $\mathscr{C}\ke$
	can be explicitly calculated, and the validity of the assumption \eqref{assump:main} can be verified.
	
	\subsection{Conditions leading the function  ${\mu}$ to vanish} 
	\label{subsec:7:1}
	In this subsection we formulate the conditions which guarantee that  
	\eqref{assump:main} holds with ${\mu}\equiv 0$. Roughly speaking, this  happens if either the diameters $d\ke^\pm$ of the passages necks are very small or the passages volumes are very large.\smallskip
	
	We have the following useful estimate on $\mathscr{C}\ke$.
	
	\begin{lemma}\label{lemma:Ckeest}
		One has the estimate
		\begin{gather}\label{Ckeest}
			\mathscr{C}\ke\leq C \min\left\{ \gamma\ke^-\rho\ke^{n-1},\,  \gamma\ke^+\rho\ke^{n-1},\, 
			{|T\ke| \eps^{-2}}\right\},
		\end{gather}
		where the constant $C$ depends only on the dimension $n$.
	\end{lemma}
	
	\begin{proof}
		We have already proven  (see \eqref{UV}--\eqref{nablaV}) that
		\begin{gather}\label{Ckeest0}
			\mathscr{C}\ke\leq C \gamma\ke^\pm\rho\ke^{n-1}.
		\end{gather}
		Next, we define  the function $w\ke\in \H^1(G\ke)$ via
		\begin{align*}
			w\ke (x)&=
			\begin{cases}
				\ds{x_n+\eps\over 2\eps},& x=(x',x_n)\in T\ke,\\
				1,&x\in B\ke^+, \\
				0,&x\in B\ke^-.
			\end{cases}
		\end{align*} 
		The function  $ w\ke$ is equal  to $1$ on $S\ke^+$ and is equal to $0$ on $S\ke^-$.
		Then, due to \eqref{Cke}, we get
		\begin{gather}\label{Ckeest1}
			\mathscr{C}\ke\leq \|\nabla w\ke\|^2_{\L(G\ke)}.
		\end{gather}
		Direct computations yield
		\begin{gather}\label{Ckeest2}
			\|\nabla w \ke\|^2_{\L(G\ke)}= \eps^{-2}{|T\ke|}  /4,
		\end{gather}
		where   $C$ depends only on  $n$.
		The desired estimate \eqref{Ckeest} follows from \eqref{Ckeest0}--\eqref{Ckeest2}.
	\end{proof}
	
	Now, we are in position to formulate the main result of this subsection.

	\begin{proposition}\label{prop:ga0}
		Let  
		$$
		\lim_{\eps\to 0}\sup_{k\in\Z^{n-1}}\gamma\ke^+=0\quad \vee\quad \lim_{\eps\to 0}\sup_{k\in\Z^{n-1}}\gamma\ke^-=0\quad \vee\quad \lim_{\eps\to 0}\sup_{k\in\Z^{n-1}}{|T\ke| \eps^{-2}\rho\ke^{1-n}}=0.
		$$
		Then \eqref{assump:main} holds with ${\mu}\equiv 0$ and 
		\begin{gather}\label{kappae:formula1}
			\kappa\e= C\cdot{\min\left\{\sup_{k\in\Z^{n-1}}\gamma\ke^+,\, \sup_{k\in\Z^{n-1}}\gamma\ke^-,\,\sup_{k\in\Z^{n-1}}{|T\ke| \eps^{-2}\rho\ke^{1-n}}\right\}}.
		\end{gather}
	\end{proposition}
	
	\begin{proof}
		Let $g\in\H^{2}(O\setminus \Gamma)$ and $h\in \H^{1}(O\setminus \Gamma)$. 
		Using the estimates \eqref{abs:fke}, \eqref{Ckeest}, we get
		\begin{align}\notag
			&\left|\suml_{k\in\Z^{n-1}}\mathscr{C}\ke
			\left(\la g \ra_{B\ke^+}-\la g \ra_{B\ke^-}\right)
			\left(\la \overline{h} \ra_{B\ke^+}-\la \overline{h} \ra_{B\ke^-}\right)\right|
			\\\notag
			&\leq
			2\left(\suml_{k\in\Z^{n-1}} \capty\ke\left(|\la g \ra_{B\ke^+}|^2+|\la g \ra_{B\ke^-}|^2\right)\right)^{1/2} 
			\left(\suml_{k\in\Z^{n-1}} \capty\ke\left(|\la h \ra_{B\ke^+}|^2+|\la h \ra_{B\ke^-}|^2\right)\right)^{1/2} 
			\\\notag 
			&\leq
			C_1\left(\suml_{k\in\Z^{n-1}} {\capty\ke \over  \rho\ke^{n-1}}\left(\|g\|^2_{\H^1(R\ke^+)} + \|g\|^2_{\H^1(R\ke^-)}\right) \right)^{1/2} 
			\left(\suml_{k\in\Z^{n-1}} {\capty\ke \over  \rho\ke^{n-1}}\left(\|h\|^2_{\H^1(R\ke^+)}+\|h\|^2_{\H^1(R\ke^-)}\right)\right)^{1/2}	 
			\\\notag 
			&\leq
			C_1\sup_{k\in\Z^{n-1}}{\capty\ke \over  \rho\ke^{n-1}} \|g\|_{\H^1(O\setminus \Gamma)} \|h\|_{\H^1(O\setminus \Gamma)}
			\leq
			C_2\kappa\e\|g\|_{\H^2(O\setminus \Gamma)} \|h\|_{\H^1(O\setminus \Gamma)} ,
		\end{align}
		where $\kappa\e$ is defined by \eqref{kappae:formula1}.
		The proposition is proven.
	\end{proof}
	
	\subsection{Periodically distributed straight passages}
	\label{subsec:7:2}
	In this subsection we consider the case of 
	periodically distributed and straight passages.
	In particular, we complement the results of \cite{DelV87}, where only the case $n\ge 3$ was addressed in full, while the  calculation of $\mu$ in the  case $n=2$, $p>0$, $0<q\le \infty$ remained to be open problem. Below we will fill this gap (see Lemma~\ref{lemma:Cke2}).\smallskip
	
	Let the passages $T\ke$, $k\in\Z^{n-1}$ be given by
	$$
	T\ke\ceq\left\{x=(x',x_n)\in\R^n:\ x' -   \rho\e k \in  d\e \D,\ |x_n| < \eps\right\},\ d\e<\rho\e/2,\ \rho\e\to 0, 
	$$
	where $\D\subset\R^{n-1}$ is a bounded Lipschitz domain such that  
	the smallest ball containing $\D$ is centered at $0$;
	without loss of generality, we assume that the radius of this ball is equal to $1$.
	Thus, we have 
	$$x\ke^\pm=(\rho\e k,\pm\eps),\quad d\ke^\pm =  d\e,\quad 
	D\ke^\pm=\{x=(x',x_n)\in\R^n:\ x'\in\D,\ x_n=\pm \eps\}.$$
	Furthermore, we assume that
	\begin{gather}\label{pe}
		p\e\leq C,\text{ where }	p\e\ceq (\GG(d\e))^{-1}\rho\e^{1-n}.
	\end{gather}  
	Evidently, conditions \eqref{assump:0}, \eqref{assump:1},   \eqref{assump:3}
	hold with $$\rho\ke \ceq \rho\e/2.$$ We have
	\begin{gather}\label{gaga}
		\ga\ke^\pm=2^{n-1}  p\e.
	\end{gather}   
	Also, the property \eqref{non:con}
	is valid with $\al\e=C\eps^{1/2}$ (see Subsection~\ref{subsec:6:1}).

	We assume that the following limits exist:
	\begin{align}\label{pq:limsp}
		p&\ceq \lim_{\eps\to 0}p\e,\ p<\infty\\ \label{pq:limsq}
		q&\ceq \lim_{\eps\to 0}q\e,\text{ where }q\e\ceq d\e^{n-1}\eps^{-1}\rho\e^{1-n},\ q\in [0,\infty]
	\end{align}
	(the restriction $p<\infty$ follows from \eqref{pe}).
	If $p=0$, then, by virtue of   Proposition~\ref{prop:ga0} and \eqref{gaga},
	condition \eqref{assump:main} holds with ${\mu}\equiv 0$.  
	The limit $q$ can be infinite. If  $q=0$, then  \eqref{assump:main} holds with ${\mu}\equiv 0$ -- again by virtue of   Proposition~\ref{prop:ga0} and the equality $q\e=|T\ke|\eps^{-2}\rho\e^{1-n}(2|\D|)^{-1}$.
	Therefore, in the following we address only the case
	\begin{gather}	\label{pq:lims+}
		0<p<\infty\quad \wedge\quad 0<q\leq \infty.
	\end{gather}

	\subsubsection{Calculation of $\capty\ke$}
	
	Since the passages $T\ke$ are identical (up to a rigid motion), it is enough to 
	calculate $\capty\ke$ only for $k=0$.  
	We introduce the sets
	\begin{align*}
		G\oe^+ \ceq G\oe\cup\Xi^\pm,\quad 	  T\oe^+  \ceq T\oe\cup\Xi^\pm,\quad
		S\oe \ceq T\oe\cap\Gamma
	\end{align*}
	(recall that ${\Xi}^\pm$ are the half-spaces lying above and below $\Gamma$, see \eqref{Xi0},
	the set $G\oe$ is defined by \eqref{BSG}).
	
	Let $U\oe$ be the solution to \eqref{BVP:U:1}--\eqref{BVP:U:4}
	for $k=0$. We also introduce the function $$U\oe^+\ceq 2\left(1-U\oe\right)\restr_{G\oe^+}\in \H^1(G\oe^+).$$  Using the  symmetry of   $G\oe$ with respect to the hyperplane $\Gamma$, one can easily show
	that 
	\begin{gather*}
		U\oe(x',x_n) = 1 - U\oe(x',-x_n),\ x=(x',x_n)\in G\oe,
	\end{gather*}
	whence, in particular, we get $U\oe =1/2$ on $S\oe$.
	Consequently, $U\oe^+$ is the solution 
	to   
	\begin{align} 
		\label{Uoe:BVP:1}\Delta U\oe^+ =0& \text{ in }G\oe^+,\\ 
		\label{Uoe:BVP:2}U\oe^+ =0 &\text{ on }S\oe^+,\\
		\label{Uoe:BVP:3}U\oe^+  =1 &\text{ on }S\oe,\\ 
		\label{Uoe:BVP:4}\ds{\partial U\oe^+\over\partial\nu}=0&\text{ on }\partial G\oe^+\setminus(S\oe^+\cup S\oe),
	\end{align} 
	and 
	\begin{gather}
		\label{Ueo:eq}
		\capty\oe={1\over2}\|\nabla U\oe^+\|^2_{\L(G\oe^+)},
		\\		
		\label{Ueo:inf}
		\forall V\in \mathscr{U}\oe^+:\ 
		\|\nabla U\oe^+\|^2_{\L(G\oe^+)} \leq \|\nabla V\|^2_{\L(G\oe^+)},
	\end{gather}
	where $\mathscr{U}^+\oe\ceq \left\{U\in  \H^1({G\oe^+}):\ U\restr_{S\oe^+}=0,\ U\restr_{S\oe}={1}\right\}$.\smallskip

	We begin from  $n=2$. In this case we have $\D=(-1,1)$, and   (cf. $\eqref{pe}, \eqref{pq:limsp}, \eqref{pq:limsq}, \eqref{pq:lims+}$):
	\begin{gather}\label{pq:lims2}
		p\e=|\ln d\e|^{-1}\rho\e^{-1} \to p\in (0,\infty),\quad
		q\e={d\e \eps^{-1}\rho\e^{-1}}\to q\in (0,\infty].
	\end{gather}

	\begin{lemma}\label{lemma:Cke2}
		Let $n=2$. Then one has:
		\begin{gather}\label{Cke:asympt2}
			\mathscr{C}\oe ={\pi p\over 2+\pi pq^{-1}}\rho\e\left( 1+ \mathcal{O}( \rho\e|\ln \rho\e| +|q\e^{-1}-q^{-1}|+|p\e-p| )\right)
		\end{gather}
		(hereinafter we set $q^{-1}=0$ for $q=\infty$).
	\end{lemma}
	
	\begin{proof}
		
		Recall that $x\oe^+=(0,\eps)$, and
		\begin{align*}
			D\oe^+ &= \left\{x=(x_1,x_2)\in\R^2:\  |x_1|<  d\e ,\ x_2=\eps\right\},\\
			P\oe^+ &= \left\{x=(x_1,x_2)\in\R^2:\  |x_1^2+x_2^2|^{1/2}<  d\e ,\ x_2>\eps\right\},
		\end{align*}
		i.e. $P\oe^+$ is a half-disc of the radius $d\e$, and $D\oe^+$ is a straight part of its boundary. We denote by $\wt D\oe^+$ the remaining part of $\partial P\oe^+$ (a semicircle):
		\begin{align*}
			\wt D\oe^+ &\ceq \partial P\oe^+\setminus D\oe^+
		\end{align*}

		Our goal is to construct a suitable approximation for the solution $U^+\oe$ to the problem \eqref{Uoe:BVP:1}--\eqref{Uoe:BVP:4}.
		Namely, we define the function $V\oe^+: G\oe^+\to\R$ by
		\begin{gather*}
			V\oe^+(x)=
			\begin{cases}
				a\e^+ \ln |x| + b\e^+, &x\in B\oe^+\setminus P\oe^+ ,\\
				c\e^+ ,& x\in P\oe^+,\\
				a\e x_2 +b\e, & x=(x_1,x_2)\in T\oe^+,
			\end{cases}
		\end{gather*}
		where $a\e^+$, $b\e^+$, $c\e^+$, $a\e$, $b\e$ are constants. 
		We choose these constant in such a way that
		the function $V\oe$ obeys the following properties:
		\begin{itemize}
			
			\item $V\oe^+$ equals $0$ on $S\oe^+$ and equals $1$ on $S\oe$.
			These assumptions lead to
			\begin{gather}\label{system:1}
				a\e^+\ln (\rho\e/2) +b\e^+ =0,\quad
				b\e =1.
			\end{gather}
			\item  $V\oe^+$ is a continuous function, which implies
			\begin{gather}\label{system:2}
				a\e^+ \ln d\e + b\e^+ = c\e^+ =  a\e\eps + b\e.
			\end{gather}
			\item The jump of the normal derivative of $V\oe^+$
			across $\wt D\oe^+$ times $|\wt D\oe^+|$ is equal
			to the jump of the normal derivative of $V\oe^+$
			across $D\oe^+$ times $|D\oe^+|$, whence we get   
			\begin{gather}\label{system:3} 
				{\pi}a\e^+   = 2  d\e a\e.
			\end{gather}
			
		\end{itemize}
		Note that the properties \eqref{system:1}--\eqref{system:2}  imply $ V\oe^+\in \mathscr{U}\oe^+$, whence, due to \eqref{Ueo:inf}, we get  
		\begin{gather}\label{U<V}
			\|\nabla U\oe^+\|^2_{\L(G\oe^+)}\leq \|\nabla V\oe^+\|^2_{\L(G\oe^+)}.
		\end{gather}	
		From \eqref{system:1}--\eqref{system:3} one can easily deduce the formulae
		\begin{gather}\label{alalal}
			a\e=-{1\over {2 d\e\over \pi}(|\ln d\e| - |\ln (\rho\e/2)|)+ \eps},\quad 
			a\e^+ = {2a\e  d\e\over \pi}        
		\end{gather}
		(the remaining constants are not required for further calculation), and then
		\begin{align}\notag
			\|\nabla V\oe^+\|^2_{\L(G\oe^+)}&=
			\|\nabla V\oe^+\|^2_{\L(B\oe^+\setminus \overline{P\oe^+})}+\|\nabla V\oe^+\|^2_{\L(T\oe^+)}
			\\\notag
			&=
			\pi (a\e^+)^2\left(|\ln d\e| - |\ln (\rho\e/2)|\right)+2\eps d\e a\e^2 
			\\\notag
			&=
			{2\pi  }|\ln d\e|^{-1}\left(2+ \pi p\e q\e^{-1}-{2\ln(\rho\e/2)\over\ln d\e}\right)^{-1}
			\\\label{Voe:est}
			&=
			{2\pi p\over 2+ \pi pq^{-1}  }\rho\e \left(1+\mathcal{O}(\rho\e|\ln\rho\e|+|q\e^{-1}-q^{-1}|+|p\e-p|)\right).
		\end{align} 
		
		Now, we represent $U\oe^+$ in the form
		\begin{gather}\label{Ueo:repres}
			U\oe^+ = V\oe^++W\oe^+,
		\end{gather}
		and estimate the reminded $W\oe^+$. 
		Using Lemma~\ref{lemma:PFTS} (namely, we use the estimate \eqref{meandiff:est} for $E_\delta=P\oe^+$ and
		$S_\delta= \wt D\oe^+$), we get  
		\begin{gather}\label{meanmean:1}
			\left|\la W\oe^+\ra_{\wt D\oe^+}-\la W\oe^+\ra_{P\oe^+}\right|
			\leq  C \|\nabla W\oe^+\|_{\L(P\oe^+)}.
		\end{gather}
		Also, again by  Lemma~\ref{lemma:PFTS} (now, we use   \eqref{meandiff:est} for $E_\delta=P\oe^+$ and
		$S_\delta=  D\oe^+$), we obtain
		\begin{gather}\label{meanmean:2}
			\left|\la W\oe^+\ra_{  D\oe^+}-\la W\oe^+\ra_{P\oe^+}\right|
			\leq  C \|\nabla W\oe^+\|_{\L(P\oe^+)}.
		\end{gather}
		Combining \eqref{meanmean:1} and \eqref{meanmean:2} we arrive 
		at the estimate
		\begin{gather}\label{meanmean}
			\left|\la W\oe^+\ra_{\wt  D\oe^+}-\la W\oe^+\ra_{  D\oe^+}\right|
			\leq  C \|\nabla W\oe^+\|_{\L(P\oe^+)}.
		\end{gather}
		
		From \eqref{U<V} and \eqref{Ueo:repres} we conclude
		\begin{align}\notag
			\|\nabla W\oe^+\|^2_{\L(G\oe^+)} &\leq -2(\nabla W\oe^+,\nabla V\oe^+)_{\L(G\oe^+)}\\\label{Woe:start}
			&=
			-2(\nabla W\oe^+,\nabla V\oe^+)_{\L(B\oe^+\setminus \overline{P\oe^+})}
			-2(\nabla W\oe^+,\nabla V\oe^+)_{\L(T\oe^+)}.
		\end{align}
		Then, integrating by parts in each term  in right-hand-side of \eqref{Woe:start} and using
		\eqref{system:3}, \eqref{meanmean} and the properties
		\begin{gather*}
			\Delta V\oe^+=0\text{ on }T\oe^+\cup (B\oe^+\setminus\overline{P\oe^+}),\quad W\oe^+=0\text{ on }S\oe^+\cup S\oe,
			\quad
			{\partial V\oe^+\over\partial \nu}=0\text{ on }\partial G\oe^+\setminus(  S\oe^+ \cup   S\oe),
		\end{gather*}
		we obtain
		\begin{align}\notag
			\|\nabla W\oe^+\|^2_{\L(G\oe^+)} 
			&\leq 2\left|(\nabla W\oe^+,\nabla V\oe^+)_{\L(G\oe^+)}\right|=
			2\left|\int_{\wt D\oe^+} W\oe^+ {\partial V\oe^+\over\partial |x|}\d s
			- \int_{D\oe^+} W\oe^+ {\partial V\oe^+\over \partial x_2}\d s\right|
			\\\notag
			&\leq
			4 a\e d\e\left|\la W\oe^+\ra_{\wt D\oe^+}-\la W\oe^+\ra_{D\oe^+}\right|
			\\\label{Woe:next}
			&\leq  C a\e d\e\|\nabla W\oe^+\|_{\L(P\oe^+)}\leq 
			Ca\e  d\e\|\nabla W\oe^+\|_{\L(G\oe^+)}.
		\end{align}	
		Dividing \eqref{Woe:next} by $\|\nabla W^+\oe\|_{\L(G\oe^+)}$  and taking into 
		account that $a\e  d\e=\mathcal{O}(\rho\e)$ (this follows easily from \eqref{pq:lims2}, \eqref{alalal}),
		we get
		\begin{align}\label{Woe:est}
			\|\nabla W^+\oe\|_{\L(G\oe^+)}=\mathcal{O}(\rho\e).
		\end{align}
		and, consequently (again using \eqref{Woe:next} and $a\e  d\e=\mathcal{O}(\rho\e)$), we have 
		\begin{align}\label{WoeVoe:est}
			\left|(\nabla W\oe^+,\nabla V\oe^+)_{\L(G\oe^+)}\right|=\mathcal{O}(\rho\e^2).
		\end{align}
		Combining   \eqref{Voe:est}, \eqref{Ueo:repres}, \eqref{Woe:est}, \eqref{WoeVoe:est} we get
		\begin{align*}
			\|\nabla U\oe^+\|^2_{\L(G\oe^+)}& = \|\nabla V\oe^+\|^2_{\L(G\oe^+)}+\|\nabla W\oe^+\|^2_{\L(G\oe^+)}+2(\nabla W\oe^+,\nabla V\oe^+)_{\L(G\oe^+)}\\
			&={2\pi p\over  2+ \pi p q^{-1}  }\rho\e \left(1+\mathcal{O}(\rho\e|\ln\rho\e|+|q\e^{-1}-q^{-1}|+|p\e-p|)\right),
		\end{align*}
		whence, taking into account \eqref{Ueo:eq}, we get the   equality \eqref{Cke:asympt2}. The lemma is proven.
	\end{proof}
	
	Let us now proceed to the dimension $n\ge 3$.
	Recall that (cf. \eqref{pq:limsp}, \eqref{pq:limsq}, \eqref{pq:lims+}):
	\begin{gather}\label{pq:lims3+}
		p\e=d\e^{n-2}\rho\e^{1-n} \to p>0,\quad
		q\e={d\e^{n-1} \eps^{-1}\rho\e^{1-n}}\to q<\infty.
	\end{gather}
	We start from the case $q=\infty$. We introduce the compact set 
	\begin{gather}\label{DD}
		{{D}}\ceq\{x=(x',x_n)\in \R^n:\ x'\in\overline{\D},\ x_n=0\}\subset\R^n.
	\end{gather}
	Let $\mathrm{cap}({{D}})$ be the Newton capacity of the set ${{D}}$, i.e.
	\begin{gather*}
		\mathrm{cap}({{D}})=\|\nabla H_D\|^2_{\L(\R^n)},
	\end{gather*}
	where $H_D(x)$ be the solution to the problem
	\begin{align}\label{HD:BVP:1}
		\Delta H_D = 0&\text{ in }\R^n\setminus{{D}},
		\\\label{HD:BVP:2}
		H_D= 1&\text{ on }{{D}},\\\label{HD:BVP:3}
		H_D\to 0&\text{ as }|x|\to \infty.
	\end{align}
	Note that $\mathrm{cap}(D)>0$ despite the Lebesgue measure of $D$ is zero.
	Standard regularity theory yield $H_D\in C^\infty(\R^n\setminus {{D}})$, furthermore, 
	$H_D(x)$ is symmetric with respect to the hyperplane $\Gamma$, which follows easily from ${{D}}\subset\Gamma$.
	Hence
	\begin{gather}
		\mathrm{cap}({{D}})=2\|\nabla H_D\|^2_{\L(\Xi^+)},\label{cap:halfspace}
		\qquad
		{\partial H_D\over \partial x_n}=0\text{ on }\Gamma\setminus {{D}} . 
	\end{gather} Finally, the function $H_D$
	satisfies the estimates 
	\begin{gather}\label{Hests}
		|H_D(x)|\leq C|x|^{2-n},\  |\nabla H_D(x)|\leq C|x|^{1-n}\text{ for }x\in \R^n\setminus\B(2 ,0)
	\end{gather}
 The proof can be carried out similarly to the proof of \eqref{lemma:U:est} and \eqref{lemma:dU:est}; for an alternative proof  see, e.g., \cite[Lemma~2.4]{MK06}.

	\begin{lemma}\label{lemma:Cke3}
		Let $n\ge 3$ and let $q=\infty$. Then
		\begin{gather}\label{Cke:asympt3}
			\capty\oe={p\over 4} \mathrm{cap}({{D}})\rho\e^{n-1}\left(1+ \mathcal{O}(\rho\e+q\e^{-1}+|p\e-p|) \right).
		\end{gather}
		
	\end{lemma}
	
	\begin{proof}
		This time we construct the approximation $V\oe^+ \in\mathscr{U}\oe^+ $ for the function $U\oe^+$ as follows:
		\begin{gather*}
			V\oe^+(x)=
			\begin{cases}
				H\oe^+(x)\psi\oe^+(x),&x\in B\oe^+,\\[1mm]
				1,& x\in T\oe^+.
			\end{cases}
		\end{gather*}
		Here $$H\oe^+(x)\ceq H_D((x-x\oe^+)d\e^{-1}),\quad\psi\oe^+\ceq \phi\left({2|x-x\oe^+|  \rho\oe^{-1}}\right)=\phi\left({4|x-x\oe^+|  \rho\e^{-1}}\right),$$ where
		$\phi:\R\to\R$ is a smooth cut-off function satisfying \eqref{phi:prop}. 
		It is easy to see that $V\oe^+\in \mathscr{U}\oe^+$, whence
		\eqref{U<V} holds true. 
		Recall that the set $\Xi\e^+$  is defined by \eqref{Xie}.
		Taking into account  \eqref{cap:halfspace}, we easily  get
		\begin{gather}\label{Voe:est+}
			\|\nabla V\oe^+\|^2_{\L(G\oe^+)}=
			\|\nabla (H\oe^+ \psi\oe^+)\|^2_{\L(\Xi\e^+)}=
			\frac12{d\e^{n-2}\mathrm{cap}({{D}})  }+\Delta\e =
			\frac12{p\e\rho\e^{n-1}\mathrm{cap}({{D}})  }+ \Delta\e,
		\end{gather}
		where   the reminder $\Delta\e$ is as follows,
		\begin{gather*}
			\Delta\e=2(\nabla   H\oe^+,\nabla (  H\oe^+(\psi\oe^+-1)))_{\L(\Xi\e^+)}+ \|\nabla (  H\oe^+ (\psi\oe^+-1))\|^2_{\L(\Xi\e^+)}.
		\end{gather*}
		Using the properties
		\begin{gather}\label{psi0:prop}
			|\psi\oe^+|\le1,\quad 
			|\nabla \psi\oe^+|\le C\rho\e^{-1},\quad
			\supp(\psi\oe^+-1)  \subset \R^n\setminus\wh P\oe^+, 
			\quad
			\supp(\nabla \psi\oe^+)  \subset B\oe^+\setminus\wh P\oe^+, 
		\end{gather}
		the estimates 
		\begin{gather}\label{Hests:d}
			|H\oe^+(x)|\leq Cd\e^{n-2}|x|^{2-n},\  
			|\nabla H\oe^+(x)|\leq Cd\e^{n-2}|x|^{1-n}\text{ for }x\in \R^n\setminus \wh P\oe^+
		\end{gather}
		(they follow from \eqref{Hests}) and \eqref{pq:lims3+}, we easily obtain
		\begin{gather}\label{deltaoe:est+}
			|\Delta\e|  = \mathcal{O}(\rho\e^n).
		\end{gather}
		
		Now, we represent $U\oe^+$ in the form \eqref{Ueo:repres}
		and estimate the reminded $W\oe^+$.
		Using \eqref{U<V} and the facts that 
		$V\oe^+$ is constant on $T\oe^+$, $W\oe^+=0$ on $S\oe^+$,  
		$\psi\oe^+=1$ in a neighborhood of $D\oe^+$, ${\partial \psi\oe^+\over\partial x_n}=0$ on $\Gamma\e^+$,
		${\partial H\oe^+\over\partial x_n}=0$ on $\Gamma\e^+\setminus D\oe^+$,
		we get
		\begin{align} 
			\|\nabla W\oe^+\|^2_{\L(G\oe^+)} \leq 
			-2(\nabla W\oe^+,\nabla V\oe^+)_{\L(B\oe^+)}=
			2(W\oe^+,\Delta (H\oe^+\psi\oe^+))_{\L(B\oe^+)}+
			\int_{D\oe^+}{\partial H\oe^+\over\partial x_n}W\oe^+\d x'. \label{byparts:3}
		\end{align}
	
		One has:
		\begin{gather*}
			(W\oe^+,\Delta (H\oe^+\psi\oe^+))_{\L(B\oe^+)}
			=
			2(W\oe^+,\nabla H\oe^+\cdot \nabla \psi\oe^+)_{\L(B\oe^+)}+
			2(W\oe^+, H\oe^+\Delta \psi\oe^+)_{\L(B\oe^+)},
		\end{gather*}
		whence, using  \eqref{psi0:prop}, \eqref{Hests:d} and 
		$$
		|\Delta \psi\oe^+|\le C\rho\e^{-2},\quad
		(\supp\nabla\psi\oe^+\cup\supp\Delta\psi\oe^+)\cap\Xi\e^+\subset 
		\overline{B\oe^+\setminus \wh P\oe^+},
		$$
		we obtain
		\begin{multline} 
			|(W\oe^+,\Delta (H\oe^+\psi\oe^+))_{\L(B\oe^+)}|
			\le C\|W\oe^+\|_{\L(B\oe^+)}
			\left(\rho\e^{-1}\|\nabla H\oe^+\|_{\L(B\oe^+\setminus 
				\overline{\wh P\oe^+})}+\rho\e^{-2}\|H\oe^+\|_{\L(B\oe^+\setminus 
				\overline{\wh P\oe^+})}\right) \\
			\leq C_1\|W\oe^+\|_{\L(B\oe^+)}\rho\e^{{n/2}-1}
			\leq C_2\rho\e^{n/2}\|\nabla W\oe^+\|_{\L(B\oe^+)}  
			\leq
			C_2\rho\e^{n/2}\|\nabla W\oe^+\|_{\L(G\oe^+)}\label{Woe:est2}
		\end{multline}
		(on the penultimate step we use Lemma~\ref{lemma:PFTS}, the inequality \eqref{Friedrichs:De} for 
		$E_\delta=B\oe^+$ and $S_\delta=S\oe^+$).

		To estimate the second term in the right-hand-side of \eqref{byparts:3}, we first observe that
		by standard regularity theory    
		${\partial H^+_D\over\partial x_n}$ is bounded on ${{D}}$, where $H^+_D\ceq H_D\restr_{x_n>0}$.
		Hence
		\begin{gather}\label{Hoe:Linfty}
			\left|{\partial H\oe^+\over\partial x_n}(x)\right|\leq {C  d\e^{-1}},\ x\in D\oe^+.
		\end{gather}
		Furthermore, since $W\oe^+=0$ on $S\oe$, we obtain
		\begin{gather*}
			W\oe^+(x',\eps)=\int_0^\eps {\partial W\oe\over \partial x_n}(x',\tau)\d \tau,\ x'\in d\e \D,
		\end{gather*}
		whence we get
		\begin{align}\label{Woe:L1}
			\int_{D\oe^+} |W\oe^+(x',\eps)|\d x'&\le \left\|{\nabla W\oe^+ }\right\|_{\mathsf{L}^1(T\oe^+)}\leq
			Cd\e^{(n-1)/2}\eps^{1/2}\|\nabla W\oe^+\|_{\L(T\oe^+)}.
		\end{align}
		From  \eqref{Hoe:Linfty} and \eqref{Woe:L1}, we deduce  
		\begin{align}\notag
			\left|\int_{D\oe^+}{\partial H\oe^+\over\partial x_n}W\oe^+\d x'\right|\le
			Cd\e^{(n-3)/2}\eps^{1/2}\|\nabla W\oe^+\|_{\L(T\oe^+)}&=
			Cp\e q\e^{-1/2} \rho\e^{(n-1)/2}\|\nabla W\oe^+\|_{\L(T\oe^+)}
			\\\label{L1Linf2}
			&\le
			C_1 q\e^{-1/2} \rho\e^{(n-1)/2}\|\nabla W\oe^+\|_{\L(G\oe^+)}.
		\end{align}
		Combining \eqref{byparts:3}, \eqref{Woe:est2}, \eqref{L1Linf2}, we obtain
		\begin{gather*}
			\|\nabla W\oe^+\|^2_{\L(G\oe^+)} \leq -2(\nabla  W\oe^+,\nabla V\oe^+)^2_{\L(G\oe^+)}\le C\rho\e^{(n-1)/2}(\rho\e^{ 1/2}+q\e^{-1/2} )   
			\|\nabla W\oe^+\|_{\L(G\oe^+)} ,
		\end{gather*}
		whence
		\begin{gather}\label{nablaW:est:final:1}
			\|\nabla W\oe^+\|_{\L(G\oe^+)} =\rho\e^{(n-1)/2}\mathcal{O}(\rho\e^{ 1/2}+q\e^{-1/2} ),
	\end{gather}
	and, consequently,	
			\begin{gather}\label{nablaW:est:final:2}	
			 |(\nabla  W\oe^+,\nabla V\oe^+)_{\L(G\oe^+)}|=\rho\e^{n-1}\mathcal{O}(\rho\e +q\e  ).
		\end{gather}

		It follows from \eqref{Voe:est+}, \eqref{deltaoe:est+}, \eqref{nablaW:est:final:1}, \eqref{nablaW:est:final:2} that 
		\begin{multline*}
			\|\nabla U\oe^+\|^2_{\L(G\oe^+)} =\|\nabla V\oe^+\|^2_{\L(G\oe^+)}+\|\nabla W\oe^+\|^2_{\L(G\oe^+)}+2(\nabla W\oe^+,\nabla V\oe^+)_{\L(G\oe^+)}\\= {1\over 2}{p \rho\e^{n-1}\mathrm{cap}({{D}}) }
			\left(1+ \mathcal{O}(\rho\e+q\e^{-1}+|p\e-p|) \right),
		\end{multline*}
		whence, using \eqref{Ueo:eq}, we arrive at the
		desired equality \eqref{Cke:asympt3}.
		The lemma is proven.
	\end{proof}
	
	Finally, we formulate the result for $0<q<\infty$.
	We introduce the sets
	\begin{align*}
		T&\ceq \{x=(x',x_n)\in\R^n:\ x'\in \D,\ x_n\in (-pq^{-1},0)\},
		\\
		P&\ceq \mathrm{int}(\overline{\Xi^+\cup T}),
		\\
		S&\ceq \{x=(x',x_n)\in\R^n:\ x'\in \D,\ x_n= -pq^{-1}\}\subset\partial P.
	\end{align*}
	Let $H_P$ be the solution to the following problem:
	\begin{gather*}
		\begin{cases}
			\Delta H_P = 0&\text{in }P,
			\\
			H_P= 1&\text{on }S,\\
			\ds{\partial H_P\over\partial \nu}=0&\text{on }\partial P\setminus S,\\
			H_P\to 0&\text{as }|x|\to \infty.
		\end{cases}
	\end{gather*}
	We set
	\begin{gather*}
		\capty(\D,p,q)\ceq \|\nabla H\|^2_{\L(P)}.
	\end{gather*}
	Note that $\capty(\D,p,q)\to
	\|\nabla H_D\|^2_{\L(\Xi^+)}= \mathrm{cap}(D)/2$ as $q\to\infty$, where
	$H_D$ is the solution to the problem \eqref{HD:BVP:1}--\eqref{HD:BVP:3}. 
	
	\begin{lemma}\label{lemma:Cke3+}
		Let $n\ge 3$ and let $0<q<\infty$. Then
		\begin{gather}\label{Cke:asympt3+}
			\capty\oe={p\over 2}\capty(\D,p,q)\rho\e^{n-1}\left(1+ \mathcal{O}(\rho\e+|   q\e^{-1}-  q^{-1}|+|p\e-p|) \right).
		\end{gather}
	\end{lemma}
	
	The proof of the above lemma  is left to the reader, since its idea is similar to the proof of Lemmas~\ref{lemma:Cke2} and \ref{lemma:Cke3}. 
	This time the approximating function $V\oe^+$ is constructed as follows.
	We define
	$$
	\wt T\oe^+\ceq d\e T+x\oe^+= 
	\left\{x=(x',x_n)\in\R^n:\ x'\in d\e D,\ x_n\in \left(\eps-pq^{-1}d\e,\eps\right)\right\}.
	$$
	Note that $p\e q\e^{-1}=d\e^{-1}\eps $, whence
	$\eps-pq^{-1}d\e=\eps(1- pp\e^{-1}q\e q^{-1})$. Taking this into account, we define the  function $V\oe^+$  as follows: if $p\e q\e^{-1}\ge pq^{-1}$ (consequently, $\wt T\oe^+\subset  T\oe^+$), we set
	\begin{gather*}
		V\oe^+(x)\ceq
		\begin{cases}
			H\oe^+(x)\psi\oe^+(x),& x\in B\oe^+,\\
			H\oe^+(x),& x\in \wt T\oe^+ ,
			\\
			1,&x\in T\oe^+\setminus \wt T\oe^+,
		\end{cases}    
	\end{gather*}
	where $H\oe^+(x)\ceq H_P((x-x\oe^+)d\e^{-1})$, $\psi\oe^+\ceq \phi\left({4|x-x\oe^+|  \rho\e^{-1}}\right)$, 
	$\phi:\R\to\R$ is a   cut-off function satisfying \eqref{phi:prop};
	otherwise, if $p\e q\e^{-1}< pq^{-1}$ (consequently, $\wt T\oe^+\supset  T\oe^+$), we set
	\begin{gather*}
		V\oe^+(x)\ceq
		\begin{cases}
			H\oe^+(x)\psi\oe^+(x),& x\in B\oe^+,\\
			H\oe^+(x),& x\in  T\oe^+ .
		\end{cases}    
	\end{gather*}

	\subsubsection{Calculation of $\mu$}
	
	Now, using   Lemmas~\ref{lemma:Cke2},\,\ref{lemma:Cke3},\,\ref{lemma:Cke3+}, we 
	can prove that condition \eqref{assump:main} holds with a constant  $\mu$. Recall that $p,\,q$ satisfy \eqref{pq:lims+} and that for $q=\infty$ we set $q^{-1}=0$.
	
	\begin{proposition}\label{prop:mu}
		Condition \eqref{assump:main} holds with 
		\begin{gather}\label{mu}
			\mu=
			\begin{cases}
				{\pi p\over 2 +\pi pq^{-1}},&n=2,
				\\[1mm]    
				{p\over 4}\mathrm{cap}({{D}}),&n\ge 3,\ q=\infty,
				\\[2mm]
				{p\over 2}\capty(\D,p,q),&n\ge 3,\ q>0,
			\end{cases}
		\end{gather}
		and 
		\begin{gather*}
			\kappa\e=  
				\mathcal{O}(  \rho\e^{1/2}+|q\e^{-1}-q^{-1}|+|p\e-p| ).
		\end{gather*} 
	\end{proposition}

	\begin{remark} 
		In the case $n\ge 3$ the formula \eqref{mu} for the constant function $\mu$ was 
		 was obtained in \cite{DelV87}. For $n=2$ the result is new.
	\end{remark}

	Before to proof the above proposition, we establish several auxiliary results.
	We introduce the following sets for $k\in\Z^{n-1}$ (recall that $\rho\ke=\rho\e/2$):
	\begin{align*}
		\Gamma\ke&\ceq \left\{x=(x_1,\dots,x_n)\in\R^n:\ |x_j-2\rho\ke k_j|<\rho\ke,\ j=1,\dots,n-1,\ x_n=0\right\},\\
		E\ke^\pm&\ceq \left\{x=(x_1,\dots,x_n)\in\R^n:\ |x_j-2\rho\ke k_j|<\rho\ke,\ j=1,\dots,n-1,\ 0<\pm x_n<\eps+\rho\ke\right\}.
	\end{align*} 
	
	\begin{lemma}\label{lemma:Eke}
		One has:
		\begin{gather}\label{Eke:est:0}
			\forall g\in \H^1(E\ke^\pm):\quad 
			|\la g \ra_{B\ke^\pm} - \la g\ra_{\Gamma\ke }|^2\leq C\rho\e^{1-n} 
			(\rho\e+\eps)\|\nabla g\|^2_{\L(E\ke^\pm)}.
		\end{gather}
	\end{lemma}
	
	\begin{proof}
		We set
		\begin{align*}
			\Gamma\ke^+&\ceq \left\{x=(x_1,\dots,x_n)\in\R^n:\ |x_j-2\rho\ke k_j|<\rho\ke,\ j=1,\dots,n-1,\ x_n= \eps\right\},\\
			Y\ke^+&\ceq \left\{x=(x_1,\dots,x_n)\in\R^n:\ |x_j-2\rho\ke k_j|<\rho\ke,\ j=1,\dots,n-1,\ \eps<x_n<\eps+\rho\ke\right\},\\
			\wt T\ke^+&\ceq \left\{x=(x_1,\dots,x_n)\in\R^n:\ |x_j-2\rho\ke k_j|<\rho\ke,\ j=1,\dots,n-1,\ 0<x_n<\eps\right\},
		\end{align*}
		that is $E\ke^+=\mathrm{int}(\overline{Y\ke^+\cup \wt T\ke^+})$.
		
		One has the following result \cite{Kh09}:	
		let $D\subset\R^n$ be a bounded convex domain, $D_1$, $D_2$ be arbitrary measurable subsets of $D$ with $|D_1|\not=0$, $|D_2|\not=0$, then
		\begin{gather*}
			\forall v\in \H^1(D):\quad \left|\langle v \rangle_{D_1}-\langle v \rangle_{D_2}\right|^2\leq
			C{(\mathrm{diam}(D))^{n+2}\over |D_1|\cdot |D_2|}\|\nabla v\|^2_{\L(D)},
		\end{gather*}
		where  the constant $C>0$ depends only on the dimesnion $n$.	
		Applying this result
		for $v=g\in \H^1(E\ke^+)$,		 $D=Y\ke^+$, $D_1=B\ke^+$, $D_2=Y\ke^+$, we get
		\begin{gather}\label{Eke:est:1}
			|\la g \ra_{B\ke^+} - \la g \ra_{Y^+\ke }|^2\leq C\rho\e^{2-n}\|\nabla g\|^2_{\L(Y\ke^+)}.
		\end{gather}
		By Lemma~\ref{lemma:PFTS} (namely, we use the inequality \eqref{meandiff:est} for  $E_\delta=Y\ke^+$ and 
		$S_\delta=\Gamma^+\ke$), we obtain
		\begin{gather}\label{Eke:est:2}
			|\la g \ra_{Y\ke^+} - \la g \ra_{\Gamma^+\ke }|^2\leq C\rho\e^{2-n} \|\nabla g\|^2_{\L(Y\ke^+)}.
		\end{gather}
		Finally, we have  the estimate
		\begin{gather}\label{Eke:est:3}
			|\la g \ra_{\Gamma\ke^+} - \la g \ra_{\Gamma\ke}|^2\leq C\rho\e^{1-n} \eps\|\nabla g\|^2_{\L(\wt T\ke^+)},
		\end{gather}
		whose proof is similar to the proof in \cite[Lemma~3.2]{CK15}.
		From \eqref{Eke:est:1}--\eqref{Eke:est:3} 
		we conclude the desired estimate \eqref{Eke:est:0} with plus sign.
		The proof of the minus sign part is analogous. The lemma is proven.
	\end{proof}

	Now, we are ready to prove Proposition~\ref{prop:mu}.
	
	\begin{proof}[Proof of Proposition~\ref{prop:mu}]
		Let $g\in \H^{2}(O\setminus \Gamma)$, $\forall h\in \H^{1}(O\setminus \Gamma)$. One has 
		\begin{align}\label{I1233}
			\suml_{k\in\Z^{n-1}}\mathscr{C}\ke\left(\la g \ra_{B\ke^+}-\la g \ra_{B\ke^-}\right)\left(\la \overline{h} \ra_{B\ke^+}-\la \overline{h} \ra_{B\ke^-}\right) 
			-\mu\int_{\Gamma}[   g ]\overline{[   h ]}\d x'  
			= I\e^1 + I\e^2 + I\e^3 ,
		\end{align}  
		where
		\begin{align*}
			I\e^1&\ceq \mathscr{C}\oe\sum_{k\in\Z^{n-1}} \left\{\left(\la g \ra_{B\ke^+}-\la g \ra_{B\ke^-}\right)\left(\la \overline{h} \ra_{B\ke^+}-\la \overline{h} \ra_{B\ke^-}\right)- \la [  g ]\ra_{\Gamma\ke} 
			\overline{\la  [  h]  \ra_{\Gamma\ke}}\right\},
			\\
			I\e^2&\ceq \mathscr{C}\oe\sum_{k\in\Z^{n-1}} \left\{\la [  g ]\ra_{\Gamma\ke} 
			\overline{\la  [  h]  \ra_{\Gamma\ke}} -\rho\e^{1-n} \int_{\Gamma\ke}[   g ]\overline{[   h ]}\d x'  \right\}
			\\
			I\e^3&\ceq (\mathscr{C}\oe\rho\e^{1-n} - \mu) 
			\int_{\Gamma }[   g ]\overline{[   h ]}\d x'  
		\end{align*}
		(here we use the fact that $\Gamma=\cup_{k\in\Z^{n-1}}\overline{\Gamma\ke}$,
		$\Gamma\ke\cap\Gamma_{j,\eps}=\emptyset$ as $k\not=j$).
		
		One has: 
		\begin{multline*}
			\mathscr{C}\oe\sum_{k\in\Z^{n-1}}\left(  \la g \ra_{B\ke^+} \la \overline{h} \ra_{B\ke^+} - \la g_+ \ra_{\Gamma\ke} 
			\overline{\la  h_+  \ra_{\Gamma\ke}}\right)\\=\mathscr{C}\oe\sum_{k\in\Z^{n-1}} \left\{ \left(\la g \ra_{B\ke^+}-\la g_+ \ra_{\Gamma\ke}\right) \la \overline{h} \ra_{B\ke^+} + \la g_+ \ra_{\Gamma\ke} 
			\overline{\left( \la  h   \ra_{B\ke^+} -\la  h_+  \ra_{\Gamma\ke}\right)}\right\},
		\end{multline*}
		whence, using Lemma~\ref{lemma:Eke}, the estimate \eqref{abs:fke}, the Cauchy-Schwarz inequality, and the equality $\capty\oe=\mathcal{O}(\rho\e^{n-1})$ (see Lemmas~\ref{lemma:Cke2},\,\ref{lemma:Cke3},\,\ref{lemma:Cke3+}), we obtain
		\begin{align}\notag
			\left|\mathscr{C}\oe\sum_{k\in\Z^{n-1}} \la g \ra_{B\ke^+} \la \overline{h} \ra_{B\ke^+} - \la g_+ \ra_{\Gamma\ke} 
			\overline{\la  h_+  \ra_{\Gamma\ke}}\right|
			&\leq C\rho\e^{n-1}\left\{\sum_{k\in\Z^{n-1}}\left|\la g \ra_{B\ke^+}   - \la g_+ \ra_{\Gamma\ke} \right|^2
			\sum_{k\in\Z^{n-1}}\left| {\la  h  \ra_{B\ke^+}}\right|^2\right\}^{1/2}
			\\\notag &+ C\rho\e^{n-1}\left\{\sum_{k\in\Z^{n-1}}\left|\la h \ra_{B\ke^+}   - \la h_+ \ra_{\Gamma\ke} \right|^2
			\sum_{k\in\Z^{n-1}}\left| {\la  g_+  \ra_{\Gamma\ke}}\right|^2\right\}^{1/2} 
			\\\notag
			&\leq C(\rho\e+\eps)^{1/2}\|\nabla g\|_{\L(\cup_{k\in\Z^{n-1}}E\ke^+)}\|h\|_{\H^1(\cup_{k\in\Z^{n-1}}R\ke^+)}+
			\\\notag &+C(\rho\e+\eps)^{1/2}\|\nabla h\|_{\L(\cup_{k\in\Z^{n-1}}E\ke^+)}\|g^+\|_{\L(\cup_{k\in\Z^{n-1}}\Gamma\ke^+)}
			\\
			&
			\leq C(\rho\e+\eps)^{1/2}\|g\|_{\H^1(O^+)}\|h\|_{\H^1(O^+)}\label{diff1}
		\end{align}
		where on the last step we use the trace inequality 
		\begin{gather}\label{trace:Gamma}
			\|g^+\|_{\L(\cup_{k\in\Z^{n-1}}\Gamma\ke^+)}=\|g^+\|_{\L(\Gamma)} \leq C\|g\|_{\H^1(O^+)}.
		\end{gather}
		Similarly, we get the estimates
		\begin{align} 
			\left|\mathscr{C}\oe\sum_{k\in\Z^{n-1}} \la g \ra_{B\ke^-} \la \overline{h} \ra_{B\ke^-} - \la g_- \ra_{\Gamma\ke} 
			\overline{\la  h_-  \ra_{\Gamma\ke}}\right|&
			\leq C(\rho\e+\eps)^{1/2}\|g\|_{\H^1(O^-)}\|h\|_{\H^1(O^-)},\label{diff2}
			\\
			\left|\mathscr{C}\oe\sum_{k\in\Z^{n-1}} \la g \ra_{B\ke^+} \la \overline{h} \ra_{B\ke^-} - \la g_+ \ra_{\Gamma\ke} 
			\overline{\la  h_-  \ra_{\Gamma\ke}}\right|&
			\leq C(\rho\e+\eps)^{1/2}\|g\|_{\H^1(O^+)}\|h\|_{\H^1(O^-)},\label{diff3}
			\\
			\left|\mathscr{C}\oe\sum_{k\in\Z^{n-1}} \la g \ra_{B\ke^-} \la \overline{h} \ra_{B\ke^+} - \la g_- \ra_{\Gamma\ke} 
			\overline{\la  h_+  \ra_{\Gamma\ke}}\right|&
			\leq C(\rho\e+\eps)^{1/2}\|g\|_{\H^1(O^-)}\|h\|_{\H^1(O^+)}.\label{diff4}
		\end{align}
		From \eqref{diff1}, \eqref{diff2}--\eqref{diff4}, we conclude the estimate
		\begin{gather}\label{I1:final}
			|I^1\e|\leq C(\rho\e+\eps)^{1/2}\|g\|_{\H^1(O\setminus \Gamma)}\|h\|_{\H^1(O\setminus \Gamma)}.
		\end{gather}
		
		Next, taking into account that $|\Gamma\ke|=\rho\e^{n-1}$, we get
		\begin{multline}
			\mathscr{C}\oe\sum_{k\in\Z^{n-1}} \left\{\la g^+\ra_{\Gamma\ke} 
			\overline{\la  h^+  \ra_{\Gamma\ke}} -\rho\e^{1-n} \int_{\Gamma\ke}g^+\overline{h^+}\d x'  \right\}\\=
			\mathscr{C}\oe\rho\e^{1-n}\sum_{k\in\Z^{n-1}} 
			\int_{\Gamma\ke}(\la g^+\ra_{\Gamma\ke}  - g^+) \overline{\la  h^+  \ra_{\Gamma\ke}}\d x'  
			+
			\mathscr{C}\oe\rho\e^{1-n}\sum_{k\in\Z^{n-1}} 
			\int_{\Gamma\ke}(\overline{\la h^+\ra_{\Gamma\ke}  - h^+}) g^+\d x'.
			\label{diff5}
		\end{multline}
		We denote 
		$$
		\wt Y\ke^+\ceq
		\left\{x\in\R^n:\ |x_j-2\rho\ke k_j|<\rho\ke,\ j=1,\dots,n-1,\ 0<x_n< \rho\ke\right\},\\
		$$
		By Lemma~\ref{lemma:PFTS} (see the estimates \eqref{Poincare:De}, \eqref{traceSigma:De}, \eqref{meandiff:est}), we have for any $v\in \H^1(\wt Y^+\ke)$:
		\begin{align*}
			\|v^+-\la v^+\ra_{\Gamma\ke}\|^2_{\L(\Gamma\ke)}&\leq 
			C \left(\rho\e^{-1}\|v -\la v^+\ra_{\Gamma\ke}\|^2_{\L(\wt Y^+\ke)}+
			\rho\e\|\nabla v \|^2_{\L(\wt Y^+\ke)}\right)\\
			&\leq 
			C_1 \left(\rho\e^{-1}\|v -\la v\ra_{\wt Y\ke^+}\|^2_{\L(\wt Y^+\ke)}+
			\rho\e^{ n-1}|\la v^+\ra_{\Gamma\ke} -\la v\ra_{\wt Y\ke^+} |^2+
			\rho\e\|\nabla v \|^2_{\L(\wt Y^+\ke)}\right)\\
			&\leq 
			C_2\rho\e\|\nabla v^+\|^2_{\L(\wt Y^+\ke)},
		\end{align*}
		(here $v^+$ stands for the trace of $v$ on $\Gamma\ke$).
		Hence, taking into account   \eqref{trace:Gamma} and $\capty\oe=\mathcal{O}(\rho\e^{n-1})$, we deduce from \eqref{diff5}:
		\begin{align}
			\notag &\left|\mathscr{C}\oe\sum_{k\in\Z^{n-1}} \left\{\la g^+\ra_{\Gamma\ke} 
			\overline{\la  h^+  \ra_{\Gamma\ke}} -\rho\e^{1-n} \int_{\Gamma\ke}g^+\overline{h^+}\d x'  \right\}\right|  \\
			\notag &\leq C\left(\sum_{k\in\Z^{n-1}} 
			\|g^+-\la g^+\ra_{\Gamma\ke}\|^2_{\L(\Gamma\ke)}\sum_{k\in\Z^{n-1}} |\la  h^+  \ra_{\Gamma\ke}|^2\rho\ke^{1-n}\right)^{1/2} 
			\\
			\notag & +C\left(\sum_{k\in\Z^{n-1}} 
			\|h^+-\la h^+\ra_{\Gamma\ke}\|^2_{\L(\Gamma\ke)}\sum_{k\in\Z^{n-1}} \|g^+\|^2_{\L(\Gamma\ke)} \right)^{1/2} 
			\\\notag
			&\leq C_1\rho\e^{1/2}\left(\|\nabla g\|_{\L(\cup_{k\in\Z^{n-1}}\wt Y\ke^+)}\|h^+\|_{\L(\cup_{k\in\Z^{n-1}}\Gamma\ke)}+
			\|\nabla h\|_{\L(\cup_{k\in\Z^{n-1}}\wt Y\ke^+)}\|g^+\|_{\L(\cup_{k\in\Z^{n-1}}\Gamma\ke)}\right)
			\\\label{diff1+}
			& \leq C_2\rho\e^{1/2}\|g\|_{\H^1(O^+)}\|h\|_{\H^1(O^+)}.
		\end{align}
		In a similar way, we get
		\begin{align}
			\label{diff2+}
			\left|\mathscr{C}\oe\sum_{k\in\Z^{n-1}} \left\{\la g^-\ra_{\Gamma\ke} 
			\overline{\la  h^- \ra_{\Gamma\ke}} -\rho\e^{1-n} \int_{\Gamma\ke}g^-\overline{h^-}\d x'  \right\}\right|  \leq C\rho\e^{1/2}\|g\|_{\H^1(O^-)}\|h\|_{\H^1(O^-)},
			\\
			\label{diff3+}
			\left|\mathscr{C}\oe\sum_{k\in\Z^{n-1}} \left\{\la g^+\ra_{\Gamma\ke} 
			\overline{\la  h^- \ra_{\Gamma\ke}} -\rho\e^{1-n} \int_{\Gamma\ke}g^+\overline{h^-}\d x'  \right\}\right|  \leq C\rho\e^{1/2}\|g\|_{\H^1(O^+)}\|h\|_{\H^1(O^-)},
			\\
			\label{diff4+}
			\left|\mathscr{C}\oe\sum_{k\in\Z^{n-1}} \left\{\la g^-\ra_{\Gamma\ke} 
			\overline{\la  h^+ \ra_{\Gamma\ke}} -\rho\e^{1-n} \int_{\Gamma\ke}g^-\overline{h^+}\d x'  \right\}\right|  \leq C\rho\e^{1/2}\|g\|_{\H^1(O^-)}\|h\|_{\H^1(O^+)}.
		\end{align}
		From \eqref{diff1+}--\eqref{diff4+}, we conclude the estimate
		\begin{gather}\label{I2:final}
			|I^2\e|\leq C \rho\e^{1/2}\|g\|_{\H^1(O\setminus \Gamma)}\|h\|_{\H^1(O\setminus \Gamma)}.
		\end{gather}
		
		Using the trace estimate $ \|g^\pm\|_{\L(\Gamma)} \leq C\|g\|_{\H^1(O^\pm)}$ and taking into account \eqref{Cke:asympt2}, \eqref{Cke:asympt3}, \eqref{Cke:asympt3+}, we get
		\begin{multline}\label{I3:final}
			|I\e^3| \leq C|\mathscr{C}\oe\rho\e^{1-n} - \mu|
			\|g\|_{\H^1(O\setminus \Gamma)}\|h\|_{\H^1(O\setminus \Gamma)}\\=
			\|g\|_{\H^1(O\setminus \Gamma)}\|h\|_{\H^1(O\setminus \Gamma)}\cdot 
			\begin{cases}
				\mathcal{O}(  \rho\e|\ln\rho\e|+|q\e^{-1}-q^{-1}|+|p\e-p| ),&n=2,
				\\[1mm]
				\mathcal{O}( \rho\e +|  q\e^{-1}-  q^{-1}|+|p\e-p|),&n\ge 3.
			\end{cases}
		\end{multline}
		
		Finally,  from \eqref{pe}, $\lim_{\eps\to 0} q\e^{-1}=q^{-1}<\infty$ (recall that $q>0$) and the equality  
		$$
		\eps=p\e q\e^{-1} d\e\text{ as }n\ge3\text{\quad and\quad }
		\eps=p\e q\e^{-1} d\e|\ln d\e|\text{ as }n=2,
		$$ we conclude that 
		\begin{gather}\label{eps:rhoe}
			\eps\rho\e^{-1}\to 0\text{ as }\eps\to 0.
		\end{gather}
		Combining \eqref{I1233}, \eqref{I1:final}, \eqref{I2:final}--\eqref{eps:rhoe} 
		we arrive at the statement	of the proposition.
	\end{proof}

	\subsection{Periodically distributed curved passages}\label{subsec:7:4}
	
	Lemma~\ref{lemma:Cke3+} and Proposition~\ref{prop:mu} suggest how to construct an example of a sieve with non-straight passages $T\ke$ for which the quantities $\capty\ke$ can  be computed explicitly and the assumption \eqref{assump:main} can be verified. 
	
	Let $n\ge 3$. We have the following sets  (cf. \eqref{Oe}--\eqref{Xie}):
	\begin{align*}
		O_1 = \{x \in\R^n:\ |x_n|<1\},\quad
		\Gamma^\pm_1 = \{x \in\R^n:\ x_n=\pm 1\},\quad 
		\Xi^\pm_1 =\{x  \in\R^n:\ \pm x_n>1\}.
	\end{align*}
	Let $T\subset\R^n$ be a domain satisfying 
	$T\subset O_1$ and such that
	$D^\pm\ceq \partial T\cap \Gamma^\pm_1$ are non-empty  relatively open in  $\Gamma^\pm_1$ connected  sets.
	We also assume that the smallest balls containing $ D^\pm$ 
	are centered at $(\underset{n-1}{\underbrace{0,\dotsm,0}},\pm 1)$.
	We define  the passages $T\ke$, $k\in\Z^{n-1}$ by downscaling $T$ and then copying-and-pasting it  periodically, with period $\rho\e\ceq \eps^{(n-2)/(n-1)}$ along $\Gamma$, i.e.
	$$T\ke\ceq  \eps T + (\rho\e k,0) .$$
	One has $x\ke^\pm=(\rho\e k,\pm\eps)$ and 
	$d\ke^\pm= d^\pm\eps$, where   $d^\pm$ are the radii of the smallest balls containing $D^\pm$. Evidently, the assumptions \eqref{assump:0}--\eqref{assump:3} hold  with $\rho\ke=\rho\e/ 2$.
	
	Let $P\ceq \mathrm{int}(\overline{\Xi^+_1\cup \Xi^-_1 \cup T})$,
	and  let $\wt H_P$ be the solution to the  problem
	\begin{gather*}
		\begin{cases}
			\Delta\wt H_P = 0&\text{in }P,
			\\ 
			\ds{\partial \wt H_P\over\partial \nu}=0&\text{on }\partial P ,
			\\[2mm]
			\wt H_P\to 1&\text{as }x\in\Xi_1^+\text{ and }|x|\to \infty.
			\\ 
			\wt H_P\to 0&\text{as }x\in\Xi_1^-\text{ and }|x|\to \infty.
		\end{cases}
	\end{gather*}
	Then one has the following equality, which can be proven in the same way as    
	Lemma~\ref{lemma:Cke3+}:
	$$\capty\ke=\rho\e^{n-1}\|\nabla \wt H_P\|^2_{\L(P)}(1+{ \mathcal{O}(\rho\e)}).$$
	Furthermore, the assumption \eqref{assump:main} is fulfilled with
	$$\mu=\|\nabla \wt H_P\|^2_{\L(P)}\text{ and }\kappa\e= \mathcal{O}(\rho\e^{ 1/2})$$
	(the proof is similar to the proof of Proposition~\ref{prop:mu}).

	\subsection{Non-periodically distributed straight passages}
	\label{subsec:7:3}
	
	In this subsection we present the example of a sieve with 
	straight passages which, in contrast to the previous subsection, are distributed
	non-periodically --- as a result, the  corresponding function $\mu$ in \eqref{assump:main}
	turns to be non-constant. 
	To simplify the presentation, we assume here that $n\ge 3$.
	\smallskip 
	
	Let  $\left\{\Gamma\ke,\ k\in\Z^{n-1}\right\}$ be a family of relatively open  subsets of $\Gamma$ such that  
	\begin{align}
		&\label{Gammake:cond5}
		\lim_{\eps\to 0}\sup_{k\in\Z^{n-1}}\diam(\Gamma\ke)=0,
		\\
		\label{Gammake:cond1}
		&\Gamma_{k,\varepsilon}\cap \Gamma_{l,\varepsilon}=\emptyset\text{ if }k\not=l,
		\\
		\label{Gammake:cond2}
		&
		\cup_{k\in\Z^{n-1}}\overline{\Gamma\ke}=\Gamma,
		\\
		&\label{Gammake:cond4}
		\exists m\in\N\ \forall x\in\Gamma:\  \#(\{k\in\Z^{n-1}:\ x\in\conv(\Gamma\ke)\})\leq m,
		\\
		&\label{Gammake:cond6}
		\lim_{\eps\to 0}\left(\eps^{-(n-2)/(n-1)}\inf_{k\in\Z^{n-1}}\diam(\Gamma\ke)\right)=\infty,
	\end{align}
	where  by $\conv(\Gamma\ke)$ we denote   the convex hull of $\Gamma\ke$, and $\#(\dots)$ stands for the cardinality of the enclosed set (that is, if all  $\Gamma\ke$ are convex, then \eqref{Gammake:cond4} holds with $m=1$).
	Furthermore, we assume that there exist a sequence the sequence of points $\left\{x\ke=(x'\ke,0)\in\Gamma\ke,\ k\in\Z^{n-1}\right\}$   such that  
	\begin{align}
		\label{r:cond}  \diam(\Gamma\ke)\leq C\dist(x\ke,\partial\Gamma\ke),
	\end{align}
	where the constant $C$ is independent of $k$  and $\eps$.
	Since $\dist(x\ke,\partial\Gamma\ke)\le \diam(\Gamma\ke)$,
	we deduce from \eqref{Gammake:cond5}:
	\begin{gather}
		\sup_{k\in\Z^{n-1}}\dist(x\ke,\partial\Gamma\ke)\to 0\text{ as }\eps\to 0.
	\end{gather}

	Let   
	the passages $T\ke$, $k\in\Z^{n-1}$ be given by
	$$
	T\ke\ceq\left\{x=(x',x_n)\in\R^n:\ x' -x'\ke \in  d\ke \D,\ |x_n| < \eps\right\}, 
	$$
	where  $\D\subset\R^{n-1}$ is a bounded Lipschitz domain such that  
	the smallest ball containing $\D$ is centered at $0$; without loss of generality,       we assume that the radius of this ball is equal to $1$. 
	We choose the numbers $d\ke$, $k\in\Z^{n-1}$ as follows,
	\begin{gather}
		\label{dke:Gammake}
		d\ke=\left(4w(x\ke)|\Gamma\ke|(\mathrm{cap}({{D}}))^{-1}\right)^{1/(n-2)},
	\end{gather}
	where $w\in C^1(\Gamma)\cap \W^{1,\infty}(\Gamma)$ is a non-negative function, 
	$D$ is a compact set defined by \eqref{DD}, $\mathrm{cap}({{D}})$ is its Newton capacity.
	One has $d\ke^\pm= d\ke$ and $x\ke^\pm =(x\ke',\pm\eps)$.
	It follows easily from \eqref{Gammake:cond1}, \eqref{r:cond}--\eqref{dke:Gammake}
	that assumptions \eqref{assump:0}--\eqref{assump:3} hold with 
	$$\rho\ke\ceq\dist(x\ke,\partial\Gamma\ke).$$
	Furthermore, \eqref{non:con} holds with $\zeta\e=C\eps^{1/2}$ (see Subsection~\ref{subsec:6:1}).

	The partition of $\Gamma$ via the subsets $\Gamma\ke$
	is illustrated on Figure~\ref{fig2}{(a)}.
	The set $\Gamma\ke$ and the relations between the numbers $d\ke$, $\diam(\Gamma\ke)$ and $\dist(x\ke,\partial\Gamma\ke)$ are shown on
	Figure~\ref{fig2}{(b)}.

	\begin{figure}[h]
		\begin{picture}(400,150) \centering
			\scalebox{1}{
				\includegraphics{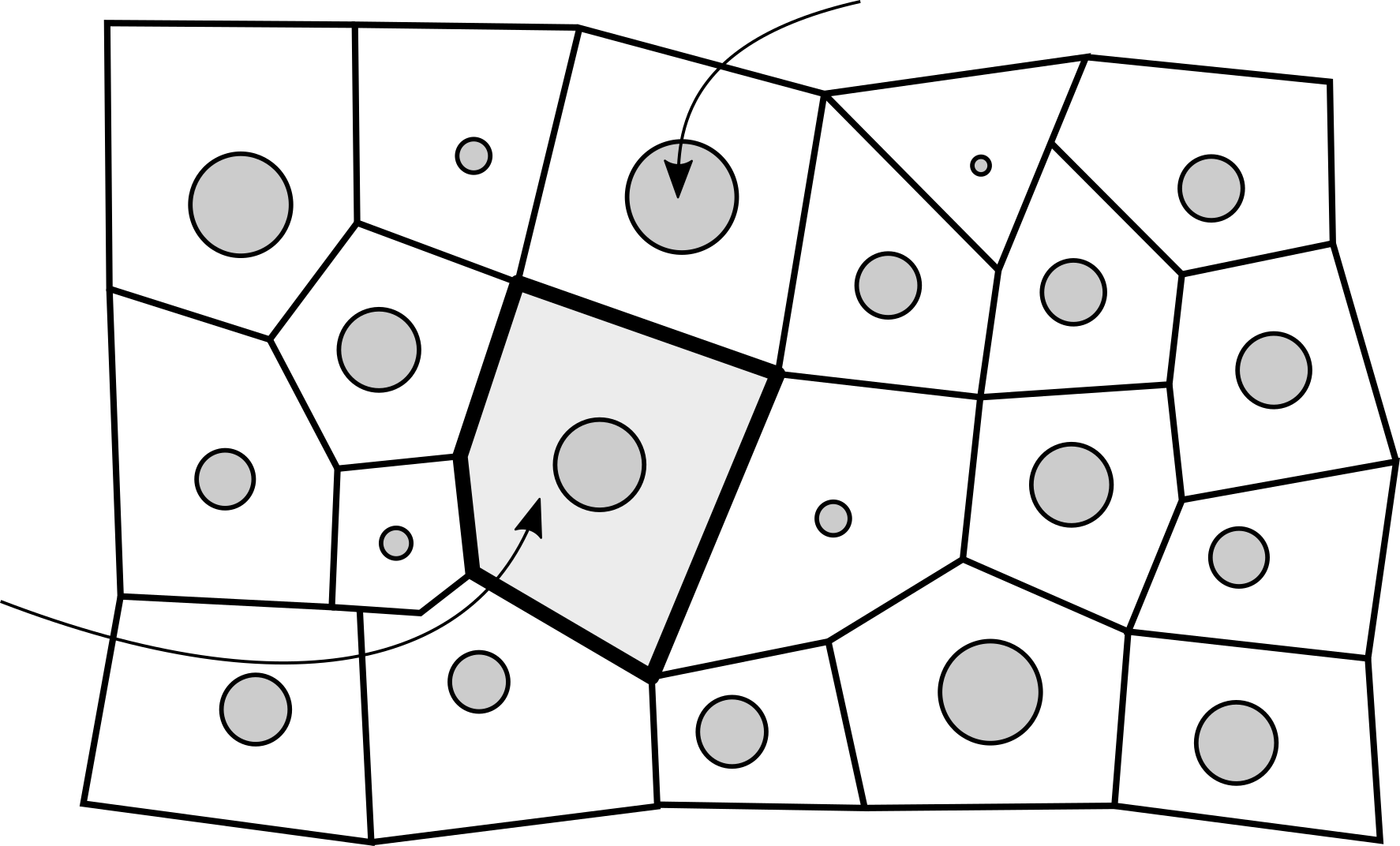}\qquad\qquad\qquad
				\includegraphics{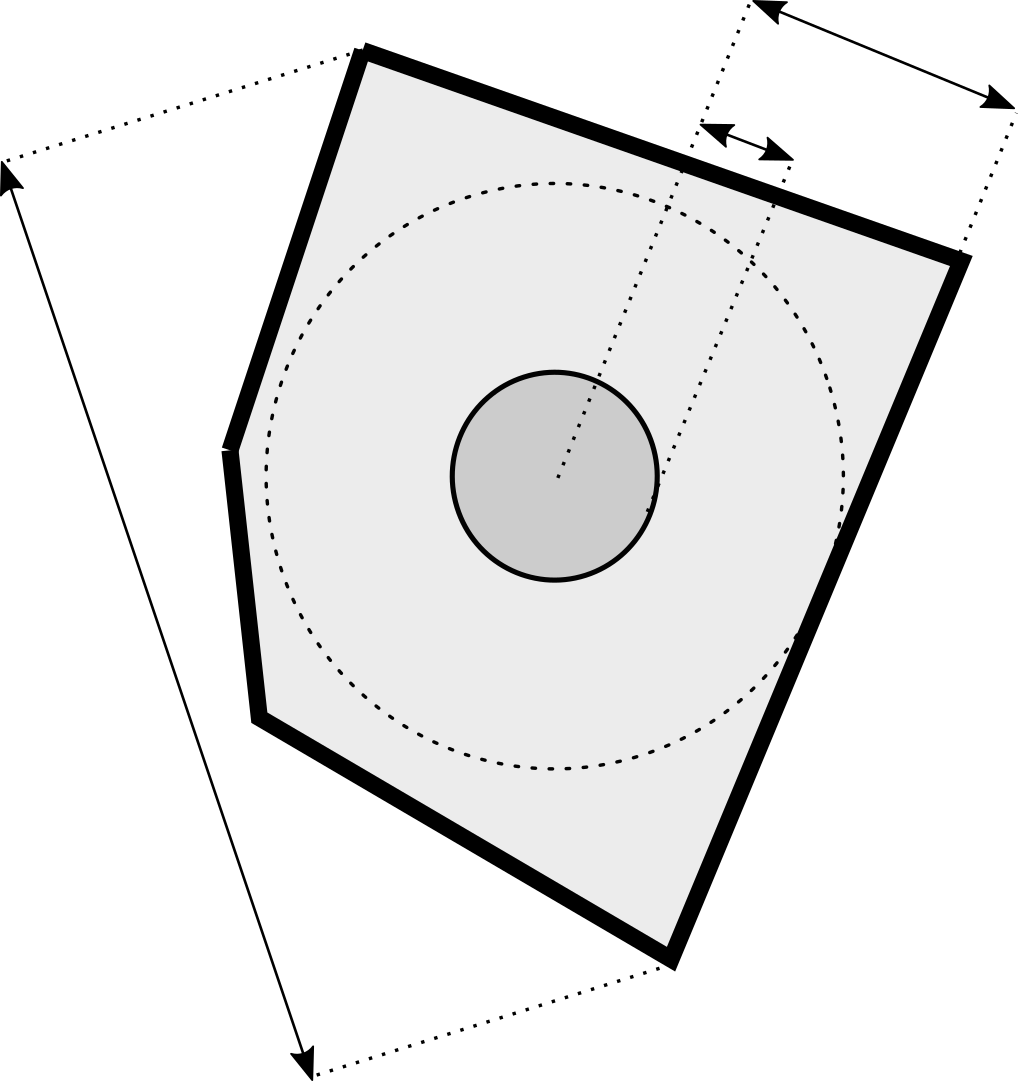}
				
				\put(-420,140){\text{(a)}}
				\put(-130,140){\text{(b)}}

				\put(-419,38){$\Gamma\ke$}
				\put(-269,127){$\mathscr{B}(d\ke,x\ke)\cap\Gamma$}
				
				\put(-59,60){$^{x\ke}$}
				\put(-55.3,71){\circle*{3}}
				
				\put(-120,62){\rotatebox{-70}{$^{\diam(\Gamma\ke)}$}}
				\put(-44,142){\rotatebox{-20}{$_{\rho\ke\ceq \dist(x\ke,\partial\Gamma\ke)}$}}
				\put(-37,115){\rotatebox{-20}{$^{d\ke}$}}}
			
		\end{picture}
		\caption{ The partition of $\Gamma$ via the subsets $\Gamma\ke$. 
		}\label{fig2}
	\end{figure}

	We denote $q\ke={d\ke^{n-1}\eps^{-1}\rho\ke^{1-n}}$. It follows easily from \eqref{Gammake:cond5}, \eqref{Gammake:cond6},
	\eqref{r:cond}, \eqref{dke:Gammake}
	that 
	\begin{gather*}
		\sup_{k\in\Z^{n-1}}q\ke^{-1}\to 0\text{ as }\eps\to 0.	
	\end{gather*}
	Repeating   verbatim the proof of Lemma~\ref{lemma:Cke3}, we get
	\begin{gather*}
		\capty\ke={1\over 4}\mathrm{cap}({{D}})d\ke^{n-2}\left(1+ \Delta\ke \right)=
		w(x\ke)|\Gamma\ke|\left(1+ \Delta\ke \right)
	\end{gather*}	
	with the remainder $\Delta\ke$ satisfying
	$$
	\Delta\ke\leq C(\rho\ke+q\ke^{-1}).
	$$

	\begin{proposition}\label{prop:nonperiod}
		Let the assumptions  \eqref{Gammake:cond5}--\eqref{Gammake:cond6}, \eqref{dke:Gammake} hold true. Then the condition \eqref{assump:main} is fulfilled with $\mu(x)=w(x)$ and $\kappa\e=C\left\{ \rho\e^{1/2},\, \sup_{k\in\Z^{n-1}}q\ke^{-1/2}\right\}$.
	\end{proposition}
	
	The proof of the above proposition is similar to the proof of Proposition~\ref{prop:mu}. The only essential difference is that the sets $\Gamma\ke$ and $E\ke^\pm$ are not necessary convex, whence
	in the right-hand-sides of some estimates one has to replace these sets by their convex hulls. For example, the crucial estimate \eqref{Eke:est:0}
	now reads
	\begin{gather*}
		\forall g\in \H^1(\conv(E\ke^\pm)):\quad 
		|\la g \ra_{B\ke^\pm} - \la g \ra_{\Gamma\ke }|^2\leq C\rho\ke^{1-n}(\rho\ke+\eps)\|\nabla g\|^2_{\L(\conv(E\ke^\pm))}.
	\end{gather*}
	However this replacement creates no extra difficulties due to the assumption \eqref{Gammake:cond4}
	(evidently, \eqref{Gammake:cond4} also holds with $E\ke^\pm$ instead of $\Gamma\ke$).
	
	\bigskip

	\section*{Acknowledgements}
	A.K. is   supported by the  Czech Science Foundation (GA\v{C}R) within the project 22-18739S ``Asymptotic and spectral analysis of operators in mathematical physics''.
	
	\section*{Data availability}
	No data was used for the research described in the article.

	\bibliographystyle{plainurl}

\end{document}